\documentclass[11pt]{amsart}
\usepackage[latin1]{inputenc}
\usepackage{amsmath,amsthm,amsfonts,amssymb,amscd,mathtools}
\usepackage{verbatim}
\usepackage{mathrsfs}
\usepackage{multicol}
\usepackage{tabularx}
\usepackage[usenames,dvipsnames]{color}
\usepackage{enumerate}
\usepackage{graphicx}
\usepackage{color,xcolor}
\usepackage{bbm}
\usepackage[all]{xy}
\usepackage{hyperref}

\theoremstyle{definition}
\newtheorem{thm}{Theorem}[subsection]
\newtheorem{prop}[thm]{Proposition}
\newtheorem{cor}[thm]{Corollary}

\newtheorem{lem}[thm]{Lemma}

\newtheorem{ex}[thm]{Example}
\newtheorem{rem}[thm]{Remark}

\numberwithin{equation}{subsection}

\def\cal#1{\text{$\mathcal{#1}$}}
\def\ord#1^#2{#1$^{\text{#2}}$}
\def\lie#1{\mathfrak{#1}}
\def\tlie#1{\tilde{\mathfrak{#1}}}
\def\hlie#1{\hat{\mathfrak{#1}}}

\def\uqr#1^#2{\text{$U_q^{#2}(\lie #1)$}}

\def\uqhr#1^#2{\text{$U_q^{#2}(\hlie #1)$}}
\def\us#1^#2{\text{$U_{\xi}^{#2}(\lie #1)$}}
\def\ush#1^#2{\text{$U_{\xi}^{#2}(\hlie #1)$}}
\def\dus#1^#2{\text{$\dot{U}_{\xi}^{#2}(\lie #1)$}}
\def\dush#1^#2{\text{$\dot{U}_{\xi}^{#2}(\hlie #1)$}}
\def\gb#1{{\mbox{\boldmath $#1$}}}

\def\wtl{{\rm wt}_\ell}
\def\wt{{\rm wt}}

\def\soc{{\rm soc}}
\def\head{{\rm head}}
\def\top{{\rm Top}}
\def\supp{{\rm supp}}

\def\ch{{\rm ch}}
\def\qch{{\rm qch}}

\def\opl_#1^#2{\text{\scriptsize$\bigoplus\limits_{\text{\normalsize$#1$}}^{\text{\normalsize$#2$}}$}}
\def\otm_#1^#2{\text{\scriptsize$\bigotimes\limits_{\text{\footnotesize$#1$}}^{\text{\footnotesize$#2$}}$}}
\def\wcal#1{{\mbox{$\widetilde{\cal #1}$}}}

\def\tqbinom#1#2{\text{$\left[\begin{smallmatrix} #1\\#2\end{smallmatrix}\right]$}}
\def\bs#1{\boldsymbol{#1}}
\def\endd{\hfill$\diamond$}

\def\ffbox#1{\setbox9=\hbox{$\scriptstyle\overline{1}$}\framebox[0.55cm][c]{\rule{0mm}{\ht9}${\scriptstyle #1}$}}

\newcommand{\g}{\mathfrak{g}}

\newcommand{\C}{\mathbb{C}}

\newcommand{\Z}{\mathbb{Z}}

\def\stab{{\rm STab}}

\begin{document}
\title[On Tensor Products of Minimal Affinization with KR Modules]{On Tensor Products of a Minimal Affinization with an  Extreme Kirillov-Reshetikhin Module for type $A$}
\author[Adriano Moura and Fernanda Pereira]{Adriano Moura and Fernanda Pereira}

\address{Departamento de Matemática, Universidade Estadual de Campinas, Campinas - SP - Brazil, 13083-859.}
\email{aamoura@ime.unicamp.br}

\address{Departamento de Matemática, Divisão de Ciências Fundamentais, Instituto Tecnológico de Aeronáutica, São José dos Campos - SP - Brazil, 12228-900.}
\email{fpereira@ita.br}
\thanks{Part of this work was developed while the second author was a visiting Ph.D. student at Institut de Mathématiques de Jussieu working under the supervision of David Hernandez. She thanks David Hernandez for his guidance during that period and Université Paris 7 for the hospitality. She also thanks Fapesp (grant 2009/16309-5) for the financial support. The work of the first author was partially supported by CNPq grant 304477/2014-1 and Fapesp grant 2014/09310-5. We thank the two anonymous referees for their detailed reports and specially thank one of them for sending us suggestions that significantly simplified the proofs of Corollary \ref{c:main}, the second statement of Proposition \ref{p:multiplicity1A}, and Proposition \ref{prop 1,..,n-1-dom}.}

\begin{abstract}
For a quantum affine algebra of type $A$, we describe the composition series of the tensor product of a general minimal affinization with a Kirillov-Resehtikhin module associated to an extreme node of the Dynkin diagram of the underlying simple Lie algebra.
\end{abstract}

\maketitle

\section{Introduction}

Unraveling the intricate structure of the category of finite-dimensional representations of quantum affine algebras has drawn the attention of experts in representation theoretic Lie theory since the early 1990s. One of the interesting problems to be addressed is that of understanding the class of irreducible affinizations of a given simple module for the quantum group $U_q(\lie g)$ associated to the underlying finite-dimensional simple Lie algebra $\lie g$. In particular, one wants to classify and describe the structure of the minimal such affinizations in the sense defined by Chari in \cite{cha:minr2}. The early work of Chari and Pressley in this direction \cite{cp:minsl,cp:minnsl} describes the classification of the Drinfeld polynomials corresponding to such minimal affinizations in the case that $\lie g$ is not of types $D$ or $E$. It turns out that, in those cases, there exists essentially one minimal affinization for any given simple module of $U_q(\lie g)$, i.e., all minimal affinizations of that module are isomorphic as modules for $U_q(\lie g)$ (affinizations which are isomorphic as modules for $U_q(\lie g)$ are said to be equivalent). These papers also classify the minimal affinizations of the so called regular representations in types $D$ and $E$. Namely, those for which the support of corresponding highest weight either does not bound a subdiagram of type $D$ or, if it does, then the trivalent node belongs to the support. The number of equivalence classes of minimal affinzations in types $D$ and $E$ depends on the highest weight. If its support does not bound a subdiagram of type $D$, there is only one class as before, but if it does, then there are typically 3 equivalent classes (essentially coming from the symmetry of the subdiagram of type $D_4$, even if the support is not symmetric).

Thus, as long as classification goes, it remains to study the irregular minimal affinizations in types $D$ and $E$.
This paper is the first part of our work towards the classification of minimal affinizations in type $D$. It contains results in the case that $\lie g$ is of type $A$ which are crucial to obtain such classification. Namely, by looking at the 3 connected components of the diagram of type $D$ after removing the trivalent node, we have a diagram of type $A_n, n\ge 1$, and two of type $A_1$. The minimal affinizations are realizable as a simple factor of the tensor product of simple modules supported in each of these connected components. The strategy is to compare ``all'' possible such tensor products to pinpoint which ones give rise to minimal affinizations. The tensor product of only two of the factors can be regarded as a module for a diagram-subalgebra of type $A$. This partially  explains our interest in tensor products of a general minimal affinization with a Kirillov-Reshetikhin module associated to an extreme node of the Dynkin diagram for type $A$ (a Kirillov-Reshetikhin module is a minimal affinization of a simple module whose highest weight is a multiple of a fundamental weight). Since these results are interesting in their own right and have strong potential to be useful for studying other aspects of the category of finite-dimensional representations of quantum affine algebras (in all types) and the proofs are quite long, we present them here by themselves. The classification of minimal affinizations for type $D$ will appear in a forthcoming publication (see also \cite{fer:thesis}).

It is important to remark that tensor products of irreducible representations of quantum affine algebras in general, and of minimal affinizations in particular, is a relevant topic not only to the understanding of the underlying category of finite-dimensional representations, but it also has very important applications or deep connections to other areas such as integrable systems in mathematical systems, combinatorics, and cluster algebras. The most studied case is that of tensor products of Kirillov-Reshetikhin modules. Such tensor products give rise to a remarkable family of short exact sequences which can be encoded in a set of recurrence relations, called $T$-systems, which have many applications in integrable system. The literature in this direction is vast and we refer the reader to \cite{her:KRconj,muyou:tsystem} for more details and references. On the other hand, the connection to cluster algebras was first discovered in \cite{hele:cluster}. More recently, it has been shown in \cite{hele:KRcluster} that the $T$-systems can be interpreted as cluster transformations in a cluster algebra having KR-modules correspond to an initial seed. This then lead to an algorithm for computing the qcharacters of KR-modules by successive approximations via the combinatorics of cluster algebras.  It would then be interesting to eventually study the results of the present paper from the perspective of $T$-systems and cluster algebras. The connection of graded limits of tensor products of finite-dimensional representations of quantum affine algebras with the notion of fusion products in the sense of \cite{fl:kosver} is also another topic of recent interest (see for instance \cite{bcm,bp,naoi:fus,naoi:krfus} and references therein).
Thus, it should also be interesting to study the graded limits of the tensor products studied here in that context as well.

We now describe the organization and the main results of this paper. In Section \ref{s:pre}, we fix the notation related to the study of finite-dimensional representations of quantum affine algebras, review the basic facts about such representations as well as the relevant known facts about minimal affinizations, and state  the main result of the paper (the combination of Theorem \ref{t:main} with Corollary \ref{c:main}). The statement can be informally described as follows. The tensor product  of an ``increasing'' minimal affinization with a KR-module supported at the last node of the Dynkin diagram is an indecomposable module of length at most $2$. We describe precisely the conditions on the Drinfeld polynomials of the tensor factors which give rise to a length-$2$ tensor product and write down an explicit formula for the Drinfeld polynomial of the ``extra'' irreducible factor. Moreover, we precisely describe the socle and the head for both orders of the tensor factors.  The other possibilities of tensor products (replacing the minimal affinization by a ``decreasing'' one or the KR-module by one supported at the first node) can be obtained from the case established in Theorem \ref{t:main} by certain duality arguments and the  precise explanation and statements are given in Sections \ref{ss:dual} and \ref{ss:Gen}. In Section \ref{s:char}, we review results about the main tool we shall use in the proof of Theorem \ref{t:main}: the theory of qcharacters. In particular, and very importantly, in Section \ref{ss:qchminA}, we review the description of the qcharacter of a minimal affinization in type $A$ in terms of semi-standard tableaux. The core of our proof is based on combinatorial analysis of ``products'' of such tableaux. The proof of Theorem \ref{t:main} is given in Section \ref{s:proof}. After explaining the general scheme of the proof in Section \ref{ss:scheme}, we proceed by describing the dominant $\ell$-weights of the tensor product in Sections \ref{ss:Jdom} and \ref{ss:domtpA}. It turns out that the set of such dominant $\ell$-weights is totally ordered and the corresponding $\ell$-weight spaces are one-dimensional (this is the statement of Proposition \ref{p:multiplicity1A}, which is also an interesting result by itself and can be considered the second most important result of the paper). In the last step of the proof, performed in Section \ref{ss:sftpA}, we start by explicitly describing which of these dominant $\ell$-weights are $\ell$-weights of the obvious irreducible factor of the tensor product (the one whose Drinfeld polynomial is the product of those of the two tensor factors). Under certain conditions on the Drinfeld polynomials of the tensor factors (the conditions in the statement of Theorem \ref{t:main}), we see that not all the dominant $\ell$-weights are $\ell$-weights of the obvious irreducible factor. Hence, the highest of the remaining ones must be the Drinfeld polynomial of an extra irreducible factor and we show that all remaining dominant $\ell$-weights  are $\ell$-weights of this extra irreducible factor. The proof of Corollary \ref{c:main}, about the dependence of the socle and the head on the order of the tensor factors, is given in Section \ref{ss:sh}.

In principle, the methods we employed here could be used to obtain similar information about the tensor product of any two minimal affinizations.
However, the combinatorics would be substantially more complicated and it is unclear if it would be manageable to obtain results as precise as we did. Also, it is unlikely that Proposition \ref{p:multiplicity1A} remains valid and multiplicity issues could turn the arguments we employed here insufficient. In light of our remarks about $T$-system and cluster algebras above, our main theorem may be regarded as a step towards studying short exact sequences related to tensor products of minimal affinizations beyond KR-modules (see also \cite{muyou:tsystem}) and, hopefully, the machinery of cluster algebras may eventually provide more tools to expand the scope of the study initiated here.

\section{Basic Notation and the Main Theorem}\label{s:pre}

Throughout the paper, let $\mathbb C, \mathbb R,\mathbb Z,\mathbb
Z_{\ge m}$ denote the sets of complex numbers, reals, integers,  and
integers bigger or equal $m$, respectively. Given a ring $\mathbb
A$, the underlying multiplicative group of units is denoted by
$\mathbb A^\times$. The dual of a vector space $V$ is denoted by
$V^*$. The symbol $\cong$ means ``isomorphic to''.

\subsection{The Algebras}\label{ss:clalg}

Let $\lie g=\lie{sl}_{n+1}(\mathbb C)$, $I=\{1,\dots,n\}$, and $\lie h$ the standard Cartan subalgebra, i.e., $\lie h$ is the span of $h_1,\dots, h_n$ with $h_i=e_{i,i}-e_{i+1,i+1}, i\in I$, where $e_{i,j}$ is the matrix whose $(i,j)$ entry is $\delta_{ij}$. Fix the set of positive roots $R^+$ so that positive root vectors are upper triangular matrices and let
$$\lie n^\pm = \bigoplus_{\alpha\in R^+}^{} \lie g_{\pm\alpha}
\quad\text{where}\quad \lie g_{\pm\alpha} = \{x\in\lie g:
[h,x]=\pm\alpha(h)x, \ \forall \ h\in\lie h\}.$$ The simple roots
will be denoted by $\alpha_i$ and the fundamental weights by
$\omega_i$, $i\in I$. $Q,P,Q^+,P^+$ will denote the root and weight
lattices with corresponding positive cones, respectively.
Equip $\lie h^*$ with the partial order $\lambda\le \mu$ iff
$\mu-\lambda\in Q^+$. Let $C = (c_{ij})_{i,j\in I}$ be the Cartan
matrix of $\lie g$, i.e., $c_{ij}=\alpha_j(h_i)$. The Weyl group is
denoted by $\cal W$.

If  $\lie a$ is a Lie algebra over $\mathbb C$, define its loop
algebra to be $\tlie a=\lie a\otimes  \mathbb
C[t,t^{-1}]$ with bracket given by $[x \otimes t^r,y \otimes
t^s]=[x,y] \otimes t^{r+s}$. Clearly $\lie a\otimes 1$ is a
subalgebra of $\tlie a$ isomorphic to $\lie a$ and, by abuse of
notation, we will continue denoting its elements by $x$ instead of
$x\otimes 1$. Then $\tlie g = \tlie n^-\oplus \tlie h\oplus \tlie
n^+$ and $\tlie h$ is an abelian subalgebra.

Fix $q\in \C^\times$ which is not a root of unity and set
\begin{equation*}
[m]=\frac{q^m -q^{-m}}{q -q^{-1}},\ \ \ \ [m]! =[m][m-1]\ldots
[2][1],\ \ \ \ \tqbinom{m}{r} = \frac{[m]!}{[r]![m-r]!},
\end{equation*}
for $r,m\in\mathbb Z_{\ge 0}$, $m\ge r$.

The quantum loop algebra  $U_q(\tlie g)$ is  the
associative $\mathbb C$-algebra with generators $x_{i,r}^{{}\pm{}}$
($i\in I$, $r\in\Z$), $k_i^{{}\pm 1}$ ($i\in I$), $h_{i,r}$ ($i\in
I$, $r\in \Z\backslash\{0\}$) and the following defining relations:

\begin{gather*}
k_ik_i^{-1} = k_i^{-1}k_i=1, \ \  k_ik_j =k_jk_i, \ \ k_ih_{j,r} =h_{j,r}k_i,\ \ [h_{i,r},h_{j,s}]=0,\\
k_ix_{j,r}^\pm k_i^{-1} = q^{{}\pm c_{ij}}x_{j,r}^{{}\pm{}},\ \ [h_{i,r} , x_{j,s}^{{}\pm{}}] = \pm\frac1r[rc_{ij}]x_{j,r+s}^{{}\pm{}},
\end{gather*}
\begin{gather*}
x_{i,r+1}^{{}\pm{}}x_{j,s}^{{}\pm{}} -q^{{}\pm c_{ij}}x_{j,s}^{{}\pm{}}x_{i,r+1}^{{}\pm{}} =
q^{{}\pm     c_{ij}}x_{i,r}^{{}\pm{}}x_{j,s+1}^{{}\pm{}} -x_{j,s+1}^{{}\pm{}}x_{i,r}^{{}\pm{}},\\ 
[x_{i,r}^\pm , x_{j,s}^\pm] = 0,\ \ \text{if $c_{ij}=0$}, \text{ whlie,}\ \ \text{if $c_{ij}=-1$},\\
x_{i,r_1}^\pm x_{i,r_2}^\pm x_{j,s}^\pm + x_{i,r_2}^\pm x_{i,r_1}^\pm x_{j,s}^\pm + x_{j,s}^\pm x_{i,r_1}^\pm x_{i,r_2}^\pm + x_{j,s}^\pm x_{i,r_2}^\pm x_{i,r_1}^\pm = [2]\left(x_{i,r_1}^\pm x_{j,s}^\pm x_{i,r_2}^\pm + x_{i,r_2}^\pm x_{j,s}^\pm x_{i,r_1}^\pm\right),\\
[x_{i,r}^+ , x_{j,s}^-]=\delta_{i,j}  \frac{ \psi_{i,r+s}^+ - \psi_{i,r+s}^-}{q - q^{-1}},
\end{gather*}
where the $\psi_{i,r}^{\pm}$ are determined by equating powers of $u$ in the formal power series
$$\Psi_i^\pm(u) = \sum_{r\in\mathbb Z}\psi_{i,\pm r}^{\pm}u^r = k_i^{\pm 1} \exp\left(\pm(q-q^{-1})\sum_{s=1}^{\infty}h_{i,\pm s}
u^s\right).$$
In particular, $\psi_{i,\pm r}^\pm=0$ if $r<0$.

Denote by $U_q(\tlie n^\pm), U_q(\tlie h)$  the subalgebras of
$U_q(\tlie g)$ generated by $\{x_{i,r}^\pm\}, \{k_i^{\pm1},
h_{i,s}\}$, respectively. Let  $U_q(\lie g)$  be  the subalgebra
generated by $x_i^\pm:=x_{i,0}^\pm, k_i^{\pm 1}, i\in I,$ and define
$U_q(\lie n^\pm), U_q(\lie h)$ in the obvious way.  $U_q(\lie g)$ is
a subalgebra of $U_q(\tlie g)$ and multiplication establishes
isomorphisms of vectors spaces:
$$U_q(\lie g) \cong U_q(\lie n^-) \otimes U_q(\lie h) \otimes U_q(\lie n^+)
\qquad\text{and}\qquad U_q(\tlie g) \cong U_q(\tlie n^-) \otimes
U_q(\tlie h) \otimes U_q(\tlie n^+).$$

For $i\in I, r\in \mathbb Z, k\in\mathbb Z_{\ge 0}$, define
$(x_{i,r}^\pm)^{(k)} = \frac{(x_{i,r}^\pm)^k}{[k]!}$. Define also
elements $\Lambda_{i,r}, i\in I, r\in\mathbb Z$ by equating powers of $u$ in
the formal power series
\begin{equation}\label{e:Lambdad}
\Lambda_i^\pm(u)=\sum_{r=0}^\infty \Lambda_{i,\pm r} u^{r}=
\exp\left(-\sum_{s=1}^\infty\frac{h_{i,\pm s}}{[s]}u^s\right).
\end{equation}
The elements $\Lambda_{i,\pm r}$ together with $k_i^{\pm1}$, $i\in
I,$ $r\in \Z$, generate $U_q(\tlie h)$ as an algebra.

 The following assignments,
$$\Delta(x_i^+) = x_i^+ \otimes 1 + k_i \otimes x_i^+,\quad
\Delta(x_i^-) = x_i^-\otimes k_i^{-1} + 1 \otimes x_i^-,  \quad
\Delta(k_i^{\pm 1}) = k_i^{\pm 1}\otimes k_i^{\pm 1},$$
$$S(x_i^+) = - k_i^{-1}x_i^+, \quad S(x_i^-) = -x_i^-k_i, \quad S(k_i^{\pm 1}) = k_i^{\mp 1},$$
$$\varepsilon (x_i^{\pm}) = 0, \quad \varepsilon(k_i^{\pm 1}) = 1,$$
for all $i\in I$, define a structure of Hopf algebra in $U_q(\g)$,
where $\Delta$ is the co-multiplication, $\varepsilon$ is the
co-unity and $S$ is the antipode. The algebra $U_q(\tlie g)$ is also
a Hopf algebra and the structure maps can be described exactly as
above using the Chevalley-Kac generators. However, a precise
expression for the comultiplication in terms of the generators
$x^\pm_{i,r}, h_{i,r}, k_i^{\pm1}$ is not known (see
\cite{cp:book} and references therein).  $U_q(\g)$ is a Hopf subalgebra of $U_q(\tlie g)$.

\subsection{The $\ell$-Weight Lattice}\label{ss:llattice}
Consider the multiplicative group $\mathcal P$ of $n$-tuples of rational
functions $\bs\mu = (\bs\mu_1(u),\ldots, \bs\mu_n(u))$ with values
in $\mathbb C$  such that $\bs\mu_i(0)=1$ for all $i\in I$. We shall refer to $\mathcal P$ as  the $\ell$-weight lattice of $U_q(\tlie g)$, to the elements of $\mathcal P$ as $\ell$-weight and to elements of the submonoid $\cal P^+$ of
$\cal P$ consisting of $n$-tuples of polynomials as dominant $\ell$-weights. Given
$a\in\mathbb C^\times$ and $i\in I$, define the fundamental $\ell$-weight $\bs\omega_{i,a}\in \cal P$ by
$$(\bs\omega_{i,a})_j(u) = 1-\delta_{i,j}au.$$
Clearly, $\cal P$ is the free abelian group generated by these
elements. If
\begin{equation}\label{e:appears}
\bs\mu = \prod_{(i,a)\in I\times\mathbb C^\times}
\bs\omega_{i,a}^{p_{i,a}}
\end{equation}
we shall say that $\bs\omega_{i,a}$ (respectively,
$\bs\omega_{i,a}^{-1}$) appears in $\bs\mu$ if $p_{i,a}>0$
(respectively, $p_{i,a}<0$).

Consider the group homomorphism (weight map) $\wt:\cal P \to P$ by
setting $\wt(\bs\omega_{i,a})=\omega_i$.  Given
$\bs\omega\in\cal P^+$ with $\bs\omega_i(u) = \prod_j (1-a_{i,j}u)$,
where $a_{i,j}\in \C$, let $\bs\omega^-\in\cal P^+$ be defined by
$\bs\omega^-_i(u) = \prod_j (1-a_{i,j}^{-1}u)$. For convenience, we will sometimes  use
the notation $\bs\omega^+ = \bs\omega$.  Given $\bs\mu\in\cal P$,
say $\bs\mu = \bs\omega\bs\varpi^{-1}$ with
$\bs\omega,\bs\varpi\in\cal P^+$, define a $\mathbb C$-algebra
homomorphism $\bs\Psi_{\bs\mu}:U_q(\tlie h)\to \mathbb C$ by setting
$\bs\Psi_{\bs\mu}(k_i^{\pm 1}) = q^{\pm \wt(\bs\mu)(h_i)}$ and
\begin{equation}\label{e:PsiLambda}
\bs\Psi_{\bs\mu}(\Lambda_i^\pm(u)) =
\frac{(\bs\omega^{\pm})_i(u)}{(\bs\varpi^{\pm})_i(u)}.
\end{equation}
One easily checks that the map $\bs\Psi:\cal P\to (U_q(\tlie h))^*$
given by $\bs\mu\mapsto \bs\Psi_{\bs\mu}$ is injective. From now on
we will identify $\cal P$  with its image in $(U_q(\tlie h))^*$
under $\bs\Psi$.

It will be convenient to introduce the following notation. Given
$i\in I, a\in\mathbb C^\times, m\in\mathbb Z_{\ge 0}$, define
\begin{equation*}
\bs\omega_{i,a,m} = \prod_{j=0}^{m-1} \bs\omega_{i,aq^{m-1-2j}}.
\end{equation*}
Also, following \cite{cm:qblock}, define
\begin{equation*}
\bs\alpha_{i,a} = \bs\omega_{i,aq,2}\prod_{j\ne i}
\bs\omega_{j,aq,-c_{j,i}}^{-1} = \bs\omega_{i,a}\bs\omega_{i,aq^2}\prod_{j:|j-i|=1} \bs\omega_{j,aq}^{-1}.
\end{equation*}
We shall refer to $\bs\alpha_{i,a}$ as a simple $\ell$-root. The
subgroup of $\cal P$ generated by the simple $\ell$-roots is called
the $\ell$-root lattice of $U_q(\tlie g)$ and will be denoted by
$\cal Q$. Let also $\cal Q^+$ be the submonoid generated by the
simple $\ell$-roots. Quite clearly $\wt(\bs\alpha_{i,a})=\alpha_i$.
Define a partial order on $\cal P$ by
$$\bs\mu\le\bs\omega \qquad{if}\qquad \bs\omega\bs\mu^{-1}\in\cal Q^+.$$
It is well-known that the elements $\bs\alpha_{i,a}$ are
multiplicatively independent, i.e., if $(i_j,a_j), j=1,\dots,m$, is
a family of distinct elements of $I\times \mathbb C^\times$, then
\begin{equation}\label{e:lrootsind}
\prod_{j=1}^m\bs\alpha_{i_j,a_j}^{k_j} = 1\ \Leftrightarrow\ k_j=0
\text{ for all }  j=1,\dots,m.
\end{equation}
For further convenience of notation, given $1\le i\le j\le n, a\in\mathbb C^\times$, we set
\begin{equation}
\bs\alpha_{i,j,a}=\prod_{k=i}^j \bs\alpha_{k,aq^{k-i}} \quad \text{and} \quad
\bs\alpha_{j,i,a}=\prod_{k=i}^j \bs\alpha_{j+i-k,aq^{k-i}}.
\end{equation}

\subsection{Finite-Dimensional Representations}
We start by reviewing some basic facts
about finite-dimensional representations of $U_q(\lie g)$. For
the details see \cite{cp:book} for instance.

Given a $U_q(\lie g)$-module $V$ and $\mu\in P$, let
$$V_\mu=\{v\in V: k_iv=q^{\mu(h_i)}v \text{ for all } i\in I\}.$$
A nonzero vector $v\in V_\mu$ is called a weight vector of weight
$\mu$. If $v$ is a weight vector such that $x_i^+v = 0$ for all
$i\in I$, then $v$ is called a highest-weight vector. If $V$ is
generated by a highest-weight vector of weight $\lambda$, then $V$
is said to be a highest-weight module of highest weight $\lambda$. A
$U_q(\lie g)$-module $V$ is said to be a weight module if
$V=\bigoplus_{\mu\in P}^{} V_\mu$.  Denote by $\cal C_q$ be the
category of all finite-dimensional weight modules of $U_q(\lie g)$.
The following theorem summarizes the basic facts about $\cal C_q$.

\begin{thm}\label{t:ciuqg} Let $V$ be an object of $\cal C_q$. Then:
\begin{enumerate}[(a)]
\item $\dim V_\mu = \dim V_{w\mu}$ for all $w\in\cal W$.
\item $V$ is completely reducible.
\item For each $\lambda\in P^+$ the $U_q(\lie g)$-module $V_q(\lambda)$ generated by a vector $v$ satisfying
$$x_i^+v=0, \qquad k_iv=q^{\lambda(h_i)}v, \qquad (x_i^-)^{\lambda(h_i)+1}v=0,\quad\forall\ i\in I,$$
is nonzero, irreducible, and finite-dimensional. If $V\in\cal C_q$ is
irreducible, then $V$ is isomorphic to $V_q(\lambda)$  for a unique
$\lambda\in P^+$.\endd
\end{enumerate}
\end{thm}

We now turn to finite-dimensional representations of $U_q(\tlie g)$.
Let $V$ be a $U_q(\tlie g)$-module. We say that a nonzero vector
$v\in V$ is an $\ell$-weight vector if there exists
$\bs\omega\in\cal P$ and $k\in\mathbb Z_{>0}$ such that
$$(\eta-\bs\Psi_{\bs\omega}(\eta))^kv=0 \quad\text{for all}\quad \eta\in U_q(\tlie h).$$
 In that case, $\bs\omega$ is said to be the $\ell$-weight of
$v$. $V$ is said to be an $\ell$-weight module if every vector of
$V$ is a linear combination of $\ell$-weight vectors. In that case,
let $V_{\bs\omega}$ denote the subspace spanned by all $\ell$-weight
vectors of $\ell$-weight $\bs\omega$. An $\ell$-weight vector $v$ is
said to be a highest-$\ell$-weight vector if
$$\eta v=\bs\Psi_{\bs\omega}(\eta)v \quad\text{for every}\quad \eta\in U_q(\tlie h)
\quad\text{and}\quad x_{i,r}^+v=0 \quad\text{for all}\quad i\in I, r\in\mathbb Z.$$
$V$ is said to be a highest-$\ell$-weight module if it is generated by a
highest-$\ell$-weight vector. Denote by $\wcal C_q$ the category of
all finite-dimensional $\ell$-weight modules of $U_q(\tlie g)$.
 $\wcal C_q$ is an abelian category stable under tensor product \cite{freres:qchar}.

Observe that if $V\in\wcal C_q$, then $V\in\cal C_q$ and
\begin{equation}\label{e:lws}
V_\lambda = \bigoplus_{\bs\omega:\wt(\bs\omega)=\lambda}^{}
V_{\bs\omega}.
\end{equation}
Moreover, if $V$ is a highest-$\ell$-weight module of highest
$\ell$-weight $\bs\omega$, then
\begin{equation}\label{e:hw}
\dim(V_{\wt(\bs\omega)}) = 1\qquad\text{and}\qquad V_\mu\ne 0
\Rightarrow \mu\le\wt(\bs\omega).
\end{equation}

The next proposition is easily established using \eqref{e:hw}.

\begin{prop}
If $V$ is a highest-$\ell$-weight module, then it has a unique maximal proper submodule and, hence, a unique irreducible quotient.\endd
\end{prop}

Given $\bs\omega\in \cal P^+$,  the Weyl module $W_q(\bs\omega)$ is
the $U_q(\tlie g)$-module defined by the quotient of $U_q(\tlie g)$
by the left ideal generated by the elements
$$x_{i,r}^+,\quad (x_{i}^-)^{\wt(\bs\omega)(h_i)+1},\quad \eta-\bs\Psi_{\bs\omega}(\eta),\quad i\in I,\ r\in\mathbb Z,\ \eta\in U_q(\tlie h).$$
In particular, it is a
highest-$\ell$-weight module. Denote by $V_q(\bs\omega)$ the
irreducible quotient of $W_q(\bs\omega)$. The next theorem was
proved in \cite{cp:weyl} and recovers the classification of the
simple objects of $\wcal C_q$ obtained previously in \cite{cp:qaa}.

\begin{thm}\label{t:weyl}
\begin{enumerate}[(a)]
\item For every $\bs\omega\in\cal P^+$, the module $W_q(\bs\omega)$ is
nonzero and, moreover, it is the universal finite-dimensional
$U_q(\tlie g)$-module with highest $\ell$-weight $\bs\omega$. 
\item If $V$ is a simple object of $\wcal C_q$, there exists unique $\bs\omega\in\cal P^+$ such that  $V\cong V_q(\bs\omega)$.\endd
\end{enumerate}	
\end{thm}

\subsection{Minimal Affinizations}\label{ss:min}

We now review the notion of minimal affinizations of an irreducible
$U_q(\lie g)$-module introduced in \cite{cha:minr2}.

Given $\lambda\in P^+$, a $U_q(\tlie g)$-module $V$ is said to be an
affinization of $V_q(\lambda)$ if there exists an isomorphism of
$U_q(\lie g)$-module,
\begin{equation}
V\cong V_q(\lambda)\oplus \opl_{\mu< \lambda}^{} V_q(\mu)^{\oplus
m_\mu(V)}
\end{equation}
for some $m_\mu(V)\in\mathbb Z_{\ge 0}$. Two affinizations of
$V_q(\lambda)$ are said to be equivalent if they are isomorphic as
$U_q(\lie g)$-modules. Notice that a highest-$\ell$-weight module of
highest $\ell$-weight $\bs\omega\in\cal P^+$ is an affinization of
$V_q(\lambda)$ if and only if $\wt(\bs\omega)=\lambda$.

The partial order on $P^+$ induces a natural partial order on the
set of  (equivalence classes of) affinizations of $V_q(\lambda)$.
Namely, if $V$ and $W$ are affinizations of $V_q(\lambda)$, say that
$V\le W$ if one of the following conditions hold:
\begin{enumerate}
\item $m_\mu(V)\le m_\mu(W)$ for all $\mu\in P^+$;
\item for all $\mu\in P^+$ such that $m_\mu(V)>m_\mu(W)$
there exists $\nu>\mu$ such that $m_\nu(V)<m_\nu(W)$.
\end{enumerate}
A minimal element of this partial order is said to be a minimal
affinization. Clearly, a minimal affinization of $V_q(\lambda)$ must
be irreducible as a $U_q(\tlie g)$-module and, hence, it is of the form
$V_q(\bs\omega)$ for some $\bs\omega\in\cal P^+$ such that
$\wt(\bs\omega)=\lambda$.

Given $i,j\in I, i\le j$, and $\lambda\in P$, set
\begin{equation*}
_i|\lambda|_j = \sum_{k=i}^j \lambda(h_k).
\end{equation*}
If $i=1$, we simplify notation and write $|\lambda|_j$ instead of
$_1|\lambda|_j$ and similarly if $j=n$. For $i>j$ we set
$_i|\lambda|_j=0$. Set also
\begin{equation*}
p_{i,j}(\lambda) =\ _{i+1}|\lambda|_j +\ _i|\lambda|_{j-1} + j-i
\quad\text{if}\quad i\le j
\end{equation*}
and $p_{i,j}(\lambda) = p_{j,i}(\lambda)$ if $j<i$. Notice that, if
$i=j$, then $p_{i,j}(\lambda)=0$ and
\begin{equation*}
p_{i,j}(\lambda) = \lambda(h_i)+\lambda(h_j)+ 2\
_{i+1}|\lambda|_{j-1}+ j-i \qquad\text{if}\quad i<j.
\end{equation*}
The next theorem proved in \cite[Theorem 2.9]{cp:small} (see also \cite{cp:minsl}) gives the classification of the minimal affinizations in type $A$.

\begin{thm}\label{teo tipo A}
For every $\lambda \in P^+$, there exists a unique class of minimal affinizations of $V_q(\lambda)$.
Moreover, $V_q(\bs\omega)$ is a minimal affinization of $V_q(\lambda)$ if and
only if there exist $a_i\in\mathbb C^\times, i\in I$, and
$\epsilon=\pm 1$ such that
\begin{equation}\label{eq cond min A}
\bs\omega=\prod_{i\in I} \bs\omega_{i,a_i,\lambda(h_i)}
\qquad\text{with}\qquad \frac{a_i}{a_j}= q^{\epsilon
p_{i,j}(\lambda)} \qquad\text{for all}\quad i<j.
\end{equation}
In that case, $V_q(\bs\omega)\cong V_q(\lambda)$ as representations
of $U_q(\g)$.\endd
\end{thm}

Notice that \eqref{eq cond min A} is equivalent to saying that there
exist $a\in\mathbb C^\times$ and $\epsilon=\pm 1$ such that
\begin{equation}\label{eq cond min A'}
\bs\omega=\prod_{i\in I} \bs\omega_{i,a_i,\lambda(h_i)}
\qquad\text{with}\qquad a_i=aq^{\epsilon p_{i,n}(\lambda)}
\qquad\text{for all}\quad i\in I.
\end{equation}
The support
of $\mu\in P$ is defined by
$${\rm supp}(\mu)=\{i\in I:\mu(h_i)\ne 0\}.$$
Note that if $\#\ \supp(\lambda)>1$, the pair $(a,\epsilon)$ in
\eqref{eq cond min A'} is unique. In that case, if $\bs\omega$
satisfies \eqref{eq cond min A} with $\epsilon=1$, we say that
$V_q(\bs\omega)$ is a decreasing minimal affinization (because the
powers of $q$ in \eqref{eq cond min A'} decrease as $i$ increases).
Otherwise, we say $V_q(\bs\omega)$ is an increasing minimal
affinization. However, if $\#\ \supp(\lambda)= 1$, $\bs\omega$ can
be represented in the form \eqref{eq cond min A'} by two choices of
pairs $(a,\epsilon)$, one for each value of $\epsilon$. We do not
fix a preferred presentation in that case. The minimal affinizations
satisfying $\#\ \supp(\lambda)\le 1$ are
called Kirillov-Reshetikhin modules.

\subsection{The Main Theorem}\label{ss:maint}

Fix $\bs\omega$ as in \eqref{eq cond min A} with $\epsilon=-1$ and set
$$i_0=\max(\supp(\lambda)), \qquad a=a_{i_0}.$$
Thus,
\begin{equation}\label{eq r_l r_i_0}
\bs\omega=\prod_{i\in\supp(\lambda)}
\bs\omega_{i,aq^{s_i},\lambda(h_i)} \qquad\text{with}\qquad s_i=-p_{i,i_0}(\lambda).
\end{equation}
For notational convenience, we define $s_i$ by \eqref{eq r_l r_i_0} for all $1\le i\le i_0$. Fix also
\begin{equation*}
\bs\varpi = \bs\omega_{n,b,k} \qquad\text{for some}\qquad b\in\mathbb C^\times, k\in\mathbb Z_{>0},
\end{equation*}
and set
\begin{equation*}
\bs\lambda = \bs\omega\bs\varpi \qquad\text{and}\qquad V = V_q(\bs\omega)\otimes V_q(\bs\varpi).
\end{equation*}

\begin{thm}\label{t:main}
V is reducible if and only if there exist $s\in\mathbb Z, p\in\supp(\lambda)$, and $k'>0$ such that $b=aq^s$ and either one of the following options hold:
\begin{enumerate}[(i)]
\item $k'\le\min\{\lambda(h_p),k\}$ and $s+k+n-p+2 = s_p-\lambda(h_p)+2k'$;
\item $k'\le\min\{|\lambda|,k\}$, $\lambda(h_{i_0})+n-i_0+2 = s-k+2k'$, and $p=\max\{i\in I:\ _i|\lambda|\ge k'\}$.
\end{enumerate}
In both cases, $V$ has length 2 and the highest $\ell$-weight $\bs\lambda'$ of the simple factor not isomorphic to $V_q(\bs\lambda)$ is
\begin{equation*}
\bs\lambda' =
\begin{cases}
\bs\lambda\left( \prod\limits_{l=1}^{k'} \bs\alpha_{n,p,aq^{s+k-1-2(l-1)}}^{-1}\right),& \text{if (i) holds},\vspace{10pt}\\
\bs\lambda\
\left(\prod\limits_{i=p+1}^{i_0}\prod\limits_{m=1}^{\lambda(h_i)}\bs\alpha^{-1}_{i,n,aq^{s_i+\lambda(h_i)-1-2(m-1)}}\right)
\left(\prod\limits_{m=1}^{d}\bs\alpha^{-1}_{p,n,aq^{s_p+\lambda(h_p)-1-2(m-1)}}\right),& \text{if (ii) holds},
\end{cases}\vspace{5pt}
\end{equation*}
where $d=k'-\ _{p+1}|\lambda|$.\endd
\end{thm}

This is the main result of this paper. It describes the unique non-obvious irreducible factor of the tensor product of an increasing minimal affinization with a Kirillov-Reshetikhin module associated to $\omega_n$, when it exists, as well as the precise condition for its existence. Duality arguments can be used to obtain similar description for the tensor product for all combinations between an increasing or decreasing minimal affinization and a Kirillov-Reshetikhin module associated to $\omega_1$ or $\omega_n$. We give the precise statements in Section \ref{s:othermain} and explain how to obtain the other cases from Theorem \ref{t:main}. Moreover, combining these duality arguments with the main result of \cite{cha:braid}, we will also prove the following in Section \ref{s:othermain}.
Let $$V' = V_q(\bs\varpi)\otimes V_q(\bs\omega).$$
It is well-know that the Grothendieck ring of $\wcal C_q$ is commutative and, hence , Theorem \ref{t:main} applies as is to $V'$ in place of $V$.

\begin{cor}\label{c:main}
	$V$ and $V'$ are indecomposable. Moreover, if condition (i) of Theorem \ref{t:main} holds, we have short exact sequences
	 \begin{equation*}
	 0\to V_q(\bs{\lambda}')\to V\to V_q(\bs{\lambda})\to 0 \qquad\text{and}\qquad 0\to V_q(\bs{\lambda})\to V'\to V_q(\bs{\lambda}')\to 0
	 \end{equation*}
	while we have
	\begin{equation*}
	0\to V_q(\bs{\lambda})\to V\to V_q(\bs{\lambda}')\to 0 \qquad\text{and}\qquad 0\to V_q(\bs{\lambda}')\to V'\to V_q(\bs{\lambda})\to 0
	\end{equation*}
	 if condition (ii) holds. \endd
\end{cor}

\begin{rem}\label{r:tmain}
    We chose to write Theorem \ref{t:main} for a Kirillov-Reshetikhin module associated to $\omega_n$ since this makes the notation of this paper closer to that needed for its application to the classification of minimal affinizations of type $D$. It will be clear from the proof of the theorem that the pair $(p,k')$ is unique if it exists. In fact, this is already obvious in the case that condition (ii) holds. The uniqueness in the case of (i) is proved after \eqref{eq p,k'} below. Moreover, it follows from \eqref{e:inoii} that conditions (i) and (ii)  cannot be simultaneously satisfied. One easily checks that the equation in condition (i) of Theorem \ref{t:main} is equivalent to saying that 
    $$\bs\omega_{p,aq^{s_p},\lambda(h_p)}\ \bs\omega_{n,aq^{s-(k'-1)},k-(k'-1)}$$ 
    corresponds to a (necessarily decreasing) minimal affinization. Similarly, the one in condition (ii) is equivalent to saying that 
    $$\bs\omega_{p,aq^{s_p+ _{p+1}|\lambda|-(k'-1)}, _p|\lambda| - (k'-1)}\ \bs\omega_{n,aq^s,k}$$
    corresponds to a (necessarily increasing) minimal affinization. Moreover, in this case,
    $$\left(\prod_{i<p}\bs\omega_{i,aq^{s_i},\lambda(h_i)}\right) \bs\omega_{p,aq^{s_p+ _{p+1}|\lambda|-(k'-1)}, _p|\lambda| - (k'-1)}\ \bs\omega_{n,aq^s,k}$$
    corresponds to a minimal affinization as well. \endd
\end{rem}

\section{Character Theory}\label{s:char}

In this section, we present the main tool we will use to prove Theorem \ref{t:main}: the notion of qcharacter of a representation of $U_q(\tlie g)$ introduced by Frenkel and Reshetikhin in \cite{freres:qchar} (see also \cite{cm:qblock}). More precisely, we shall use Nakajima's tableaux description of the qcharacters \cite{nak:tableau}.

\subsection{Characters and qCharacters}
Let $\mathbb Z[P]$ be the integral group ring over $P$ and denote by
$e:P\to\mathbb Z[P], \lambda\mapsto e^\lambda$, the inclusion of $P$
in $\mathbb Z[P]$ so that $e^\lambda e^\mu=e^{\lambda+\mu}$. The
character of an object $V$ from $\cal C_q$ is defined by
\begin{equation*}\label{e:chd}
\ch(V) = \sum_{\mu\in P} \dim(V_\mu) e^\mu.
\end{equation*}
For $\lambda\in P^+$, let $m_\lambda(V)$ be the multiplicity of
$V_q(\lambda)$ as a simple factor of $V$. It is well-known that the
numbers $m_\mu(V)$ can be computed from $\ch(V)$ and vice-versa.

Similarly, for an object $V$ from $\wcal C_q$ and $\bs\omega\in\cal
P^+$, let $m_{\bs\omega}(V)$ be the multiplicity of $V_q(\bs\omega)$
as a simple factor of $V$. We now turn to the concept which plays a
role analogous to character for the category $\wcal C_q$. It was
introduced in \cite{freres:qchar} under the name of qcharacter. In
particular, one can compute the multiplicities $m_{\bs\omega}(V)$
from the qcharacter of $V$.

Let $\Z[\cal P]$ be the integral group ring over $\cal P$. Given
$\chi\in \Z[\cal P]$, say $$\chi = \sum_{\bs\mu\in\cal P}
\chi(\bs\mu)\ \bs\mu,$$ we identify it with the function $\cal
P\to\mathbb Z, \bs\mu\to \chi(\bs\mu)$. Conversely, any function
$\cal P\to\mathbb Z$ with finite support can be identified with an
element of $\Z[\cal P]$. The qcharacter of $V\in \wcal C_q$ is the
element $\qch(V)$ corresponding to the function
\begin{equation*}
\bs\mu\mapsto \dim(V_{\bs\mu}).
\end{equation*}
We shall denote by $\wt_\ell(\chi)$ the support
of $\chi\in\Z[\cal P]$. In particular, we set
\begin{equation*}
\wt_\ell(V) = \wt_\ell(\qch(V)) = \{\bs\mu\in\cal P:V_{\bs\mu}\ne
0\}.
\end{equation*}

Given an $\ell$-weight module $V$ and a vector subspace $W$ of $V$,
let $W_{\bs\mu}=W\cap V_{\bs\mu}$. We shall say that $W$ is an
$\ell$-weight subspace of $V$ if
\begin{equation*}
W = \opl_{\bs\mu\in\cal P}^{} W\cap V_{\bs\mu}.
\end{equation*}
In that case, we set
\begin{equation*}
\qch(W) = \sum_{\bs\mu\in\cal P} \dim W_{\bs\mu}\ \bs{\mu} \qquad\text{and}\qquad \wtl(W) = \wtl(\qch(W)).
\end{equation*}

Although the tensor product of $\ell$-weight vectors is not an $\ell$-weight vector in general, we still have the following result \cite[Lemma 2]{freres:qchar}:

\begin{prop}\label{p:qchtp}
For every $V,W\in\wcal C_q$, $\qch(V\otimes
W)=\qch(V)\qch(W)$.\endd
\end{prop}

\subsection{Tableaux and $\ell$-Weights}\label{ss:tabqch}
In this subsection we review Nakajima's description of $\ell$-weights in terms of tableaux  \cite{nak:tableau}.
Fix $a\in\mathbb C^\times$ and set
\begin{equation}\label{eq Y A}
Y_{i,r}:=\bs\omega_{i, aq^r} \qquad\text{and}\qquad
A_{i,r}:=\bs\alpha_{i,aq^{r-1}}, \quad i\in I,r\in \Z.
\end{equation}
In particular,
\begin{equation}\label{e:AfromYs}
A_{i,r} = Y_{i,r-1} Y_{i,r+1} Y_{i-1,r}^{-1} Y_{i+1,r}^{-1},
\end{equation}
where we set $Y_{0,r} = Y_{n+1,r} = 1$ for convenience. We also introduce the following notation. Given $i,j\in I, i\le j,
r\in\Z, m\in \Z_{\ge 0}$, define
\begin{equation*}
Y_{i,r,m} = \prod_{k=0}^{m-1} Y_{i,r+2k}=\bs\omega_{i,aq^{m+r-1},m},
\end{equation*}
\begin{equation}\label{e:A3index}
A_{i,j,r}=\prod_{k=i}^j A_{k,r+k-i+1} \quad \text{and} \quad
A_{j,i,r}=\prod_{k=i}^j A_{j+i-k,r+k-i+1}.
\end{equation}

\begin{rem}
The reason for defining $A_{i,r}:=\bs\alpha_{i,aq^{r-1}}$ instead of
simply $A_{i,r}:=\bs\alpha_{i,aq^r}$ is to match with the notation of
\cite{freres:qchar}. The notation $Y_{i,s}$ and $A_{i,s}$, originally used in  \cite{freres:qchar}, is more commonly used in the literature about qcharaters and this is the reason we switch to this notation. Note that, if $\bs\omega$ is as in \eqref{eq r_l r_i_0}, then
\begin{equation}\label{eq l-weight min aff}
\bs\omega=\prod_{i\le i_0} Y_{i,r_i,\lambda(h_i)} \quad\text{where}\quad  r_i = s_i-\lambda(h_i)+1.
\end{equation}
In particular, for all $i\le i_0$ we have
\begin{equation}\label{e:rifromri0}
r_i=r_{i_0}-2\: _{i}|\lambda|_{i_0-1}+i-i_0.
\end{equation}
\endd
\end{rem}

Consider the fundamental representation $V_q(Y_{1,s})=V_q(\bs\omega_{1,aq^s})$. Its
qcharacter is known to be given by
\begin{equation}\label{e:qchstrep}
\qch(V_q(Y_{1,s})) = Y_{1,s}\left(1+\sum_{j=1}^n
A_{1,j,s}^{-1}\right) = \sum_{i=0}^n Y_{i,s+i+1}^{-1}Y_{i+1,s+i}.
\end{equation}
Represent the element $Y_{i-1,s+i}^{-1}Y_{i,s+i-1}, i=1,\dots,n+1$,
by the picture $\ffbox{i}_{\:s}$. Given such a box
$\ffbox{i}_{\:s}$, we shall refer to $i$ as the content of the box
and to $s$ as its support. Thus, $\wt_\ell(V_q(Y_{1,s}))$ can be
described by the following graph
\begin{equation*}
\ffbox{1}_{\:s} \xrightarrow{1,s+1} \ffbox{2}_{\:s}
\xrightarrow{2,s+2} \cdots \xrightarrow{n,s+n} \ffbox{n\!+\!1}_{\:s}
\end{equation*}
where the label $(i,s+i)$ on the $i$-th arrow indicates that
$\ffbox{i+1}_{\:s}$ is obtained from $\ffbox{i}_{\:s}$ by
multiplication by $A_{i,s+i}^{-1}$. Then, \eqref{e:qchstrep} can be re-written as
\begin{equation*}
\qch(V_q(Y_{1,s})) = \sum_{i=1}^{n+1}\ \ffbox{i}_{\:s}.
\end{equation*}

Let $\mathbf B = \{ 1,\dots,n+1\}$ equipped with the usual ordering
$<$ coming from $\mathbb Z$. Given $k,s\in\mathbb Z, k> 0$, a
column tableau $T$ of length $k$ with support starting in $s$ is a map
\begin{equation*}
   T: \{1,\dots,k\} \to \mathbf B\times \mathbb Z
\end{equation*}
such that, if we denote by $T(j)_2$ the $\mathbb Z$-component of $T(j)$, then
\begin{equation}\label{eq T(j)_2}
T(j)_2 = s+2(k-j) \qquad\text{for all}\qquad j=1,\dots,k.
\end{equation}
We represent $T$ by the picture:
\begin{equation}\label{e:dcoltab}
\renewcommand{\arraycolsep}{0pt}
\begin{array}{|c|}
   \hline \hbox to 0.55cm{\hfill$\scriptstyle i_1$\hfill} \\
   \hline \hbox to 0.55cm{\hfill$\vdots$\hfill} \\
   \hline \hbox to 0.55cm{\hfill$\scriptstyle{i_k}$\hfill} \\
   \hline
\end{array}_{\:s}  \qquad\text{where}\qquad i_j = T(j)_1
\end{equation}
and $T(j)_1$ denotes the $\mathbf B$-component of $T(j)$. For
notational convenience, we set $T(0)_1=0.$ Notice that we can think
of this picture as a vertical juxtaposition of the boxes
$\ffbox{i_j}_{\:s+2(k-j)}$ with explicit mention of the support of
the $k$-th box only since the others are recovered from it. Given
such a tableau, we associate to it an element $\bs\omega^T\in\cal P$
given by
\begin{equation*}
\bs\omega^T = \prod_{j=1}^{k}\ \ffbox{i_j}_{\:s+2(k-j)}.
\end{equation*}

\begin{rem}
Nakajima's original definition regards $T$ as a map $\mathbb
Z\to\mathbf B\cup\{0\}$ such that $T(a)=0$ if and only if
$a\notin\{s,s+2,\dots,s+2(k-1)\}$. Thus, in our notation, $T(j)_2$
corresponds to the $j$-th element of the support of $T$ in
Nakajima's notation while $T(j)_1$ is the value it assumes at that
element.\endd
\end{rem}

A tableau $T$ is a finite sequence of column tableaux
$T=(T_1,T_2,\ldots,T_m)$. If $T_j$ has length $k_j$ and support starting at $s_j$, the shape of $T$ is defined as
the sequence $((k_1,s_1),(k_2,s_2),\ldots,(k_m,s_m))$. We represent $T$
graphically by ordered horizontal juxtaposition of the associated pictures
\eqref{e:dcoltab} in such a way that the boxes with equal support form a horizontal row:
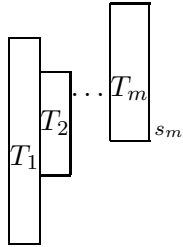
\begin{equation*}
\renewcommand{\tabcolsep}{0pt}
\begin{tabular}{|c|c|c|c|c|c}
\cline{5-5} \multicolumn{3}{c}{}&&&
\\
\cline{1-1} &\multicolumn{2}{c}{}&&&
\\
\cline{2-2} &&\multicolumn{1}{c}{\dots}&&$T_m$&
\\
&$T_{2}$&\multicolumn{1}{c}{}&&&${}_{\:s_m}$
\\
\cline{5-5} $T_1$&&\multicolumn{3}{c}{}
\\
\cline{2-2} &\multicolumn{4}{c}{}&
\\
&\multicolumn{4}{c}{}&
\\
\cline{1-1}
\end{tabular}
\end{equation*}
In particular, if the picture is connected, the supports of all
boxes have the same parity and can be recovered from the support
$s_m$ of the last box of $T_m$.  We associate to  a tableau $T$ the
element $\bs\omega^T\in\cal P$ given by
\begin{equation*}
\bs\omega^T = \prod_{j=1}^m \bs\omega^{T_j}.
\end{equation*}
Henceforth, we shall only consider tableaux whose associated picture is connected.

A tableau $T$ is said to be column-increasing (or simply increasing) if the contents in each
column strictly increase from top to bottom. Note that a column
tableau is increasing of length equal to the content of its last box if
and only if $T(j)_1 = j$ for all $j$. In pictures, $T$ is of the form
\begin{equation}\label{e:coldomtab}
\renewcommand{\arraycolsep}{0pt}
\begin{array}{|c|}
   \hline \hbox to 0.55cm{\hfill$\scriptstyle 1$\hfill} \\
   \hline \hbox to 0.55cm{\hfill$\scriptstyle 2$\hfill} \\
   \hline \hbox to 0.55cm{\hfill$\vdots$\hfill} \\
   \hline \hbox to 0.55cm{\hfill$\scriptstyle{i}$\hfill} \\
   \hline
\end{array}_{\:s}
\end{equation}
for some $i \in \{ 1, 2,\ldots, n+1\}, s\in\Z$. If $T$
is such a column tableau, then
\begin{equation}\label{e:fundfromtab}
\bs\omega^T = Y_{i,s+i-1}.
\end{equation}
In particular, if $T$ is an increasing column tableau of length
$n+1$, i.e., if $T$ has the form
\begin{equation*}
\renewcommand{\arraycolsep}{0pt}
\begin{array}{|c|}
   \hline\hbox to 0.55cm{\hfill$\scriptstyle 1$\hfill}\\
   \hline\hbox to 0.55cm{\hfill$\scriptstyle 2$\hfill}\\
   \hline\hbox to 0.55cm{\hfill$\vdots$\hfill}\\
   \hline\hbox to 0.55cm{\hfill$\scriptstyle{n\!+\!1}$\hfill}\\
   \hline
\end{array}_{\:s}
\end{equation*}
for some $s\in\mathbb Z$, then $\bs\omega^T= 1$. Hence, adding
increasing columns of length $n+1$ to a tableau $T$ does not change
$\bs\omega^T$.

Two tableaux $T$ and $T'$ are said to be equivalent if, for all
$(i,a)\in\mathbf B\times\mathbb Z$, we have
$$\#\{j: (i,a)\in\mathcal Im(T_j)\}\;=\;\#\{j: (i,a)\in\mathcal Im(T_j')\}.$$
In terms of pictures, $T'$ is obtained from $T$ by permuting the contents of the boxes in the
same row. It is easy to see that $\bs\omega^T=\bs\omega^{T'}$  if
$T$ and $T'$ are equivalent. The converse is not true, but ``almost'':

\begin{lem}\cite[Lemma 4.4]{nak:tableau}\label{lem lambdaT=lambdaT'}
Let $T$ and $T'$ be tableaux. The elements $\bs\omega^T$ and
$\bs\omega^{T'}$ are equal if and only if $T$ and $T'$ become
equivalent after adding several increasing column tableaux of length
$n+1$ to $T$ and $T'$.\endd
\end{lem}

\begin{lem}\cite[Lemma 4.5]{nak:tableau}\label{lem ldom}
Let $T$ be a tableau. Then, $\bs\omega^T\in\cal P^+$ if and only if
$T$ is equivalent to a tableau $T'$ whose columns are of the form \eqref{e:coldomtab}.\endd
\end{lem}

We end this subsection presenting the elementary modifications in column tableaux associated to $\ell$-roots.
Let $T$ be a tableaux of shape $(k,s)$ and suppose
$j\in\{1,\dots,k\}$ is such that $i:=T(j)_1\le n$. Then, given $i\le i'\le n$, one easily checks that $\bs\omega^T
A_{i,i',s+2(k-j)+i-1}^{-1} = \bs\omega^{T'}$ where $T'$ is
obtained from $T$ by replacing the content of the $j$-th box by
$i'+1$. In pictures:
\begin{equation}\label{e:rem action tableau A}
\renewcommand{\arraycolsep}{0pt}
\begin{array}{|c|l}
   \cline{1-1} \hbox to 0.55cm{\hfill$\vdots$\hfill} & \\
   \cline{1-1} \hbox to 0.55cm{\hfill$\scriptstyle i$\hfill} & {} \\
   \cline{1-1} \hbox to 0.55cm{\hfill$\vdots$\hfill} & {}_{{}_{\:s}}\\
  \cline{1-1}
\end{array}\ A_{i,i',s+2(k-j)+i-1}^{-1} = \
\renewcommand{\arraycolsep}{0pt}
\begin{array}{|c|l}
   \cline{1-1} \hbox to 0.55cm{\hfill$\vdots$\hfill} & \\
   \cline{1-1} \hbox to 0.55cm{\hfill$\scriptstyle i'\!+\!1$\hfill} & {} \\
   \cline{1-1} \hbox to 0.55cm{\hfill$\vdots$\hfill} & {}_{{}_{\: s}}\\
  \cline{1-1}
\end{array}.
\end{equation}

\subsection{The qCharacters of Minimal Affinizations}\label{ss:qchminA}

We now study the qcharacters of minimal affinziations in terms of tableaux.

A tableau $T$ with shape $((k_1,s_1),(k_2,s_2),\ldots,(k_m,s_m))$ is
said to be semi-standard if it is column-increasing and satisfies:
\begin{enumerate}
\item $s_1\ge s_2\ge\cdots\ge s_m$;
\item $(i,s)\in\mathcal Im(T_j)$ and $(i',s-2)\in\mathcal Im(T_{j+1}) \Rightarrow i\ge i'$.
\end{enumerate}
In terms of pictures, the sequences of diagonal contents from left
to right and top to bottom  are decreasing (not necessarily
strictly).  Notice that this implies that the sequences of row contents are strictly decreasing.
Given a tableau $T$, we will denote by $\stab(T)$ the set
of semi-standard tableau with the same shape as $T$.

Recall the definition of increasing and decreasing minimal
affinizations given after Theorem \ref{teo tipo A}. Untill the end of this subsection we fix $\lambda\in
P^+$ and $\bs\omega\in \cal P^+$ such that $V_q(\bs\omega)$ is a minimal affinization of $V_q(\lambda)$. Suppose first that $V_q(\bs\omega)$ is an
increasing minimal affinization  and, hence, $\bs\omega$ is give by \eqref{eq l-weight min aff}.
Using \eqref{e:fundfromtab}, one easily sees that $\bs\omega=\bs\omega^T$, where $T=(T^n,T^{n-1},\dots,T^1)$
with $T^i$ omitted if $\lambda(h_i)=0$
and, otherwise, $T^i =(T_1^i,\dots,T_{\lambda(h_i)}^i)$ with
$T_{j}^i$ column-increasing with length equal the content of its
last box and support starting at $r_{i}+2(\lambda(h_i)-j)-i+1$:
\begin{equation}\label{eq tableau min aff}
T_{j}^{i}=\renewcommand{\arraycolsep}{0pt}
\begin{array}{|c|}
   \hline\hbox to 0.55cm{\hfill$\scriptstyle 1$\hfill}\\
   \hline\hbox to 0.55cm{\hfill$\vdots$\hfill}\\
   \hline\hbox to 0.55cm{\hfill$\scriptstyle{i}$\hfill}\\
   \hline
\end{array}_{\:r_{i}+2(\lambda(h_i)-j)-i+1}
\end{equation}
Notice that $T$ has $|\lambda|$ columns and
\begin{equation}\label{e:colofmina}
\bs\omega^{T_{j}^{i}}=Y_{i,r_{i}+2(\lambda(h_{i})-j)}.
\end{equation}
Observe also that, if the support of the $j$-th column of $T$ starts at
$s$, then that of the $(j+1)$-th column starts at $s-2$. Indeed,
if they are both columns of $T^i$, this is obvious. Otherwise,
consider the last column of $T^i$ and suppose the next column is the
first one of $T^k$. Then,
\begin{eqnarray*}
\left(r_{i}-i+1\right)-\left(r_{k}+2(\lambda(h_{k})-1)-k+1\right)&=&
\left(r_{i}-(r_k+2(\lambda(h_{k})-1))\right)-i+k \\
&\stackrel{\eqref{e:rifromri0}}{=}& (i-k+2)-i+k \\
&=&2.
\end{eqnarray*}
Moreover, the top of each column is in a row below  the top of the previous column because the length of the rows decrease. In pictures, $T$ has the form:

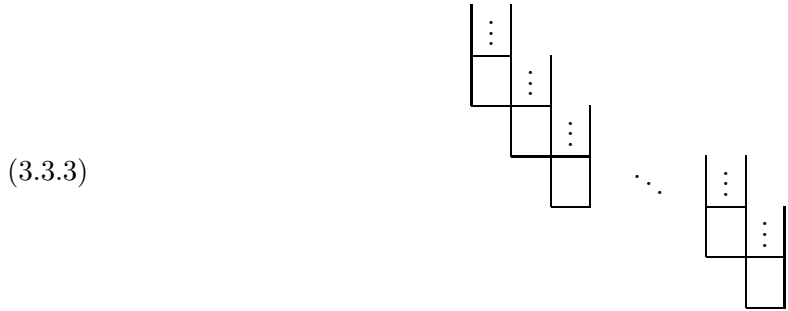
\begin{equation}\label{rem tableau stairs}
\renewcommand{\tabcolsep}{0pt}
\begin{tabular}{c|c|c|c|c|c|c|c|c|}
& \vdots & \multicolumn{7}{c}{} \\
\cline{2-2} & \phantom{\dots}   & \vdots &  \multicolumn{6}{c}{} \\
\cline{2-3} \multicolumn{1}{c}{} &  & \phantom{\dots} & \vdots & \multicolumn{5}{c}{} \\
\cline{3-4} \multicolumn{2}{c}{} &\phantom{\vdots} & \phantom{\dots} &
\multicolumn{1}{c}{\phantom{\dots}} & \multicolumn{1}{c}{$\ddots$} & & \vdots & \multicolumn{1}{c}{}  \\
\cline{4-4} \cline{8-8} \multicolumn{5}{c}{} & \multicolumn{1}{c}{\phantom{\dots}} &  &  & \vdots \\
\cline{8-9} \multicolumn{6}{c}{\phantom{\vdots}} & \multicolumn{1}{c}{\phantom{\dots}} & \phantom{\dots} &   \\
\cline{9-9} \multicolumn{8}{c}{} &
\multicolumn{1}{c}{\phantom{\dots}}
\end{tabular}
\end{equation}

Similarly, if $V_q(\bs\omega)$ is a decreasing minimal affinization, then $\bs\omega=\bs\omega^T$, where $T=(T^1,T^{2},\dots,T^n)$ with $T^i$ omitted if $\lambda(h_i)=0$
and, otherwise, $T^i =(T_1^i,\dots,T_{\lambda(h_i)}^i)$ with
$T_{j}^i$ as in \eqref{eq tableau min aff}. Again, $T$ has $|\lambda|$ columns and, if the support of the first box of the $j$-th column of $T$ is $s$, then
the support of the first box of the $(j+1)$-th column is $s-2$. This time, the bottom of each column is in a row below the bottom of the previous column and we
 have a picture of  the form:
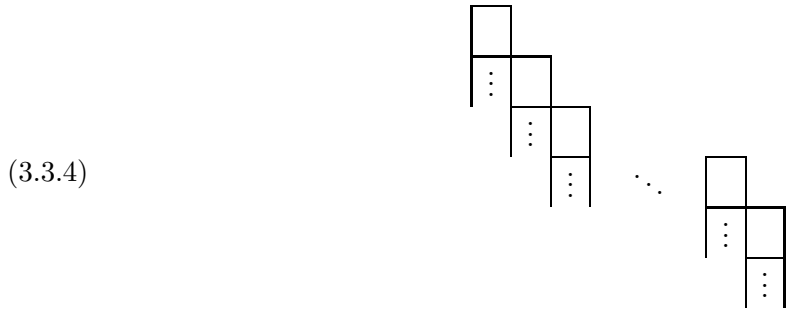
\begin{equation}\label{rem tableau stairs dec}
\renewcommand{\tabcolsep}{0pt}
\begin{tabular}{c|c|c|c|c|c|c|c|c|}
\cline{2-2}& \phantom{\vdots} & \multicolumn{7}{c}{} \\
\cline{2-3} &  \vdots   & \phantom{\dots} &  \multicolumn{6}{c}{} \\
\cline{3-4} \multicolumn{1}{c}{} & \phantom{\dots} & \vdots & \phantom{\dots} & \multicolumn{5}{c}{} \\
\cline{4-4}  \cline{8-8} \cline{8-8} \multicolumn{2}{c}{}
&\phantom{\vdots} & \vdots &
\multicolumn{1}{c}{\phantom{\dots}} & \multicolumn{1}{c}{$\ddots$} & & \phantom{\dots} & \multicolumn{1}{c}{}  \\
\cline{8-9} \multicolumn{5}{c}{} & \multicolumn{1}{c}{\phantom{\dots}} &  & \vdots & \phantom{\vdots} \\
\cline{9-9} \multicolumn{6}{c}{\phantom{\vdots}} &
\multicolumn{1}{c}{\phantom{\dots}} & \phantom{\dots} & \vdots  \\
 \multicolumn{8}{c}{} & \multicolumn{1}{c}{\phantom{\dots}}
\end{tabular}
\end{equation}

Henceforth, when we say that $T$ is the semi-standard tableau such that $\bs\omega^T=\bs\omega$, we mean the tableau we have described above. For the remainder of this subsection $T$ denotes this tableau.

\begin{lem}\label{l:lmin}
If $S\in\stab(T)\setminus\{T\}$, then $\bs\omega^S\notin\mathcal P^+$.
\end{lem}

\begin{proof}
Suppose $\bs\omega^S\in\mathcal P^+$. Then, by Lemma \ref{lem ldom},  $S$ is equivalent to a tableau having all of its columns  of the form \eqref{e:coldomtab}. But $T$ is clearly the unique element of \stab(T) with this property. 
\end{proof}

The next theorem describes the qcharacters of minimal affinizations in terms of semi-standard tableaux. For the proof see \cite[Theorems 3.8 and 3.10]{her:min}, \cite[Corollary 7.6 and Remark 7.4 (i)]{muyou:path},  and references therein. We recall that $V\in\wcal C_q$ is said to be $\ell$-minuscule (or special) if
$\#\wt_\ell(V)\cap\mathcal P^+=1$ and it is said to be thin (or quasi $\ell$-minuscule) if  $\dim (V_q(\bs\omega)_{\bs\mu})\le 1$ for all
$\bs\mu\in\mathcal P$.

\begin{thm}\label{t:qcharmin}
$V_q(\bs\omega)$ is thin and
$$\qch(V_q(\bs\omega))=\sum_{S\in\stab(T)} \bs\omega^{S}.$$
In particular, $V_q(\bs\omega)$ is also $\ell$-minuscule.\endd
\end{thm}

Notice that it follows from this theorem that
\begin{equation}
S,S'\in\stab(T) \quad\Rightarrow\quad \bs\omega^S=\bs\omega^{S'} \quad\text{iff}\quad S=S'.
\end{equation}

\begin{ex}\label{ex l-weights KR n}
We now make explicit the qcharacter of the Kirillov-Reshetikhin modules associated to $k\omega_n$. Thus, suppose $\bs\varpi=Y_{n,r,k}$ for some $r$ and set $V=V_q(\bs\varpi)$. Note that $T=(T_1,\ldots,T_k)$ is the
semi-standard tableau where each column $T_j$ has length $n$, the
content of the last box is $n$, and the shape of $T$ is
$$((n,r+1+2(k-1)-n),\dots, (n, r+3-n),(n,r+1-n)).$$
By Theorem \ref{t:qcharmin}, the $\ell$-weights of
$V$ are given by elements of $\stab(T)$.  Thus, we have the highest $\ell$-weight
$\bs\varpi=\bs\omega^T$ and all the $\ell$-weights in
$\wt_\ell(V)\setminus\{\bs\varpi\}$ are obtained from
$T$ by changing the contents of the boxes of $T$ without breaking
condition of being semi-standard.
Consider the case $k=1$ first. Then, the corresponding semi-standard tableaux are $T_{1,j}, 1\le j\le n+1$, given as follows:
$$\renewcommand{\arraycolsep}{0pt}
\begin{array}{|c|}
   \hline\hbox to 0.55cm{\hfill$\scriptstyle 1$\hfill}\\
   \hline\hbox to 0.55cm{\hfill$\vdots$\hfill}\\
   \hline\hbox to 0.55cm{\hfill$\scriptstyle{n\!-\!2}$\hfill}\\
   \hline\hbox to 0.55cm{\hfill$\scriptstyle{n\!-\!1}$\hfill}\\
   \hline\hbox to 0.55cm{\hfill$\scriptstyle{n}$\hfill}\\
   \hline
\end{array}_{\:r+1-n}\!=\bs\omega^{T_{1,n+1}}\quad \xrightarrow{n,r+1} \quad \renewcommand{\arraycolsep}{0pt}
\begin{array}{|c|}
   \hline\hbox to 0.55cm{\hfill$\scriptstyle 1$\hfill}\\
   \hline\hbox to 0.55cm{\hfill$\vdots$\hfill}\\
   \hline\hbox to 0.55cm{\hfill$\scriptstyle{n\!-\!2}$\hfill}\\
   \hline\hbox to 0.55cm{\hfill$\scriptstyle{n\!-\!1}$\hfill}\\
   \hline\hbox to 0.55cm{\hfill$\scriptstyle{n\!+\!1}$\hfill}\\
   \hline
\end{array}_{\:r+1-n}\!=\bs\omega^{T_{1,n}} \quad \xrightarrow{n-1,r+2}
\quad
\renewcommand{\arraycolsep}{0pt}
\begin{array}{|c|}
   \hline\hbox to 0.55cm{\hfill$\scriptstyle 1$\hfill}\\
   \hline\hbox to 0.55cm{\hfill$\vdots$\hfill}\\
   \hline\hbox to 0.55cm{\hfill$\scriptstyle{n\!-\!2}$\hfill}\\
   \hline\hbox to 0.55cm{\hfill$\scriptstyle{n}$\hfill}\\
   \hline\hbox to 0.55cm{\hfill$\scriptstyle{n\!+\!1}$\hfill}\\
   \hline
\end{array}_{\:r+1-n}\!=\bs\omega^{T_{1,n-1}}$$
$$ \xrightarrow{n-2,r+3}\quad \cdots \quad \xrightarrow{2,r+n-1} \quad
\renewcommand{\arraycolsep}{0pt}
\begin{array}{|c|}
   \hline\hbox to 0.55cm{\hfill$\scriptstyle 1$\hfill}\\
   \hline\hbox to 0.55cm{\hfill$\scriptstyle{3}$\hfill}\\
   \hline\hbox to 0.55cm{\hfill$\scriptstyle{4}$\hfill}\\
   \hline\hbox to 0.55cm{\hfill$\vdots$\hfill}\\
   \hline\hbox to 0.55cm{\hfill$\scriptstyle{n\!+\!1}$\hfill}\\
   \hline
\end{array}_{\:r+1-n}\!=\bs\omega^{T_{1,2}} \quad \xrightarrow{1,r+n} \quad \renewcommand{\arraycolsep}{0pt}
\begin{array}{|c|}
   \hline\hbox to 0.55cm{\hfill$\scriptstyle 2$\hfill}\\
   \hline\hbox to 0.55cm{\hfill$\scriptstyle{3}$\hfill}\\
   \hline\hbox to 0.55cm{\hfill$\scriptstyle{4}$\hfill}\\
   \hline\hbox to 0.55cm{\hfill$\vdots$\hfill}\\
   \hline\hbox to 0.55cm{\hfill$\scriptstyle{n\!+\!1}$\hfill}\\
   \hline
\end{array}_{\:r+1-n}\!=\bs\omega^{T_{1,1}}.$$
Using \eqref{e:rem action tableau A}, one checks that the pair
$(j,r+j)$ written over the arrows correspond to the multiplication
by $A_{j,r+j}^{-1}$. Therefore,
\begin{equation*}
\qch(V) = \sum_{p=1}^{n+1}\bs\omega^{T_{1,p}}= \bs\varpi\left(1+\sum_{p=1}^{n} A^{-1}_{n,p,r}\right).
\end{equation*}

For $k>1$, we first notice that we can do the
same sequence of changes on the first column. Suppose we have done $j$ changes on the first column. Then
we can do the same type of changes on the second column up to the
$j$-th change and so on. In other words, the $\ell$-weights of
$V$ are parameterized by the set of partitions $J(n,k)
= \{\bs j = (j_1,j_2,\dots,j_k): 0\le j_k\le \cdots\le j_2\le j_1\le
n\}$ and the $\ell$-weight associated  to $\bs j\in J$ is
\begin{equation}\label{e:l-weights KRnroot}
\bs\varpi_{\bs j} = \bs\varpi \prod_{l=1}^k A_{n,n+1-j_l,r+2(k-l)}^{-1},
\end{equation}
where we use the convention that $A_{n,n+1,s}=1$ for all $s$\footnote{Note that, for $j_l\ne 0, A_{n,n+1-j_l,r+2(k-l)}$ is given by the second definition in \eqref{e:A3index}. Thus, the convention here, used when $j_l=0$, comes from the usual convention for products applied to the second definition in \eqref{e:A3index}.}.
In terms of fundamental $\ell$-weights, we have
\begin{equation}\label{e:ex l-weights KR n}
\bs\varpi_{\bs j} = Y_{n,r,k} \prod_{l=1}^k
\left(Y^{-1}_{n,r+2(k-l)}\
Y^{-1}_{n-j_l+1,r+2(k-l)+j_l+1}\ Y_{n-j_l,r+2(k-l)+j_l}\right)^{1-\delta_{0,j_l}}.
\end{equation}
\endd
\end{ex}

\subsection{On Irreducible Tensor Products of Minimal Affinizations}

We take a short pause in the study of $\ell$-weights via tableaux to prove a proposition which, in particular, implies the existence of the number $s$ in the statement of Theorem \ref{t:main}. Let $\bs\omega,\bs\varpi\in\mathcal P^+$ correspond to minimal affinizations. Then, by Theorem \ref{teo tipo A}, there exist $a,b\in\mathbb C^\times, r_i,s_i\in\mathbb Z, \lambda,\mu\in P^+$, such that
\begin{equation*}
\bs\omega = \prod_{i\in I} \bs\omega_{i,aq^{r_i},\lambda(h_i)} \qquad\text{and}\qquad
\bs\varpi = \prod_{i\in I} \bs\omega_{i,bq^{s_i},\mu(h_i)}.
\end{equation*}

\begin{prop}\label{p:PZ}
If $V_q(\bs\omega)\otimes V_q(\bs\varpi)$ is reducible, there exists $s\in\mathbb Z$ such that $a/b = q^s$.
\end{prop}

\begin{proof}
Let $V=V_q(\bs\omega)\otimes V_q(\bs\varpi), \bs\lambda=\bs\omega\bs\varpi$, and suppose $a/b\ne q^s$ for all $s\in\mathbb Z$. We claim that $D:=\wt_\ell(V)\cap\mathcal P^+=\{\bs\lambda\}$, which clearly implies the proposition. Indeed, any element of $D$ is of the form $\bs\mu\bs\nu$ with $\bs\mu\in\wt_\ell(V_q(\bs\omega))$ and $\bs\nu\in\wt_\ell(V_q(\bs\varpi))$ by Proposition \ref{p:qchtp}. If $\bs\mu\ne\bs\omega$, then $\bs\mu\notin\mathcal P^+$ by Lemma \ref{l:lmin}. Then, it follows from Theorem \ref{t:qcharmin} that there exists $i\in I, r\in\mathbb Z$, such that $\bs\omega_{i,aq^r}^{-1}$ appears in $\bs\mu$. As $\bs\mu\bs\nu\in\mathcal P^+$, it follows that $\bs\omega_{i,aq^r}$ must appear in $\bs\nu$. However, Theorem \ref{t:qcharmin} also implies that, if $\bs\omega_{i,c}$ appears in $\bs\nu$, then $c=bq^t$ for some $t$, yielding the desired contradiction.
\end{proof}

\begin{rem}\label{r:PZ}
It follows that we can assume from now on that the parameter $b$ in Theorem \ref{t:main} is of the form $aq^s$ for some $s$. In fact, it is well-known (see \eqref{e:tau_a} below) that we can assume without loss of generality that $a=1$ and we shall do so. We also use $a=1$ in \eqref{eq Y A}. \endd
\end{rem}

Let us also recall the description of the simple modules as tensor products of Kirillov-Reshetikhin modules in the case $\lie g=\lie{sl}_2$. Thus, let $i$ be the unique element of $I$. Given $\bs\omega\in\cal P^+$, it is not difficult to see that there
exist unique $m>0$, $a_{j}\in\C^\times$, and $r_{j}\in\Z_{\ge 1}$ such that
$$\bs\omega = \prod_{j=1}^m \bs\omega_{i,a_{j},r_{j}} \quad\text{with}\quad
\frac{a_{j}}{a_{l}}\ne q^{\pm(r_{j}+r_{l}-2p)} \quad\text{for all}\quad j\ne l \quad\text{and}\quad 0\le p<\min\{r_{j},r_{l}\}.$$
This decomposition is called the $q$-factorization of $\bs\omega$. It was proved in \cite[Theorem 4.11]{cp:qaa} that
\begin{equation}\label{teo sl2 irred}
V_q(\bs\omega)\cong V_q(\bs\omega_{i,a_1,r_1})\otimes \cdots \otimes V_q(\bs\omega_{i,a_m,r_m}).
\end{equation}

\subsection{Diagram Subalgebras and Sublattices}\label{ss:diagsub} Some of the next properties of qcharacters and tableaux that we will describe are related to the technique of restricting to diagram subalgebras. In this subsection, we fix the necessary notation.

By abuse of language, we will refer to any subset $J$ of $I$ as a
subdiagram of the Dynkin diagram of $\lie g$. Let $\lie g_J$ be the
Lie subalgebra of $\lie g$ generated by $x_{\alpha_j}^\pm, j\in J$,
and define $\lie n^\pm_J, \lie h_J$ in the obvious way. Let also
$Q_J$ be the subgroup of $Q$ generated by $\alpha_j, j\in J$, and
$R^+_J=R^+\cap Q_J$. Given $\lambda\in P$, $\lambda_J$ is the
restriction of $\lambda$ to $\lie h_J^*$ and let $\lambda^J\in P$ be
such that $\lambda^J(h_j)=\lambda(h_j)$ if $j\in J$ and
$\lambda^J(h_j)=0$ otherwise.  Diagram subalgebras $\tlie g_J$ are
defined in the obvious way.

Consider also the subalgebra $U_q(\tlie g_J)$ generated by $k_j^{\pm
1}, h_{j,r}, x^\pm_{j,s}$ for all $j\in J, r,s\in \Z, r\ne 0$. If
$J=\{j\}$, the algebra $U_q(\tlie g_j):=U_q(\tlie g_J)$ is
isomorphic to $U_{q}(\tlie{sl}_2)$. Similarly we define the
subalgebra $U_q(\lie g_J)$, etc.

For $\bs\omega\in\cal P$, let $\bs\omega_J$ be the associated
$J$-tuple of rational functions and let $\cal
P_J=\{\bs\omega_J:\bs\omega\in \cal P\}$. Similarly define $\cal
P_J^+$. Notice that $\bs\omega_J$ can be regarded as an element of
the $\ell$-weight lattice of $U_q(\tlie g_J)$. Let $\pi_J:\cal P\to
\cal P_J$ denote the map $\bs\omega\mapsto\bs\omega_J$. If $J=\{j\}$
is a singleton, we write $\pi_j$ instead of $\pi_J$. An
$\ell$-weight $\bs\omega\in \cal P$ is said to be $J$-dominant if
$\bs\omega_J\in \cal P_J^+$. Let also $\cal Q_J\subset\cal P_J$ (respectively, $\cal Q_J^+$) be
the subgroup (submonoid) generated by $\pi_J(\bs\alpha_{j,a}), j\in
J,a\in\mathbb C^\times$. When no confusion arises, we shall simply
write $\bs\alpha_{j,a}$ for its image in $\cal P_J$ under $\pi_J$.
Let
$$\iota_J:\Z[\cal Q_J]\rightarrow \Z[\cal Q],$$
be the ring homomorphism such that
$\iota_J(\bs\alpha_{j,a})=\bs\alpha_{j,a}$ for all $j\in J,
a\in\mathbb C^\times$. We shall often abuse of notation and identify
$\cal Q_J$ with its image under $\iota_J$. In particular, given
$\bs\mu\in\cal P$, we set
\begin{equation*}
\bs\mu\cal Q_J =\{ \bs\mu\bs\alpha: \bs\alpha\in\iota_J(\cal Q_J)\}.
\end{equation*}
It will also be useful to introduce the element $\bs\omega^J\in\cal
P$ defined by
$$(\bs\omega^J)_j(u)=\bs\omega_j(u) \quad\text{if}\quad j\in J \quad\text{and}\quad (\bs\omega^J)_j(u)=1 \quad\text{otherwise.}$$

If $\bs\omega\in \cal P$ is $J$-dominant for some subdiagram $J$, set
$$\chi_J(\bs\omega)=\bs\omega\cdot \iota_J(\bs\omega_J^{-1}\qch(V_q(\bs\omega_J))).$$

\begin{prop}\label{prop appears}\cite[Corollary 3.15]{her:small}
Let $J\subset I$, $\bs\omega\in \cal P^+$ and suppose $\bs\mu\in
\cal P$ satisfies:
\begin{enumerate}
\item $\bs\mu\in \wt_\ell(V_q(\bs\omega))$,
\item $\bs\mu\in \cal P_J^+$,
\item there is no $J$-dominant $\bs\varpi > \bs\mu$ satisfying $\bs\varpi\in\wt_\ell(V_q(\bs\omega))$ and
$\bs\mu \in \wt_\ell(\chi_J (\bs\varpi))$.
\end{enumerate}
Then $\wt_\ell(\chi_J (\bs\mu))\subseteq \wt_\ell(V_q(\bs\omega))$.\endd
\end{prop}

\begin{rem}\label{rem appears}
Notice that taking $\bs\mu = \bs\omega$ in Proposition \ref{prop appears},
it follows that $\wt_\ell(\chi_J (\bs\omega))\subseteq
\wt_\ell(V_q(\bs\omega))$.\endd
\end{rem}

\subsection{Further Combinatorial Properties of Tableaux}\label{ss:combtab}
We now collect several technical lemmas on the combinatorics of tableaux.

Suppose $T$ is an increasing column tableau. We say that $T$ has a gap at the $j$-th box if
$$T(j)_1-T(j-1)_1>1.$$
The number $T(j)_1-T(j-1)_1-1$ will be referred to as the size of
the gap. In particular, for $j=1$, $T$ has a gap of size $i_1-1$ at
the first row iff $T(1)_1= i_1>1$. Notice also that, if the length of $T$ is $n$, then it has at most of one gap, necessarily of size 1.

\begin{lem}\label{lem gap}
Let $T$ be a column increasing tableau of shape $(k,s)$ with a gap.
More precisely, suppose $T(j)_1=l_1$ and $T(j+1)_1=l_2$ with $1\le
l_1<l_2-1\le n$, for some $j\in \{1,\ldots,k-1\}$. Then
$Y_{l_1,s+2(k-j)+l_1-1}$ and $Y^{-1}_{l_2-1,s+2(k-(j+1))+l_2}$
appear in $\bs\omega^T$.  Moreover, if $j=k-1$, then
$Y_{l_2,s+l_2-1}$ also appears in $\bs\omega^T$.
\end{lem}

\begin{proof}
By hypothesis, $T$ contains the boxes
$$\ffbox{l_1}_{\:s+2(k-j)}=Y^{-1}_{l_1-1,s+2(k-j)+l_1}Y_{l_1,s+2(k-j)+l_1-1}$$
and
$$\ffbox{l_2}_{\:s+2(k-j)-2}=Y^{-1}_{l_2-1,s+2(k-j)+l_2-2}Y_{l_2,s+2(k-j)+l_1-3}.$$
Since $l_1<l_2-1$, the negative power produced by
$\ffbox{l_2}_{\:s+2(k-j)-2}$ cannot be canceled with the positive
power produced by $\ffbox{l_1}_{\:s+2(k-j)}$ (the box immediately
above it). Also, since $T$ is increasing, $T(j')_1<l_1$ for all
$j'<j$ and $T(j'')_1>l_2$ for all $j''>j+1$. Thus, there is no other
possibility for canceling $Y^{-1}_{l_2-1,s+2(k-j)+l_2-2}$ implying
that $Y^{-1}_{l_2-1,s+2(k-j)+l_2-2}$ appears in $\bs\omega^T$. The
proof that $Y_{l_1,s+2(k-j)+l_1-1}$ appears in $\bs\omega^T$ is
similar. The last statement is also proved in the same manner.
\end{proof}

\begin{lem}\label{lem i-1 above i}
Let $T$ be a semi-standard tableau with shape as in \eqref{rem
tableau stairs} and $(i,s)\in\bf B\times\mathbb Z$. Suppose the box
$\ffbox{i}_{\:s}$ is part of the the $j$-th column of $T$. Then:
\begin{enumerate}[(a)]
\item The box $\ffbox{i}_{\:s}$ is not in any other column of $T$.
\item If $\ffbox{i-1}_{\:s+2}$ is a box in $T$, it must be in the $j$-th column.
\item if $\ffbox{i+1}_{\:s-2}$ is a box in $T$, it must be in the $j$-th column.
\end{enumerate}
\end{lem}

\begin{proof}
 We write down the proof of (b) only since the other items
are similar. Suppose
 $\ffbox{i-1}_{\:s+2}$ appears in the $(j+m)$-th
column, $m\ge 1$. Since $T'$ is as in \eqref{rem tableau stairs},
this column has a box supported at $s-2m$. Since $T'$ is columns
increasing, the content $c$ of the box supported at $s-2m$ is at
least $i+m>i$. This contradicts the assumption that $T'$ is
semi-standard because the box $\ffbox{i}_{\:s}$ in column $j$ and
the box $\ffbox{c}_{\:s-2m}$ in column $j+m$ are in the same
diagonal from left to right and top to bottom. Suppose now that
$\ffbox{i-1}_{\:s+2}$ is in the $(j-m)$-th column, $m\ge 1$. This
time \eqref{rem tableau stairs} implies that this column has a box
supported at $s+2m$. Since all columns are increasing, the content
$c$ of the box supported at $s+2m$ is at most $i-m<i$. This
contradicts the assumption that $T'$ is semi-standard because the
box $\ffbox{c}_{\:s+2m}$ in column $j-m$ and the box
$\ffbox{i}_{\:s}$ in column $j$ are in the same diagonal from left
to right and top to bottom.
\end{proof}

\begin{rem}\label{r:gapsdontcancel}
Note that Lemma \ref{lem i-1 above i} implies that the contributions to $\bs\omega^T$ coming from each gap of a given column of $T$ as described in Lemma \ref{lem gap} are not canceled by terms in other columns.\endd
\end{rem}

The next lemma can be easily proved combinatorially and it is also a
consequence of Theorem \ref{t:qcharmin} together with the fact that
the Frenkel-Mukhin algorithm applies for computing the qcharacters
of minimal affinizations (see \cite{her:min} and references
therein).

\begin{lem}\label{l:connectedquiver}
Let $T$ be a semi-standard tableau such that $V_q(\bs\omega^T)$ is a minimal affinization. Then, for any $T'\in\stab(T)$, there exists $m\ge 0, i_j\in I, s_j\in\mathbb Z$, and elements $T_j\in\stab(T), 0\le j\le m$, such $T_0=T, T_m=T'$ and $\bs\omega^{T_{j+1}} = \bs\omega^{T_{j}}A_{i_j,s_j}^{-1}$.\endd
\end{lem}

\subsection{Right Negativity}\label{ss:qcharprop}
Let $\cal P_\mathbb Z$ denote the subgroup of $\cal P$ generated by $Y_{i,s}, i\in I,s\in\mathbb Z$, and we similarly
define the subgroup $\cal Q_\mathbb Z$ of $\cal Q$ and the monoids
$\cal P_\mathbb Z^+$ and $\cal Q_\mathbb Z^+$.
The following concept defined in \cite{fremuk:qchar} will be useful in the proof of Theorem \ref{t:main}. Given $\bs\omega\in\cal P_\mathbb Z\setminus\{\bs 1\}$, set
\begin{equation}\label{eq def r(omega)}
r(\bs\omega):=\max\{s\in \Z : Y_{i,s}^{\pm 1} \text{ appears in }
\bs\omega \text{ for some } i\in I\}.
\end{equation}
Then, $\bs\omega$ is said to be right negative if $Y_{i,r(\bs\omega)}$ does not appear in
$\bs\omega$ for all $i\in I$.
Observe that the product of right negative $\ell$-weights is a right
negative $\ell$-weight and a dominant $\ell$-weight is not right
negative. Observe also that
\begin{equation}
r(Y_{i,r,k}) = r+2(k-1) \qquad\text{for all}\qquad i\in I,r\in \Z,k\in \Z_{\ge 0}.
\end{equation}

\begin{ex}
Return to Example \ref{ex l-weights KR n} and recall \eqref{e:ex l-weights KR n}.
Notice that, if $j_l>0$ for some $l$, there exists a gap of size $1$ at the
$(n-j_l+1)$-th row of the $l$-th column of the associated semi-standard tableau. Moreover, if
$\bs\varpi_{\bs j} \ne \bs\varpi$, then $j_1>0$ which implies
$Y^{-1}_{n+1-j_1,r+2(k-1)+j_1+1} Y_{n-j_1,r+2(k-1)+j_1}$ appears in
$\bs\varpi_{\bs j}$ and
\begin{equation}\label{eq r(varpi_j)}
r(\bs\varpi_{\bs j})=r+2(k-1)+j_1+1.
\end{equation}
Indeed, since $l_1<l_2$ implies $j_{l_1}\ge j_{l_2}$, we have
$r+2(k-l_1)+j_{l_1}>r+2(k-l_2)+j_{l_2}$.\endd
\end{ex}

The following was proved in \cite[Theorem 3.2]{nak:t-analogs}.

\begin{prop}\label{prop right neg}
Let $\bs\omega=Y_{i,r,k}$ for some $i\in I$, $r\in \Z$, and $k\in \Z_{\ge 0}$. Then, all the elements
of $\wt_\ell(V_q(\bs\omega))\setminus\{\bs\omega\}$ are right negative. Moreover, if
$\bs\mu\in\wt_\ell(V_q(\bs\omega))\setminus\{\bs\omega\}$ is such
that $r(\bs\mu)\le r+2k$, then
$$\bs\mu=Y_{i,r,s} Y_{i,r+2(s+1),k-s}^{-1}\prod_{j\::\:c_{ij}=-1}Y_{j,r+2s+1,k-s}
\qquad \text{for some} \quad s=0,\ldots,k-1.$$ In particular,
$r(\bs\mu)=r+2k>r(\bs\omega)$.\endd
\end{prop}

\section{Proof of the Main Theorem}\label{s:proof}

\subsection{The Scheme of the Proof}\label{ss:scheme}

Fix the notation of Theorem \ref{t:main}. The scheme of the proof is as follows. First, using the combinatorics of semi-standard tableaux studied in Section \ref{ss:combtab}, we describe the set
$$D:=\wt_\ell(V)\cap\mathcal P^+$$
and, as a byproduct, we obtain the following proposition.
\begin{prop}\label{p:multiplicity1A}
The partial order on $\mathcal P$ induces a total order on $D$ and
\begin{equation*}
\dim(V_{\bs\nu})=1 \qquad\text{for all}\qquad \bs\nu\in D.
\end{equation*}
\endd
\end{prop}

This will be done in Sections \ref{ss:Jdom} and \ref{ss:domtpA}. Moving on, we shall see that, if neither of the conditions (i) and (ii) of  Theorem \ref{t:main} are satisfied, then  $\sum_{\bs\nu\in D}\bs\nu$ is part of $\qch(V_q(\bs\lambda))$ and, hence, $V$ is irreducible. Otherwise, we will see that there exists a decomposition of $D$ as a disjoint union
\begin{equation*}
D = D_1\cup D_2
\end{equation*}
such that
\begin{enumerate}[(a)]
\item $\sum_{\bs\nu\in D_1}\bs\nu$ is part of $\qch(V_q(\bs\lambda))$;
\item $\bs\lambda'\in D_2$ and $\bs\lambda'\notin \wt_\ell(V_q(\bs\lambda))$;
\item $\sum_{\bs\nu\in D_2}\bs\nu$ is part of $\qch(V_q(\bs\lambda'))$.
\end{enumerate}
This clearly implies Theorem \ref{t:main}. For proving these
properties of $D_1$ and $D_2$ we will use the results on qcharacters
that were reviewed in Section \ref{ss:qcharprop}.  The
aforementioned properties of $D$ will follow from Lemmas
\ref{l:tpfactorsA<} and \ref{lem lweight in A} below.

Throughout this section we let $S$ be the semi-standard tableau such
that $\bs\omega^S=\bs\omega$ and $T$ be the one such that
$\bs\omega^T=\bs\varpi$. Recall also from Remark \ref{r:PZ} that we
have set  $a=a_{i_0}=1$. We shall work with the expression \eqref{eq
l-weight min aff} for $\bs\omega$ and, similarly, we let
$r\in\mathbb Z$ be such that $\bs\varpi=Y_{n,r,k}$, i.e., $r=s-k+1$
where $s$ is the number in Theorem \ref{t:main} (cf. Proposition
\ref{p:PZ}).  We now rephrase Theorem \ref{t:main} in the context of
tableaux and in terms of the numbers $r_i$ and $r$ since this is the
way we will work in the proof.

\begin{thm}\label{t:mainr}
$V$ has length at most 2 and is reducible if and only if there exist $p\in\supp(\lambda)$ and
$k'>0$ such that either one of the following options hold:
\begin{enumerate}[(i)]
\item $k'\le\min\{\lambda(h_p),k\}$ and $r+2k+n-p+2 = r_p+2k'$;
\item $k'\le\min\{|\lambda|,k\}$ and  $r_{i_0}+2\lambda(h_{i_0})+n-i_0+2 = r+2k'$, and $p=\max\{i\in I:\ _i|\lambda|\ge k'\}$.
\end{enumerate}
If (i) holds,  $\bs\lambda'=\bs\omega\bs\omega^{T'}$ where $T'$ is the only element of $\stab(T)$ whose gaps are all located at the $p$-th boxes of the first $k'$ columns. If (ii) holds, $\bs\lambda'=\bs\omega^{S'}\bs\varpi$ where $S'$ is obtained from $S$ by replacing the contents of last boxes of the first $k'$ columns by $n+1$.\endd
\end{thm}

\begin{rem}
The formulas for $\bs\lambda'$ as given in Theorem \ref{t:main}, or
rather in terms of the elements defined in \eqref{eq Y A} and
\eqref{e:A3index}, will be obtained as we develop the proof of
Theorem \ref{t:mainr}. Since we will work with this rephrasing of
Theorem \ref{t:main}, we no longer maintain $s$ and $s_i$ fixed as
in Theorem \ref{t:main} and allow ourselves to use these symbols for
additional local parameters in the several steps of the
proof.\endd
\end{rem}

\subsection{Fail of Dominance at $n$ for Increasing Minimal Affinizations}\label{ss:Jdom}
As the first step towards describing the set $D$, we describe the elements of $\wt_\ell(V_q(\bs\omega))$ which are $J$-dominant where
\begin{equation*}
J = I\setminus\{n\}.
\end{equation*}
Recall the definition of the map $\pi_J$ in Section \ref{ss:diagsub} and the definition of right negative elements in Section \ref{ss:qcharprop}.

\begin{lem}\label{l:nrnA}
If $\bs\nu\in\wt_\ell(V_q(\bs\omega))$ is not right negative, then
$r(\bs\nu)=r_{i_0}+2(\lambda(h_{i_0})-1)$ and
$\pi_{i_0}(\bs\nu)=Y_{i_0,r_{i_0},\lambda(h_{i_0})}$.
\end{lem}

\begin{proof}
The statement is clear if $\bs\nu = \bs\omega$. Suppose
$\bs\nu\ne\bs\omega$ and write $\bs\nu=\bs\omega^{S'}$ with
$S'\in\stab(S)$. Since $S'\ne S$, there exists $1\le l\le|\lambda|$
such that the $l$-th column $S'_l$ of $S'$ has a gap. Assume $l$ is
the smallest such index. It follows from Lemma \ref{lem gap} that
$Y_{i,r(\bs\omega^{S_l'})}^{-1}$ appears in $\bs\omega^{S'_l}$ for
some $i\in I$. By Lemma \ref{lem i-1 above i},
$Y_{i,r(\bs\omega^{S_l'})}$ does not appear in
$\bs\omega^{S'\setminus S'_l}$, where $S'\setminus S'_l$ is the
tableaux obtained from $S'$ by removing its $l$-th column. Hence,
$Y_{i,r(\bs\omega^{S_l'})}^{-1}$ appears in $\bs\nu$. Since $\bs\nu$
is not right negative, there must exist $1\le l'\le|\lambda|$ such
that $S'_{l'}$ is gap-free and
$r(\bs\omega^{S_{l'}'})>r(\bs\omega^{S_l'})$. One easily checks
that, if $l'>l$ and $S'_{l'}$ has no gaps, then
$r(\bs\omega^{S_{l'}'})<r(\bs\omega^{S_l'})$. Therefore, we must
have $1\le l'<l$ and, since both  $S'_1$ and $S'_{l'}$ are gap-free,
it follows that $r(\bs\omega^{S_{1}'})>r(\bs\omega^{S_{l'}'})$ which
proves the first statement of the lemma. Since $S'$ is semi-standard
and the first column is gap-free, all columns of length $i_0$ must
also be gap-fee which implies the second statement.
\end{proof}

Henceforth, we denote by $l_j$  the length of the $j$-th column of
$S.$ Given $1\le c\le f\le |\lambda|$ and  $0< p\le n+1$, consider
the tableau $S_{c,f,p}$ which is the unique tableau in $\stab(S)$ satisfying:
\begin{enumerate}[(1)]
\item each column has at most one gap, necessarily at its last box;
\item if $j<c$ or $j>f$, the $j$-th column of $S_{c,f,p}$ does not have a gap;
\item  if $c\le j\le f$, the content of the last box of the $j$-th column is $\max\{l_j,p\}$.
\end{enumerate}
For notational convenience, we set
$S_{c,f,p}=S_{c,|\lambda|,p}$ for $f>|\lambda|$ and $S_{c,f,p}=S$
for $f<c$ as well as for $c>|\lambda|$ and for $p=0$.  Note that the
support of the last box of the $c$-th column of $S$ is
\begin{equation*}
s_c = r_{l_c} + 2(\: _{l_c}|\lambda|-c\:)-l_c+1
\end{equation*}
Then,  Lemma \ref{lem gap} implies that
\begin{align}
\notag
\bs\omega^{S_{c,f,p}} = &
\left( \prod_{j=0}^{f-c} Y_{l_{c+j}-1,s_c-2j+l_{c+j}}\ Y_{p-1,s_c+p-2j}^{-1}\  Y_{p,s_c+p-2j-1} \right)
\left(\prod_{i= l_c+1}^{n} Y_{i,r_i,\lambda(h_i)}\right)
\\ \label{e:omegaScfpgen} \hfill\\\notag &\times
\left(\prod_{i=1}^{l_f-1} Y_{i,r_i,\lambda(h_i)}\right)
\left( \prod_{\substack{j<c:\\ l_j=l_c}} Y_{l_c,r_{l_c}+2(\: _{l_c}|\lambda|-j\:)} \right)
\left( \prod_{\substack{j>f:\\ l_j=l_f}} Y_{l_f,r_{l_f}+2(\: _{l_f}|\lambda|-j\:)} \right).
\end{align}
The terms in the first parenthesis of \eqref{e:omegaScfpgen} are the
ones corresponding to the columns of $S_{c,f,p}$ which are not equal
to those of $S$, while the remaining terms come from the columns
which were not modified. By Remark \ref{r:gapsdontcancel}, there are
no cancelations in \eqref{e:omegaScfpgen}.
Here is the main result of this subsection.

\begin{prop}\label{prop 1,..,n-1-dom}
The $J$-dominant elements of
$\wt_\ell(V_q(\bs\omega))$ are $\bs\omega^{S_{1,f,n+1}}$, $0\le f\le |\lambda|$.
\end{prop}

\begin{proof}
It is clear from  \eqref{e:omegaScfpgen} that
$\bs\omega^{S_{1,f,n+1}}$ is $J$-dominant for all $0\le f\le
|\lambda|$. For the converse, by Theorem \ref{t:qcharmin}, any element of
$\wt_\ell(V_q(\bs\omega))$ can be represented by an element of
$\stab(S)$. Let $S'\in\stab(S)$ and suppose its $l$-th column has a gap in a box whose content is $i$. By Lemma \ref{lem gap} and Remark \ref{r:gapsdontcancel}, it follows that $Y^{-1}_{i-1,r}$ appears in $\bs{\omega}^{S'}$ and, hence, $i=n+1$ and the gap must be in the last box of the column. 
This immediately implies that $S'= S_{1,f,n+1}$  where $f$ is the number of the last column having $n+1$ as the content of its last box.
\end{proof}

We shall need some extra information about the elements
$\bs\omega^{S_{c,f,p}}$. Given $1\le j\le |\lambda|$, set
\begin{equation}\label{e:jisdjoflj}
d_j= j-\ _{l_j+1}|\lambda|.
\end{equation}
Thus, the $j$-th column is the $d_j$-th one of length $l_j$.
Notice that, by definition of $S_{c,f,p}$ and \eqref{e:rem action tableau A}, we have
\begin{equation}\label{e:Scfpgodown}
\bs\omega^{S_{c,f,p}} = \bs\omega^{S_{c,f-1,p}}\
A^{-1}_{l_f,p-1,r_{l_f}+2(\lambda(h_{l_f})-d_f)}
\end{equation}
for all $1\le c\le f\le |\lambda|, l_c< p\le n+1$. Iterating this,
we get
\begin{equation}\label{e:ScfpfromSbyAs}
\bs\omega^{S_{c,f,p}}=\bs\omega \left(\prod_{j=c}^{f}A^{-1}_{l_j,p-1,r_{l_j}+2(\lambda(h_{l_j})-d_j)}\right).
\end{equation}

\begin{rem}
Note that the element $S'$ in Theorem \ref{t:mainr} is $S_{1,k',n+1}$. The formula for $\bs\lambda'$ in case (ii) of Theorem \ref{t:main} is then easily deduced from \eqref{e:ScfpfromSbyAs}.\endd
\end{rem}

The following formulas are easily obtained from \eqref{e:omegaScfpgen} (note that, if $l_f=n$, then $i_0=n$).
\begin{equation}\label{eq S1,j,n+1}
\pi_n(\bs\omega^{S_{1,f,n+1}}) =
\begin{cases}
Y^{-1}_{n,r_{i_0}+2(\lambda(h_{i_0})-f+1)+n-i_0,f}, & \text{ if } l_f<n,\\
Y_{n, r_n,\lambda(h_n)-f}\ Y^{-1}_{n,r_n+2(\lambda(h_n)-f+1),f}, &
\text{ if } l_f=n.
\end{cases}
\end{equation}
Let $\bs\eta$ be so that \eqref{e:ScfpfromSbyAs} reads $\bs\omega^{S_{c,f,p}}=\bs\omega\bs\eta$. A careful inspection of \eqref{e:ScfpfromSbyAs} and \eqref{e:A3index} shows that $\bs\eta$ can be written in the form
\begin{equation}\label{e:Scfpeta}
\bs\eta= \prod_{\xi\in\Xi}A_\xi^{-1} \quad\text{for some} \quad \Xi=\Xi_{c,f,p}\subseteq I\times \Z.
\end{equation}
Moreover, \eqref{e:lrootsind} implies that such $\Xi$ is unique. Notice also that, if
$$s=\min\{t:(i,t)\in\Xi\text{ for some } i\}$$
(such minimum is afforded by the pair $(l_f,r_{l_f}+2(\lambda(h_{l_f})-d_f)+1)$),
then $s-l_f$ is the support of the last box of the $f$-th column of $S$. Moreover, if $c=\min\{j:l_j<p\}$, then
\begin{equation}\label{e:xi_0}
\max\{i:(i,t)\in \Xi \text{ for some } t\}=p-1, 
\end{equation}
\begin{equation*}
(p-1,s+p-l_f-1)\in\Xi, \qquad\text{and}\qquad (i,s+p-l_f-1)\notin\Xi \quad\text{for}\quad i\ne p-1.
\end{equation*}

\subsection{Dominant $\ell$-Weights in $V$}\label{ss:domtpA}
We now give the description of $D$ and prove Proposition \ref{p:multiplicity1A}. As in Section \ref{ss:Jdom}, we set $J=I\setminus\{n\}$. We will consider separately the following two subcases
\begin{equation}\label{eq ri<rn}
r_{i_0}+2\lambda(h_{i_0})\le r+2k
\end{equation}
and
\begin{equation}\label{eq rn<ri}
r+2k \le r_{i_0}+2\lambda(h_{i_0}).
\end{equation}
Assume first that \eqref{eq ri<rn} holds.

\begin{lem}\label{lema dominant product}
The elements of $D$ are of the form $\bs\nu\bs\varpi$ with
$\bs\nu\in\wt_\ell(V_q(\bs\omega))$.
\end{lem}

\begin{proof}
Let $\bs\nu\in\wt_\ell(V_q(\bs\omega))$ and
$\bs\mu\in\wt_\ell(V_q(\bs\varpi))$ be such that
$\bs\nu\bs\mu\in\cal P^+$. In particular, $\bs\nu\bs\mu$ is not
right negative. Suppose by contradiction that $\bs\mu\ne \bs\varpi$.
Then, by Proposition \ref{prop right neg}, $\bs\mu$ is right
negative and it follows from \eqref{eq r(varpi_j)} that
$$r(\bs\mu)=r+2(k-1)+j+1$$
for some $1\le j\le n$. Since the product of right negative elements
is again right negative, $\bs\nu$ is not right negative and Lemma
\ref{l:nrnA} implies that
\begin{equation*}
r(\bs\nu)=r_{i_0}+2(\lambda(h_{i_0})-1).
\end{equation*}
Together with \eqref{eq ri<rn}, this implies that
$r(\bs\nu)<r(\bs\mu)$. It then follows that $\bs\nu\bs\mu$ is right
negative, yielding the desired contradiction.
\end{proof}

Note that, since $V_q(\bs\omega)$ is thin by Theorem \ref{t:qcharmin}, Lemma
\ref{lema dominant product} and Proposition \ref{p:qchtp} imply the second statement of
Proposition \ref{p:multiplicity1A} in the present case. We shall now
prove that
\begin{equation}\label{eq D}
D=\{\bs\omega^{S_{1,f,n+1}}\bs\varpi:f=0,1,\ldots,k'\}
\end{equation}
where $0\le k'\le k$ is either zero or given by the following relation (cf. condition (ii) in Theorem \ref{t:mainr}):
\begin{equation}\label{e:k'i0}
r_{i_0}+2\lambda(h_{i_0})+n-i_0+2 = r+2k'.
\end{equation}
Indeed, since the elements in $D$ are of the form $\bs\nu\bs\varpi$
with $\bs\nu\in \wt_{\ell}(V_q(\bs\omega))$, it follows that
$\bs\nu$ must be $J$-dominant and, hence, by
Proposition \ref{prop 1,..,n-1-dom}, we must have
$\bs\nu=\bs\omega^{S_{1,f,n+1}}$ for some $0\le f\le |\lambda|$. It
now easily follows from \eqref{eq S1,j,n+1} that $\bs\omega^{S_{1,f,n+1}}\bs\varpi \in\mathcal
P^+$ if and only if $0\le f\le k'\le k$.

Notice that the first statement of Proposition \ref{p:multiplicity1A} follows easily from
\eqref{eq D} and \eqref{e:Scfpgodown}, which completes the proof of Proposition \ref{p:multiplicity1A} when \eqref{eq ri<rn} holds.  Before moving on, notice that equality in \eqref{eq ri<rn} implies that there is no $k'>0$ satisfying \eqref{e:k'i0}. Indeed, if such $k'$ existed, we would have
\begin{equation}\label{e:inoii}
2k' = r_{i_0}+2\lambda(h_{i_0})+n-i_0+2 - r  \stackrel{\eqref{eq ri<rn}}{=} n-i_0+2+2k > 2k
\end{equation}
which contradicts $k'\le k$. This shows that conditions (i) and (ii) of Theorem \ref{t:main} cannot be simultaneously satisfied.

From now on till the end of this subsection, assume \eqref{eq rn<ri} holds.

\begin{lem}\label{lem no dom of form}
If $\bs\mu\in \wt_\ell(V_q(\bs\omega))$ is such that
$\bs\mu\bs\varpi\in D$, then $\bs\mu=\bs\omega$.
\end{lem}

\begin{proof}
Obviously, if $\bs\mu\bs\varpi\in D$, $\bs\mu$ must be $J$-dominant.
Let $s=r_{i_0}+2\lambda(h_{i_0})+n-i_0$. By \eqref{eq S1,j,n+1}, $Y^{-1}_{n,s}$ appears in all $J$-dominant
$\ell$-weights of $\wt_\ell(V_q(\bs\omega))$ except $\bs\omega$. We
claim that, if $\bs\mu\ne\bs\omega$, then $Y^{-1}_{n,s}$ appears in
$\bs\mu\bs\varpi$, which proves the lemma. By definition of $\bs\varpi$, $Y_{n,r'}$
appears in $\bs\varpi$ iff $r' = r+2(j-1)$ for some $j=1,\dots,k$.
Since
$$r+2(k-1)\stackrel{\eqref{eq rn<ri}}{\le} r_{i_0}+2(\lambda(h_{i_0})-1)<s,$$
the claim follows.
\end{proof}

Suppose there exists $p\in\supp(\lambda)$ and $k'\in\{1,\ldots,\lambda(h_p)\}$ satisfying (cf. condition (i) in Theorem \ref{t:mainr}):
\begin{equation}\label{eq p,k'}
r+2k+n-p+2= r_p+2k'.
\end{equation}
Observe that the pair $(p,k')$ is unique, if it exists. Indeed,
assume $(p,k')$ and $(p',k'')$ satisfy \eqref{eq p,k'}. If $p=p'$ we
must obviously have $k'=k''$. Otherwise, without loss of generality,
assume that $$p-p'<0.$$ To obtain a contradiction, observe that \eqref{e:rifromri0} implies
$$r_{p'}-(r_p+2(\lambda(h_p)-1))\ge p'-p+2$$
(equality holds if $p'=\min\{i\in\supp(\lambda): i>p\}$). It follows that,
\begin{eqnarray*}
p-p'&=&r+2(k-1)+n-p'+2-(r+2(k-1)+n-p+2)\\
&\stackrel{\eqref{eq p,k'}}{=}&r_{p'}+2(k''-1)-(r_p+2(k'-1))\ \ge\ r_{p'}+2(1-1)-(r_p+2(\lambda(h_p)-1))\\
&\ge& p'-p +2\ >\ 0.
\end{eqnarray*}
If a pair $(p,k')$ satisfying \eqref{eq p,k'} does not exist, we set
$p=k'=0$. Henceforth, we assume that either $(p,k')$ satisfies \eqref{eq p,k'} or $p=k'=0$.

If $p\ne 0$, given $0\le m\le k$, consider the tableau $T_{m,p}$ which is the unique tableau in $\stab(T)$ satisfying:
\begin{enumerate}[(1)]
\item if $j>m$, the $j$-th column of $T_{m,p}$ does not have gaps;
\item each of the first $m$ columns of $T_{m,p}$ has a gap;
\item all gaps occur at the $p$-th box of the corresponding column.
\end{enumerate}
In particular, $T_{0,p}=T$ and, for convenience, we set
$T_{m,n+1}=T$ and $T_{m,p}=T_{k,p}$ for $m>k$. In the spirit
of \eqref{e:l-weights KRnroot},
\begin{equation}\label{e:Tmpaspart}
T_{m,p} \quad\text{corresponds to the partition \mbox{$\bs j$} given by}\quad j_i
= \begin{cases} n+1-p, & \text{if } i\le m,\\ 0, & \text{if } i> m.\end{cases}
\end{equation}
Note that, if $p\ne 0$, \eqref{e:rem action tableau A} implies
\begin{equation}\label{e:Ttpgodown}
\bs\omega^{T_{m+1,p}} = \bs\omega^{T_{m,p}}\ A^{-1}_{n,p,r+2(k-m-1)}
\qquad\text{for all}\qquad 0\le m<k.
\end{equation}
Iterating we get
\begin{equation}\label{eq T D in A ii}
\bs\omega^{T_{m,p}}=\bs\varpi\; \prod_{l=1}^{m} A^{-1}_{n,p,r+2(k-l)}
\qquad\text{for all}\qquad 1\le m\le k.
\end{equation}

\begin{rem}\label{r:T'}
Note that the element $T'$ in Theorem \ref{t:mainr} is $T_{k',p}$. The formula for $\bs\lambda'$ in case (i) of Theorem \ref{t:main} is then easily deduced from \eqref{eq T D in A ii}.\endd
\end{rem}

Set
\begin{equation}\label{e:defc}
c = \min\{j: l_j<p\} =1\ +\ _p|\lambda|
\end{equation}
and consider the subset $D'$ of $\wtl(V)$ defined as follows.
\begin{equation*}
\bs\mu\in D'\ \Leftrightarrow\ \bs\mu =
\begin{cases}
\bs\omega\; \bs\omega^{T_{m,p}},& \text{ for } 0\le m\le k'\\
\bs\omega^{S_{c,f,p}}\bs\omega^{T_{m,p}}, & \text{ for } k'\le m\le k, f\le |\lambda|, f=c+m-k'- \epsilon \text{ with } \epsilon\in\{0,1\}.
\end{cases}
\end{equation*}
We will show that
\begin{equation}\label{e:tensorA}
D=D'.
\end{equation}
Together with \eqref{e:Ttpgodown} and \eqref{e:Scfpgodown},
\eqref{e:tensorA} easily implies the first statement of Proposition
\ref{p:multiplicity1A}. Notice that \eqref{e:tensorA} implies that, if $p=k'=0$, then $D=\{\bs\lambda\}$. However, we will
prove that this is true independently  as part of the proof of
\eqref{e:tensorA} (see Proposition \ref{prop without dom}).

We begin the proof of \eqref{e:tensorA} by investigating the
elements in $D$ of the form $\bs\omega\bs\nu$ with
$\bs\nu\in\wtl(V_q(\bs\varpi))$. Recall from Example \ref{ex
l-weights KR n} that the elements of $\wtl(V_q(\bs\varpi))$ are in
bijection with the set $J(n,k) = \{\bs j = (j_1,j_2,\dots,j_k): 0\le
j_k\le \cdots\le j_2\le j_1\le n\}$. Let $\bs j$ and $T'$ be the
tuple and tableaux associated to $\bs\nu$, respectively. In
particular, if $\bs\nu\ne\bs\varpi$, we have $j_1>0$ and Lemma
\ref{lem gap} implies that the first column of $T'$ contributes with
the appearance of the factor $Y_{i_1,r+2k+n-i_1}^{-1}$ in
$\bs\nu$, where
\begin{equation}\label{e:i_1}
i_1:=n-j_1+1.
\end{equation}
This implies that $Y_{i_1,r+2(k-1)+n-i_1+2}$ must appear in $\bs\omega$.
Recalling that
\begin{equation*}
\bs\omega = \prod_{i\in I} Y_{i,r_i,\lambda(h_i)} = \prod_{i\in I}\prod_{l=1}^{\lambda(h_i)} Y_{i,r_i+2(l-1)},
\end{equation*}
it follows that there exists $1\le l\le\lambda(h_{i_1})$ such that
\begin{equation*}
r+2(k-1)+n-i_1+2 = r_{i_1}+2(l-1).
\end{equation*}
In other words, the pair $(i_1,l)$ satisfies \eqref{eq p,k'} and,
hence, $i_1=p$ and $l=k'$. In particular, we have shown the following lemma.

\begin{lem}\label{lem no dom of form 2}
If $p=k'=0$ and $\bs\nu\in \wt_\ell(V_q(\bs\varpi))$ is such that $\bs\omega\bs\nu\in D$, then
$\bs\nu=\bs\varpi$.\endd
\end{lem}

\begin{prop}\label{prop without dom}
If $p=k'=0$, then $D=\{\bs\lambda\}$.
\end{prop}

\begin{proof}
By Lemma \ref{lem no dom of form}, there is no element in $D$ of the
form $\bs\mu\bs\varpi$ with $\bs\mu\in
\wt_\ell(V_q(\bs\omega))\setminus\{\bs\omega\}$. Suppose
$\bs\mu\bs\nu\in D$  with $\bs\mu\in
\wt_\ell(V_q(\bs\omega))$ and $\bs\nu\in
\wt_\ell(V_q(\bs\varpi))\setminus\{\bs\varpi\}$. As before, let $\bs j$ and $T'$ be the
tuple and tableaux associated to $\bs\nu$, respectively, and set $i=i_1$ as defined in \eqref{e:i_1}.
Then, applying Lemma \ref{lem gap} to the first column of $T'$ as before, we see that the terms
\begin{equation}\label{e:l-weight fund n}
Y_{i,s+i-1}^{-1}\quad\text{and}\quad Y_{i-1,s+i-2} \quad\text{appear in}\quad \bs\nu,
\end{equation}
where $$s=r+2(k-i)+n+1.$$
It then suffices to show that the first listed factor in \eqref{e:l-weight fund n} cannot
be canceled by a factor of $\bs\mu$. Let $S'\in\stab(S)$ be such that $\bs\omega^{S'}=\bs\mu$ and observe that, for
canceling that term, $S'$ should have the box  $\ffbox{i}_s$ and, if $\ffbox{i'}_{s-2}$ is also a box in $S'$, then $i'\ne i+1$. Suppose the box $\ffbox{i}_s$  occurred at the $l$-th column of $S'$. Assume first that it were the last box of the column. We have two cases:
\begin{enumerate}
\item the $l$-th column of $S'$ has length $i$;
\item the $l$-th column of $S'$ has length strictly smaller than
$i$.
\end{enumerate}
We will get a contradiction in both cases. By working with the other columns of $T'$ similarly to how we deduced \eqref{e:l-weight fund n}, we see that
\begin{equation}\label{e:prop without dom}
\text{if} \quad Y_{j,s'}\quad \text{appears in}\quad \bs\nu,
\quad\text{then}\quad j\ge i-1\ \text{ and }\ s'\le s+i-2.
\end{equation}

\noindent Case (i). Since $S'$ has the same shape of $S$, the columns of
length $i$ have their last box supported at $r_i+2j-i-1$, for
$j=1,\ldots,\lambda(h_i)$. Thus, there must exist $j$ such that
$$s=r_i+2j-i-1,\quad\text{which implies}\quad r+2k+n-i+2=r_i+2j$$
and, hence, $(i,j)$ satisfies \eqref{eq p,k'}, contradicting the hypothesis of the proposition.

\noindent Case (ii). In this case, the $l$-th column of $S'$ must have a gap.
Suppose that $\ffbox{j}_{\:s'}$  is
a box of the $l$-th column of $S'$ such that $\ffbox{j-1}_{\:s'+2}$
is not a box of this column. Evidently, $j\le i$ and $s'\ge s$. By
Lemma \ref{lem gap} and Remark \ref{r:gapsdontcancel}, $Y_{j-1,s'+j}^{-1}$
appears in $\bs\omega^{S'}$. Since $\bs\mu\bs\nu\in D$,
$Y_{j-1,s'+j}$ must appear in $\bs\nu$ and, therefore, $j=i$ by the first assertion in
\eqref{e:prop without dom}. This implies $s'=s$ and, hence, $s'+j=s+i$. However, the second assertion in \eqref{e:prop without dom} implies that
$s'+j<s+i$, yielding the desired contradiction.

Finally, assume that $\ffbox{i}_{\:s}$ is not the last
box of the $l$-th column of $S'$. This implies that there exists
$j>i+1$ such that $\ffbox{j}_{s-2}$ is a box in this
column. By Lemma \ref{lem gap} and Remark \ref{r:gapsdontcancel},
$Y^{-1}_{j-1,s-2+j}$ appears in $\bs\omega^{S'}$ and,
hence, $Y_{j-1,s-2+j}$ must appear in $\bs\nu$. Since
$j>i+1$, we get $s-2+j> s+i-1$ contradicting the second assertion in \eqref{e:prop without dom} again.
\end{proof}

Now, suppose $(p,k')$ satisfies \eqref{eq p,k'} and set
\begin{equation}\label{e:s}
s=r+2(k-p)+n-1,
\end{equation}
i.e.,
\begin{equation}\label{e:suppk'}
s \text{ is the support of the $p$-th box of the first column of }T.
\end{equation}
Then, \eqref{eq p,k'} implies that
\begin{equation}\label{e:suppofb}
s+2\ \text{ is the support of the the last box of the $b$-th column of } S
\end{equation}
where
\begin{equation}\label{e:b}
b=c-k'.
\end{equation}
Note that, by \eqref{e:defc}, the $b$-th column of $S$ is the $k'$-th one of length $p$ counted from right to left.
Since $S$ has the
form \eqref{rem tableau stairs}, it follows that, for all $1\le m\le
k$, if $s'$ is the support of a box in the $l$-th column $S$, then
\begin{equation}\label{e:columnpos}
s'\le s+2-2(m-1) \quad\Rightarrow\quad  l\ge b+m-1.
\end{equation}
Observe also that, by \eqref{e:suppk'}, $s-2(m-1)$ is the support of the $p$-th
box of the $m$-th column of $T$.

\begin{lem}\label{lem dom t,p}
Every element of $D$ is of the form $\bs\omega^{S'}\bs\omega^{T_{m,p}}$ for some $0\le m\le k$ and $S'\in\stab(S)$.
\end{lem}

\begin{proof}
We will prove that, if $T'\in\stab(T)$, then
\begin{equation}
T'\ne T_{m,p} \ \Rightarrow\ \bs\omega^{T'} \bs\omega^{S'}\notin\cal P^+ \ \text{for all}\  S'\in\stab(S).
\end{equation}
Thus, suppose $T'\ne T_{m,p}$ for all $0\le m\le k$ and, by
contradiction, suppose there exists $S'$ such that $\bs\omega^{T'}
\bs\omega^{S'}\in\cal P^+$. Since $T_{0,p}=T$, we must have $T'\ne
T$ and, hence, the first column of $T'$ must contain a gap
(necessarily of size $1$ since the length of the column is $n$). The hypothesis $T'\ne T_{m,p}$ implies
that $T'$ has a column containing a gap not located at its $p$-th box.
Suppose the $m'$-th column of $T'$ is the first such column and
assume the gap occurs at the $j$-th box. Then, Lemma \ref{lem gap}
and Remark \ref{r:gapsdontcancel} imply that $Y^{-1}_{j,r'}$ appears in $\bs\omega^{T'}$ where
\begin{equation}\label{e:defr'rands}
r'=r+2(k-m'+1)+n-j = s+2(p-m'+1)-j+1
\end{equation}
and, hence, $Y_{j,r'}$ must appear in $\bs\omega^{S'}$. This means that
$\ffbox{j}_{\:r'-j+1}$ must be a box in $S'$.  Since this
box occurs in only one column of $S'$ by Lemma \ref{lem
i-1 above i}, it follows that
\begin{equation}\label{e:notaboxS'}
\ffbox{j+1}_{\:r'-j-1} \ \text{  is not a box in }\ S'
\end{equation}
(otherwise, it would cancel the $Y_{j,r'}$ coming from
$\ffbox{j}_{\:r'-j+1}$).

Suppose $j>p$ and notice that this implies that,
\begin{equation}\label{e:c=p}
\text{if}\ \ \ffbox{i}_{s'}\ \text{ is a box in } T' \text{ with } s'>s, \text{ then it is the } i\text{-th box of its column}.
\end{equation}
Indeed, the condition $s'>s$ implies that this box is among the
first $p-1$ boxes of its column. Recall that, since all columns of
$T$ have length $n$, each column of  $T'$ has at most one gap.
Hence, the assumption on $j$ implies that all gaps in $T'$ occur
either at or after the $p$-th box of each column. Hence, the content
of the boxes of $T'$ supported at $s'$ must be equal to its position
in the column. This proves \eqref{e:c=p}.

Say $\ffbox{j}_{r'-j+1}$ is in the $l$-th column of $S'$.
Since $j>p$, by definition of $r'$ we have
$$r'-j+1<r+2(k-m')+n-2p+3=s-2(m'-1)+2,$$
which, together with \eqref{e:columnpos}, implies that $l\ge b+m'-1\ge b$. As $S'$ is
semi-standard, its $b$-th column must have a box whose content is at
least $j$. Since the length of the $b$-th column is $p$, this
implies that the $b$-th column of $S'$ has a gap. Suppose the
$j'$-th box of the $b$-th column of $S'$ has a gap and let $d$ be
the content of this box. In particular, since the columns are increasing, we have
\begin{equation}\label{e:d>j'}
d>j'
\end{equation}
Thus, $\ffbox{d}_{s+2(p-j')+2}$ is a box in the $b$-th column of $S'$
while $\ffbox{d-1}_{s+2(p-j')+4}$ is not. This implies that
$Y^{-1}_{d-1,s+2(p-j')+d+2}$ appears in $\bs\omega^{S'}$ and hence,
$\ffbox{d-1}_{s+2(p-j')+4}$ must be a box in $T'$, say, at its
$l'$-th column. Observe that, $s+2(p-j')+4=s+2(p-(j'-2))$ is the
support of the $(j'-2)$-th box of the first column of $T'$ and,
hence, it is also the support of the $(j'-l'-1)$-th box of the $l'$-th
column of $T'$. Moreover, since $s+2(p-j')+4>s$, \eqref{e:c=p}
implies that $d-1=j'-l'-1$. Therefore,
\begin{equation*}
d= j'-l'\le j',
\end{equation*}
yielding a contradiction with \eqref{e:d>j'}.

Suppose now that $j<p$. Since $T'$ is semi-standard, this implies
that $m'=1$ and \eqref{e:defr'rands} gives that
$$r'-j+1=s+2+2(p-j)$$
is the support of the $j$-th box of the $b$-th column of $S'$ whose
content is at least $j$ (because the columns of $S'$ are
increasing). Since $S'$ is semi-standard, this implies that
\begin{equation}\label{e:jinl>b}
\ffbox{j}_{\:r'-j+1} \quad\text{is a box in the } l\text{-th column of } S' \text{ with } l\ge b.
\end{equation}
We treat the cases $l=b$ and $l>b$ separately.

If $l=b$, \eqref{e:notaboxS'} implies that the $b$-th column of $S'$
has $\ffbox{d}_{\:r'-j-1}$ with $d>j+1$. Lemma \ref{lem gap} and
Remark \ref{r:gapsdontcancel} then imply that  $Y^{-1}_{d-1,
r'-j+d-1}$ appears in $\bs\omega^{S'}$, forcing
$\ffbox{\:d-1}_{\:r'-j+1}$  to be a box in $T'$. Observe that
$r'-j+1=s+2(p-(j-1))$ is the support of the $(j-1)$-th box of the
first column of $T'$ and recall that, since this column has a gap in
the $j$-th box, all boxes above it have have content equal to its
position in the column. In particular, the content of the box
supported at $r'-j+1$ is $j-1<d-1$. Since $T'$ is semi-standard, the
content of the boxes of the remaining columns supported at $r'-j+1$
must be at most $j-1$ contradicting that $\ffbox{\:d-1}_{\:r'-j+1}$
is a box in $T'$.

Finally, suppose $l>b$. Using that $S'$ is semi-standard, \eqref{e:jinl>b} implies
that the $b$-th column of $S'$ contains a box
$\ffbox{j'}_{\:r'-j+1+2(l-b)}$ with $j'\ge j$. Since, by \eqref{e:defr'rands}
\begin{equation*}
r'-j+1+2(l-b) =s+2+2(p-(j-(l-b)))
\end{equation*}
is the support of the $(j-(l-b))$-th box of the $b$-th column of
$S'$ by \eqref{e:suppofb} and $j-(l-b)< j$, it follows that there
exists a gap in the $b$-th column of $S'$ at the $j''$-th box for
some $j''\le j-(l-b) < j$. Let $d$ be the content of this box. In
particular, $d>j''$. This time, the usual application of Lemma
\ref{lem gap} and Remark \ref{r:gapsdontcancel} imply that
$\ffbox{\:d-1}_{\:r'-j+3+2(j-j'')}$ must be a box in $T'$.  Let us
show that this is impossible, yielding the desired contradiction.
Since
$$r'-j+3+2(j-j'')=s+2(p-(j''-2)),$$
\eqref{e:suppk'} implies that $r'-j+3+2(j-j'')$ is the support of the $(j''-2)$-th box of the first column of $T'$, which has a gap at the $j$-th box.
Thus, since $j''<j$ and the boxes above the $j$-th have their contents equal to their
position in the column, $T'$ has $\ffbox{\:j''\!-\!2}_{\:r'-j+3+2(j-j'')}$ in the first column. As $d>j''$, $\ffbox{\:d-1}_{\:r'-j+3+2(j-j'')}$ is not in the first column of $T'$. Since the sequence of contents of each row of $T'$ is decreasing, $\ffbox{\:d-1}_{\:r'-j+3+2(j-j'')}$ cannot be in any other column of $T'$ as well.
\end{proof}

For the next step, note that, by specializing \eqref{e:ex l-weights KR n} to the partitions described in \eqref{e:Tmpaspart}, we get
\begin{align}\label{e:omegaTtp}
\bs\omega^{T_{m,p}}  = Y_{p-1,s+p-2(m-1),m}\ Y_{p,s+p-2m+3,m}^{-1}\ Y_{n,s-2(k-p)-n+1,k-m}.
\end{align}

\begin{lem}\label{lem dom c,f,p}
If $0\le m\le k$ and $S'\in\stab(S)$ are such that $\bs\omega^{S'}\bs\omega^{T_{m,p}}\in D$, then
$S'=S_{c,f,p}$ for some $f\le |\lambda|$.
\end{lem}

\begin{proof}
If, $S'=S$, then $S'=S_{c,f,p}$ for any $f<c$ and there is nothing to do. Thus,
assume $S'\ne S$. It follows from Lemma \ref{lem no dom of form}
that we must also  have $T_{m,p}\ne T$, i.e., $m\ge 1$.
Observe that, if $p=1$, then $S_{c,f,p}=S$ since, in this case,
$f\le |\lambda|<c$.

We need to study the structure of gaps in $S'$. Since each gap
contributes with the appearance of a term of the form $Y_{i,r'}^{-1}$
in $\gb\omega^{S'}$  for some $r'$, it follows that $Y_{i,r'}$ must appear in
$\gb\omega^{T_{m,p}}$.
By \eqref{e:omegaTtp},  $Y_{i,r'}$
must be among the following elements:
\begin{equation}\label{e:cfpTpositive}
Y_{p-1, r+2(k-m')+n-p+1},\ 1\le m'\le m, \quad\text{and}\quad
Y_{n,r+2m''}, \ 0\le m''<k-m.
\end{equation}
In particular, $i\in\{p-1,n\}$.

Let us show that we must have $i=p-1$ (in particular, it
follows that each column of $S'$ can have at most one gap).
Indeed, suppose $i=n$ and that the corresponding gap occurs at the $l$-th column of $S'$, necessarily at the last box whose content is then $n+1$.
Since $S'$ is semi-standard of the form \eqref{rem tableau stairs}, this
implies that the contents of all the last boxes of the columns to
the left of the $l$-th column are also equal to $n+1$. In
particular, if $j\le\min\{l,b\}$, the content of the last box of the
$j$-th column of $S'$ is $n+1$. Moreover, \eqref{e:suppofb} and
\eqref{rem tableau stairs} imply that the support of this box is
$s+2(b-j+1)$. Using Lemma \ref{lem gap}, this implies that
$Y_{i,r'} = Y_{n,r+2(k+b-j+n-p+1)}$,
which contradicts \eqref{e:cfpTpositive} since $k+(b-j)+(n-p)+1\ge k +1 >k-m$.

Since, as we have observed, each column of $S'$
has at most one gap, it remains to show that there exists $f\le|\lambda|$ such that the $j$-th column of $S'$ has
a gap iff $c\le j\le f$, the gap occurs at the last box and its size is $p-l_j$.

We start by showing that, if the $j$-th column of $S'$ has a gap, then $j\ge c$. Assume, by contradiction, that $j<c$.
Let $s'$ be the support of the box where the gap occurs. In particular, since the
support of the last box of this column is $s+2(b-j+1)$ by \eqref{e:suppofb}, we have
$s'\ge s+2(b-j+1)$. Assume first that $j\le b$. Lemma \ref{lem gap} and Remark \ref{r:gapsdontcancel} imply that $Y_{p-1,s'+p}^{-1}$ appears in $\bs\omega^{S'}$.
But
\begin{equation*}
s'+p\ge s+2(b-j+1)+p
\stackrel{\eqref{e:s}}{=}r+2(k+b-j)+n-p+1>r+2(k-1)+n-p+1,
\end{equation*}
which contradicts \eqref{e:cfpTpositive}. In particular, the $b$-th column of $S'$ does not have a gap. If $b<j<c$, then $l_j=l_b=p$ (see the comment after \eqref{e:b}). Since, $S'$ is semi-standard, this implies that the $b$-th column also has a gap, yielding the desired contradiction.

Next, we show that if the $j$-th column has a gap, it must occur at
its last box. Indeed, Lemma \ref{lem gap} and \eqref{e:cfpTpositive}
imply that the content of this box  must be $p$. Hence, if
this were not the last box, it would follow that the content of the
last box is at least $p+1$. Since  $S'$ is semi-standard of the form
\eqref{rem tableau stairs}, this would imply that the last box of
the $b$-th column is at least $p+1$ implying that the $b$-th column
would have a gap (because it has $p$ boxes), yielding a
contradiction. The same reasoning implies that, if $c\le j'\le j$,
the last box of the $j'$-th column has content $p$ and, hence, has a
gap (because its length is at most $p-1$). Notice also that, since
there is no other gap in the $j$-th column, all the boxes but the
last one must have content equal to their position in the column. In
particular, the content of its
$(l_j-1)$-th box is $l_j-1$ and, hence, the size of the gap is $p-l_j$, as claimed.
\end{proof}

In light of Lemmas \ref{lem dom t,p} and \ref{lem dom c,f,p}, in order to complete the proof of \eqref{e:tensorA}, it remains to check that
\begin{equation}\label{e:tensorA'}
\bs\omega^{S_{c,f,p}}\bs\omega^{T_{m,p}}\in \mathcal P^+ \quad \Longleftrightarrow \quad
\begin{aligned}
& \text{either }\ 0\le m\le k' \text{ and } f<c\ \text{ or }\  k'\le m\le k,\\ & f\le|\lambda| \text{ and, } f=c+m-k'- \epsilon \text{ with } \epsilon\in\{0,1\}.
\end{aligned}
\end{equation}
But, for $c$ as defined in \eqref{e:defc} and noting that $l_f<p$, if $s_c$ is as in \eqref{e:omegaScfpgen}, then \eqref{e:suppofb} implies that $s_c=s-2(k'-1)$ and we can re-write \eqref{e:omegaScfpgen} as
\begin{align}\notag
\bs\omega^{S_{c,f,p}} = &  \left( \prod_{j=0}^{f-c}
Y_{l_{c+j}-1,s-2(k'-1+j)+l_{c+j}}\  Y_{p-1,s+p-2(k'-1+j)}^{-1}\ Y_{p,s+p-2(k'-1+j)-1} \right)
\\ \label{e:omegaScfp} \hfill\\\notag
&\times \left( \prod_{\substack{j>f:\\ l_j=l_f}}
Y_{l_f,r_{l_f}+2(\lambda(h_{l_f})-d_j)} \right) \left(\prod_{i=
p}^{n} Y_{i,r_i,\lambda(h_i)}\right) \left(\prod_{i=1}^{l_f-1}
Y_{i,r_i,\lambda(h_i)}\right).
\end{align}
Now, a simple comparison of \eqref{e:omegaScfp} with \eqref{e:omegaTtp} proves \eqref{e:tensorA'}.

It remains to prove the second statement of Proposition
\ref{p:multiplicity1A}, which is clear in the case that
$D=\{\bs\lambda\}$. Thus, we can assume \eqref{eq p,k'}
holds. Fix $\bs\mu\in D$, say
\begin{equation*}
\bs\mu = \bs\omega^{S_{c,f,p}}\bs\omega^{T_{m,p}}
\end{equation*}
with $m$ and $f$ as in \eqref{e:tensorA'}. Suppose $S'\in\stab(S), T'\in\stab(T)$
are such that
$$\bs\omega^{S'}\bs\omega^{T'}=\bs\mu.$$ Since $V_q(\bs\omega)$ and
$V_q(\bs\varpi)$ are thin, we are left to show that $S'=S_{c,f,p}$
and $T'=T_{m,p}$. For doing this, notice that, working with
\eqref{eq T D in A ii} similarly to how we deduced \eqref{e:Scfpeta}
and \eqref{e:xi_0}, it follows that
\begin{equation*}
\bs\omega^{S_{c,f,p}} = \bs\omega \prod_{\xi\in\Xi_1} A_\xi^{-1} \qquad\text{and}\qquad \bs\omega^{T_{m,p}} = \bs\varpi \prod_{\xi\in\Xi_2} A_\xi^{-1}
\end{equation*}
with $\Xi_1,\Xi_2\subseteq I\times\mathbb Z$ satisfying
\begin{equation}\label{e:Xj=Xj1}
(i,t)\in \Xi_1\ \Rightarrow\ i<p
\end{equation}
while
\begin{equation}\label{e:Xj=Xj2}
(i,t)\in \Xi_2\ \Rightarrow\ i\ge p\quad\text{and}\quad t\le s+p.
\end{equation}
Note that \eqref{e:Xj=Xj1} is a rephrasing of \eqref{e:xi_0}. Setting $\Xi=\Xi_1\cup\Xi_2$, it
then follows from \eqref{e:lrootsind} that there must exist a
partition $\Xi=\Xi_1'\cup\Xi_2'$ such that
\begin{equation*}
\bs\omega^{S'} = \bs\omega \prod_{\xi\in\Xi_1'} A_\xi^{-1} \qquad\text{and}\qquad \bs\omega^{T'} = \bs\varpi \prod_{\xi\in\Xi_2'} A_\xi^{-1}.
\end{equation*}
We will show that
\begin{equation}\label{e:XjinXj'}
\Xi_j\subseteq \Xi_j' \quad\text{for}\quad j=1,2,
\end{equation}
which clearly completes the proof of Proposition \ref{p:multiplicity1A}.

We first show \eqref{e:XjinXj'} for $j=2$. By contradiction, suppose
there exists $(i,t)\in \Xi_2\cap \Xi_1'$. By \eqref{e:rem action
tableau A}, this implies that, for obtaining $S'$ from $S$, a
modification $\ffbox{i}_{\: t-i} \to \ffbox{i+1}_{\: t-i}$ was
performed in some column of $S$. Since $t-i\le s$ by
\eqref{e:Xj=Xj2}, \eqref{e:suppofb} implies that this modification
can only be performed in columns to the right of the $b$-th column
of $S$. In particular, the content of the last box of the $b$-th column
of $S'$ is equal to $p$. We claim that the above modification also
implies that the content of the last box of $(b+1)$-th column of
$S'$ is at least $p+1$, contradicting the fact that $S'$ is
semi-standard. Indeed, since $i\ge p$ by \eqref{e:Xj=Xj2}, the last
box of the modified column has content larger or equal to $p+1$. The condition on diagonals for semi-standard tableaux implies that the same holds for the last box of the $(b+1)$-th column of $S'$ as claimed.

To prove \eqref{e:XjinXj'} for $j=1$, it suffices to show that we actually have $\Xi_2=\Xi_2'$. Suppose $\Xi_2\subsetneqq \Xi_2'$ and, hence, $\Xi_1\cap\Xi_2'\ne\emptyset$. By Lemma \ref{lem dom t,p}, there exists $0\le m'\le k$ such that $T' = T_{m',p}$ and, hence, \eqref{e:Xj=Xj2} is valid for $\Xi_2'$ in place of $\Xi_2$ (it is valid for $T_{l,p}$ for any $l$). It follows that there exists $(i,t)\in\Xi_1$ with $i\ge p$ contradicting \eqref{e:Xj=Xj1}. Hence, $m=m'$ showing that $\Xi_2=\Xi_2'$.

\subsection{The Simple Factors of $V$}\label{ss:sftpA}

We now complete the proof of Theorem \ref{t:mainr}. If neither \eqref{e:k'i0} nor \eqref{eq p,k'}
hold, then $D$ is a singleton by \eqref{eq D} and \eqref{e:tensorA} and, hence, $V$ is irreducible.

Assume first that \eqref{e:k'i0} holds and set
\begin{equation*}
\bs\mu_j=\bs\varpi\bs\omega^{S_{1,j,n+1}} \quad\text{for}\quad  0\le j\le k'.
\end{equation*}
Recall that $S_{1,j,n+1}=S_{1,|\lambda|,n+1}$ for all $j\ge|\lambda|$. Thus, $D=\{\bs\mu_j: 0\le j\le \min\{k',|\lambda|\}\}$ by \eqref{eq D}, while \eqref{e:Scfpgodown} implies that $\bs\mu_{j+1}\le\bs\mu_j$. In particular,
\begin{equation}\label{e:mu'(ii)}
\bs\lambda'= \bs\mu_{k'} = \min D.
\end{equation}
Under the present assumption, as explained in Section \ref{ss:scheme}, Theorem \ref{t:mainr} follows from the second statement of
Proposition \ref{p:multiplicity1A} together with the following lemma.

\begin{lem}\label{l:tpfactorsA<}
We have:
\begin{enumerate}[(a)]
\item $\bs\mu_j\in\wtl(V_q(\bs\lambda))$ for all $0\le j\le\min\{|\lambda|,k'-1\}$. In particular, if $k'>|\lambda|$, $D\subseteq\wtl(V_q(\bs\lambda))$.
\item $\bs\mu_{k'}\notin \wtl(V_q(\bs\lambda))$ if $k'\le|\lambda|$.
\end{enumerate}
\end{lem}

\begin{proof}
Since $\bs\mu_0=\bs\lambda$, part (a) clearly holds for $j=0$. For $j>0$, consider
$$\bs\mu_j'= \bs\omega^{S_{1,j,n}}\bs\varpi$$
and observe that $S_{1,j,n}=S$ if $l_j=n$. We claim that
\begin{equation*}
\bs\mu_j'\in\wtl(V_q(\bs\lambda))\quad\text{for all}\quad j>0,
\end{equation*}
which is obvious if $l_j=n$. Thus, assume $l_j<n$. Using Remark
\ref{rem appears} with $\bs\lambda$ in place of $\bs\omega$ and
$J=I\setminus\{n\}$, we have
$\wt_\ell(\chi_{J}(\bs\lambda))\subseteq \wt_\ell(V_q(\bs\lambda))$
and, hence, it suffices to show that
\begin{equation}\label{e:mu'inchiJ}
\bs\mu_j'\in\wtl(\chi_J(V_q(\bs\lambda))).
\end{equation}
By definition of $S_{1,j,n}$, all of its columns of length smaller than $n$ coincide with the corresponding column of $S$ and, hence, $\bs\lambda_J=\bs\omega_J$. Therefore, $V_q(\bs\lambda_J)$ is an increasing minimal affinization whose
highest weight corresponds to the tableau $S'$ obtained from $S$ by removing the columns of length $n$. Note that the $j$-th column of $S$ becomes the $j'$-th column of $S'$ where $j'=j-\lambda(h_n)$. Theorem \ref{t:qcharmin} then
implies that $(\bs\omega^{S'_{1,j',n}})_J\in
\wtl(V_q(\bs\lambda_J))$. One easily checks that
\begin{equation*}
\bs\lambda\cdot
\iota_J\left((\bs\omega_J^{-1}\bs\omega^{S'_{1,j',n}})_J\right) =
\bs\mu'_j
\end{equation*}
completing the proof of \eqref{e:mu'inchiJ}.

In particular, it follows that, for all $j$, $\bs\mu'_j$ satisfies condition (i)  of Proposition \ref{prop appears} with $\bs\lambda$ in place of $\bs\omega$.
We show next that, setting $J=\{n\}$, conditions (ii) and (iii) of Proposition \ref{prop appears} are also satisfied by $\bs\mu'_j$ and, therefore,
\begin{equation}
\wt_\ell(\chi_{\{n\}}(\bs\mu_j'))\subseteq \wt_\ell (V_q(\bs\lambda)).
\end{equation}
Proceeding as in the proof of \eqref{eq S1,j,n+1}, one sees that
\begin{equation}\label{eq mu_j' n}
\pi_n(\bs\omega^{S_{1,j,n}}) =
\begin{cases}
Y_{n,r_{i_0}+2(\lambda(h_{i_0})-j)+n-i_0,j}, & \text{ if } l_j<n,\\
Y_{n,r_n,\lambda(h_n)},& \text{ if } l_j=n,
\end{cases}
\end{equation}
which implies condition (ii). Condition (iii) is obvious if $l_j=n$. If $l_j<n$, \eqref{e:ScfpfromSbyAs}
applied to $S_{1,j,n}$ implies that there exists $\Xi\subseteq (I\setminus\{n\})\times\mathbb Z$ such that
\begin{equation*}
\bs\mu'_j=\bs\lambda \prod_{\xi\in\Xi} A_\xi^{-1}.
\end{equation*}
Since, for every $\{n\}$-dominant $\bs\nu$, the elements of $\wt_\ell(\chi_{\{n\}}(\bs\nu))$ are of the form $\bs\nu \prod_{r\in \mathbb Z}A^{-m_r}_{n,r}$ with $m_r\in\mathbb Z_{\ge 0}$, condition (iii) follows.

Part (a) of the lemma now follows if we show that
\begin{equation}\label{e:tpfactorsA<a}
\bs\mu_j\in\wt_\ell(\chi_{\{n\}}(\bs\mu_j')) \quad\text{for
all}\quad 1\le j\le\min\{|\lambda|,k'-1\}.
\end{equation}
Observe that, since $S_{1,j,n+1}$ is obtained from $S_{1,j,n}$ by
modifying the contents of the last boxes of the first $j$ columns
from $n$ to $n+1$, \eqref{e:rem action tableau A} together with
\eqref{e:rifromri0} and \eqref{e:k'i0} implies that
\begin{equation}\label{e:tpfactorsA<a'}
\bs\mu_{j}=\bs\mu_{j}' \left( \prod_{l=1}^jA^{-1}_{n,r+2(k'-l)-1}\right).
\end{equation}
Assume first that $l_j<n$ and observe that
\begin{eqnarray*}
\pi_n(\bs\mu_j')&\stackrel{\eqref{eq mu_j' n} }{=}& Y_{n,r_{i_0}+2(\lambda(h_{i_0})-j)+n-i_0,j} \ Y_{n,r,k}
\stackrel{\eqref{e:k'i0}}{=} Y_{n,r+2(k'-1-j),j} \ Y_{n,r,k}.
\end{eqnarray*}
One easily checks that, if $j\le\min\{|\lambda|,k'-1\}$, the above
is the $q$-factorization of $\pi_n(\bs\mu_{j}')$. Thus, \eqref{teo sl2 irred} implies that
\begin{equation*}
V_q(\pi_n(\bs\mu_{j}'))\cong V_q(Y_{n,r+2(k'-1-j),j})\otimes
V_q(Y_{n,r,k}).
\end{equation*}
one now deduces \eqref{e:tpfactorsA<a} by applying
\eqref{e:l-weights KRnroot} to the first factor of this tensor
product and comparing with \eqref{e:tpfactorsA<a'}. Indeed, the
$\ell$-roots in \eqref{e:tpfactorsA<a'} coincide with those
appearing in \eqref{e:l-weights KRnroot} for the constant non zero
partition (i.e., for the lowest weight).

If $l_j=n$, which implies $\bs\mu'_j=\bs\lambda, j\le\lambda(h_n)\le|\lambda|$, and $i_0=n$, then
$$\pi_n(\bs\lambda) \stackrel{\eqref{eq mu_j' n} }{=} Y_{n,r_n,\lambda(h_n)}\ Y_{n,r,k} \stackrel{\eqref{e:k'i0}}{=} Y_{n,r-2(\lambda(h_n)-k'+1),\lambda(h_n)} Y_{n,r,k}$$
and this is the $q$-factorization of $\pi_n(\bs\lambda)$ provided
$\lambda(h_n)\le k'-1$ (in particular, $j\le k'-1$). One then
deduces \eqref{e:tpfactorsA<a} as in the previous case (but
$\bs\mu_j$ may arise from a weight higher than the lowest one in
$V_q(Y_{n,r-2(\lambda(h_n)-k'+1),\lambda(h_n)})$). Finally, if
$\lambda(h_n)>k'-1$, then
\begin{equation*}
\pi_n(\bs\lambda) = Y_{n,r_n,\lambda(h_n)+k-k'+1} \ Y_{n,r,k'-1}
\end{equation*}
is the $q$-factorization of $\pi_n(\bs\lambda)$ and, hence
\begin{equation*}
V_q(\pi_n(\bs\lambda)) \cong V_q(Y_{n,r_n,\lambda(h_n)+k-k'+1})\otimes V_q(Y_{n,r,k'-1}).
\end{equation*}
This time, we apply \eqref{e:l-weights KRnroot} to the second factor of this tensor product to obtain $\bs\mu_j$ for $j\le k'-1$, thus proving
\eqref{e:tpfactorsA<a}.

To prove part (b), it suffices to show that there exist two simple modules $V'$ and $V''$ such that
\begin{equation}\label{e:tpfactorsA<b}
\bs\mu_{k'}\notin \wtl(V'\otimes V'')  \quad\text{and}\quad V_q(\bs\lambda) \text{ is a simple factor of } V'\otimes V''.
\end{equation}
For doing this, let $V'=V_q(\bs\omega^{S'})$ where $S'$ is the tableau formed by the first
$k'-1$ columns of $S$. In particular, $V'$ is an increasing minimal affinization. Evidently, $\bs\lambda(\bs\omega^{S'})^{-1} = \bs\varpi\bs\omega^{S\setminus S'}\in\mathcal P^+$ and, therefore, letting $V''$ be the simple module having this as its highest $\ell$-weight, the second claim in \eqref{e:tpfactorsA<b} is satisfied.

To complete the proof of \eqref{e:tpfactorsA<b}, we begin by observing that $V''$ is also an increasing minimal affinization. More precisely, we show that the tableau $S''$ obtained by the juxtaposition of $T$ and $S\setminus S'$ is of the form \eqref{rem tableau stairs}. Note that the hypothesis $k'\le|\lambda|$ implies that $S\setminus S'\ne\emptyset$ and that the support of the last box of its first column (the $k'$-th column of $S$) is $r_{l_{k'}}-l_{k'}+1+2(\lambda(h_{l_k'})-d_{k'})$ (recall \eqref{e:jisdjoflj}). On the other hand, the support of the last box of the last column of $T$ is $r-n+1$ and, hence, we need to show that
\begin{equation*}
(r-n+1)-2=r_{l_{k'}}-l_{k'}+1+2(\lambda(h_{l_{k'}})-d_{k'}).
\end{equation*}
But this is easily checked using \eqref{e:rifromri0} and \eqref{e:k'i0}.

Using \eqref{e:ScfpfromSbyAs} as before, we see that there exists $\Xi\subseteq I\times\mathbb Z$ such that
\begin{equation*}
\bs\mu_{k'}=\bs\lambda \prod_{\xi\in\Xi} A_\xi^{-1}.
\end{equation*}
Thus, \eqref{e:tpfactorsA<b} follows if we show that there is no partition $\Xi=\Xi'\cup \Xi''$ such that
\begin{equation*}
\bs\omega^{S'}\prod_{\xi\in\Xi'}
A_\xi^{-1}\in\wtl(V') \qquad\text{and}\qquad
\bs\omega^{S''}\prod_{\xi\in\Xi''} A_\xi^{-1}\in\wtl(V'').
\end{equation*}
By contradiction, assume such a partition exists. It follows from \eqref{e:rem action tableau A} and Lemma \ref{l:connectedquiver} that each
element $A_\xi^{-1}, \xi\in\Xi$, corresponds to adding $1$ to the
content of a particular box of some element of either $\stab(S')$ or $\stab(S'')$ so that the new tableau remains semi standard. More precisely,
if $\xi=(i,l)$, then \eqref{e:rem action tableau A} implies that the
modification associated to $A_\xi^{-1}$ is of the form
\begin{equation*}
\ffbox{i}_{\:l-i} \longrightarrow \ffbox{i+1}_{\:l-i}.
\end{equation*}
Inspecting \eqref{e:ScfpfromSbyAs}, one checks that
\begin{equation*}
\max\{l-i: (i,l)\in\Xi\} = r_{i_0}+2\lambda(h_{i_0})-i_0-1 =:s.
\end{equation*}
Note that \eqref{e:rifromri0} and \eqref{e:k'i0} imply that
\begin{equation}\label{eq s+2}
s+2\ \text{ is the support of the last box of the $(k-k'+1)$-th
column of } T
\end{equation}
(which is the $k'$-th one counted from right to left).
Let $\xi=(i,l)$ be such that $l-i=s$. This means that the boxes of
$S'$ and $S''$ supported above $s$ are not being modified.
Together with \eqref{eq s+2}, it follows that the first $k-k'+1$ columns of
$S''$ are not modified. In particular, since the first $k$ columns of $S''$ come from $T$ and $V_q(\bs\omega^{S''})$ is a minimal
affinization, all first $k$ columns of $S''$ are not modified.

On the other hand, another inspection of \eqref{e:ScfpfromSbyAs} shows that
$$(n,s-2(k'-1)+n)\in\Xi.$$
This corresponds to a modification of the form
\begin{equation*}
\ffbox{n}_{\:s-2(k'-1)} \longrightarrow \ffbox{n\!+1}_{\:s-2(k'-1)},
\end{equation*}
in some column of some element of $\stab(S')\cup\stab(S'')$. Since the content is $n$, this must be the last box the column.
Since the modified tableau is semi standard, the last box of each column to the left must also have $n+1$
as content. Since the first columns of $S''$ are left unmodified,  it follows that $(n,s-2(k'-1)+n)\notin\Xi''$. Observing that, by construction, the box of
$S'$ having the lowest support is supported at $s-2(k'-1)+2$, it
follows that $(n,s-2(k'-1)+n)$ cannot be in $\Xi'$ as well, yielding
the desired contradiction.
\end{proof}

It remains to prove Theorem \ref{t:mainr} when \eqref{eq p,k'} holds. In that case, by \eqref{e:tensorA},
$$D=\left\{\bs\omega^{S_{c,c+m-k'-1,p}}\bs\omega^{T_{m,p}},\;\bs\omega^{S_{c,c+m-k',p}}\bs\omega^{T_{m,p}}
:  0\le m\le \min \left\{k,|\lambda|_{p-1}+k'\right\}\right\},$$ where
$c=\ _p|\lambda|+1$ (recall that $S_{c,f,p}=S$ for $f<c$ and for $c>|\lambda|$, $S_{c,f,p}=S_{c,|\lambda|,p}$ for
$f>|\lambda|$, and $T_{m,p}=T_{k,p}$ for $m>k$). Let
$$\bs\mu_m:=\bs\omega\bs\omega^{T_{m,p}} \quad\text{for}\quad 0\le m\le k$$
and observe that $\bs\lambda'=\bs\mu_{k'}$ and
$$\{\bs\nu\in D: \bs\nu\ge \bs\lambda'\} = \{\bs\mu_m: 0\le m\le \min\{k',k\}\}.$$
In particular, if either $k'>k$ or $p=\min\supp(\lambda)$, then
$\bs\lambda'=\bs\mu_k$ is the smallest element of
$D$. As explained in Section \ref{ss:scheme}, Theorem \ref{t:mainr}
follows from the second statement of Proposition
\ref{p:multiplicity1A} together with the following lemma.

\begin{lem}\label{lem lweight in A} We have
\begin{enumerate}[(a)]
\item $\{\bs\nu\in D: \bs\nu>\bs\lambda'\}\subseteq \wtl(V_q(\bs\lambda))$.
\item If $k'>k$, then $\bs\lambda'\in \wtl(V_q(\bs\lambda))$ or, equivalently, $D\subseteq \wtl(V_q(\bs\lambda))$.
\item If $k'\le k$, then $\bs\lambda'\notin \wtl(V_q(\bs\lambda))$.
\item If $p>\min\supp(\lambda)$, then $\left\{\bs\omega^{S_{c,c+m-k',p}}\bs\omega^{T_{m,p}}: k'\le m\le \min \left\{k,|\lambda|_{p-1}+k'\right\}\right\}\subseteq \wtl(V_q(\bs\lambda))$.
\item If $k'\le k$, then $\left\{\bs\omega^{S_{c,c+m-k'-1,p}}\bs\omega^{T_{m,p}} :  k'\le m\le \min \left\{k,|\lambda|_{p-1}+k' \right\}\right\}\subseteq \wtl(V_q(\bs\lambda'))$.
\end{enumerate}
\end{lem}

\begin{proof}
The proof of part (c) is similar in spirit to that of part (b) of the previous lemma, but more complicated. Similarly, the proofs of the other parts have the same spirit of that of  part (a) of the previous lemma. In particular, the arguments will contain the proof of a chain of claims of the form
\begin{equation*}
\bs{\nu}\in \wtl(\chi_J(V_q(\bs{\mu}_J)))\subseteq \wtl(V_q(\bs{\pi}))
\end{equation*}
for some $\bs{\pi}\in\mathcal P^+, J\subseteq I$ connected, $\bs{\mu}\in\wtl(V_q(\bs{\pi}))\cap\mathcal P_J^+$, and  $\bs{\nu}\in\mathcal P$. At the beginning of each part we will give a summary of the associated chain by drawing a picture of the form
\begin{equation*}
\bs{\pi} {\longrightarrow} \bs{\mu} \stackrel{J}{\longrightarrow} \bs{\nu}
\end{equation*}
For most arrows we will have to check that the hypothesis of Proposition \ref{prop appears} applies with $\bs{\pi}$ in place of $\bs{\omega}$.
Note that, after part (b), there is nothing to prove in part (d) if $k'>k$. Hence, we will assume without further mention that $k'\le k$ once parts (a) and (b) are proved.\vspace{10pt}

\noindent {\bf Parts (a) and (b):} 
Since $\bs\mu_0=\bs\lambda$, we have $\bs\mu_0\in \wtl(V_q(\bs\lambda))$. Given $1\le m\le k$, define
$$\bs\mu_m'=\bs\omega\bs\omega^{T_{m,p+1}}.$$
The summary of the proof is given by the picture 
$$\bs{\lambda} \xrightarrow{\{p+1,\cdots,n\}} \bs\mu'_m \stackrel{\{p\}}{\longrightarrow} \bs\mu_m$$ 
We begin by proving that
\begin{equation}\label{e:claima}
\bs\mu_m'\in\wtl(V_q(\bs\lambda))\quad\text{for all}\quad 1\le m\le k
\end{equation}
which is obvious if $p=n$ since $T_{m,n+1}=T$. Thus, assume $p<n$. Using Remark \ref{rem appears} with $\bs\lambda$ in place of $\bs\omega$ and
$J=\{p+1,\ldots,n\}$, we have
$\wt_\ell(\chi_{J}(\bs\lambda))\subseteq \wt_\ell(V_q(\bs\lambda))$
and, hence, in order to prove \eqref{e:claima}, it suffices to show that
\begin{equation*}
\bs\mu_m'\in\wtl(\chi_{J}(V_q(\bs\lambda)))\quad\text{for all}\quad 1\le m\le k.
\end{equation*}
Since $p\notin J$, applying Proposition \ref{prop without dom} to the algebra $U_q(\tlie g_J)$ with $\bs\omega_J$ in place of
$\bs\omega$ and $\bs\varpi_J$ in place of $\bs\varpi$, we get
$$V_q(\bs\omega_{J})\otimes V_q(\bs\varpi_{J})\cong V_q(\bs\lambda_{J}).$$
The $\ell$-weights of $V_q(\bs\varpi_{J})$ are also determined by the elements of $\stab(T)$.
Therefore, $(\bs\omega^{T_{m,p'}})_{J}\in \wtl(V_q(\bs\varpi_{J}))$ for all $m,p'$.
One easily checks that
\begin{equation*}
\bs\lambda\cdot \iota_{J}\left((\bs\omega^{-1}\bs\varpi^{-1}\bs\omega^{T_{m,p+1}})_{J}\right)
= \bs\mu'_m
\end{equation*}
completing the proof of \eqref{e:claima}.

Next, we use Proposition \ref{prop appears} with $\bs\lambda$ in
place of $\bs\omega$ and $J=\{p\}$. The previous paragraph implies
that $\bs\mu'_m$ satisfies condition (i) of that Proposition for all
$m$. Equation \eqref{e:omegaTtp}, with $p+1$ in place of $p$,
implies that condition (ii) is also satisfied. As for condition
(iii) with $p<n$ (if $p=n$ this condition is obvious), \eqref{eq T D
in A ii} implies that
$$\bs\mu_m'=\bs\lambda \cdot \prod_{l=1}^{m}A^{-1}_{n,p+1,r+2(k-l)}.$$
On the other hand, for any $\bs\nu\in\mathcal P_{\{p\}}^+$, the
elements of $\wt_\ell(\chi_{\{p\}}(\bs\nu))$ are of the form $\bs\nu
\prod_{s\in \mathbb Z}A^{-m_s}_{p,s}$ with $m_s\in\mathbb Z_{\ge
0}$. Since $p<n$, condition (iii) must also be satisfied. It then follows from
Proposition \ref{prop appears} that
$$\wt_\ell(\chi_{\{p\}}(\bs\mu_m'))\subseteq \wt_\ell (V_q(\bs\lambda)).$$
Parts (a) and (b) of the lemma now follows if we show that
\begin{equation}\label{e: mu t factors A}
\bs\mu_m\in\wt_\ell(\chi_{\{p\}}(\bs\mu_m')) \quad\text{for all}\quad 1\le m\le\min\{k'-1,k\}.
\end{equation}
Observe that \eqref{eq T D in A ii} implies that
\begin{equation}\label{eq mu_t'}
\bs\mu_{m}=\bs\mu_{m}'\left(\prod_{l=1}^{m}A_{p,r+2(k-l)+n-p+1}^{-1} \right).
\end{equation}
Moreover, \eqref{e:omegaTtp}  implies that
$$\pi_p(\bs\mu_{m}')=Y_{p,r+2(k-m)+n-p,m}\; Y_{p,r_p,\lambda(h_p)}\stackrel{\eqref{eq p,k'}}{=} Y_{p,r_p+2(k'-1-m),m} \; Y_{p,r_p,\lambda(h_p)}.$$
One easily checks that, for $m\le\min\{k'-1,k\}$, the above is the $q$-factorization of $\pi_p(\bs\mu_{m}')$. Thus, \eqref{teo sl2 irred} implies that
\begin{equation*}
V_q(\pi_p(\bs\mu_{m}'))\cong V_q(Y_{p,r_p+2(k'-1-m),m})\otimes V_q(Y_{p,r_p,\lambda(h_p)})=V_q(Y_{p,r+2(k-m)+n-p,m})\otimes V_q(Y_{p,r_p,\lambda(h_p)}).
\end{equation*}
We now proceed as in the Proof of part (a) of Lemma \ref{l:tpfactorsA<}. Namely, applying \eqref{e:l-weights KRnroot} for the subalgebra $U_q(\tlie g_{p})$  to the first factor of the above tensor product and comparing with \eqref{eq mu_t'} one deduces \eqref{e: mu t factors A}.

\vspace{10pt}
\noindent {\bf Part (c):} Consider the tableau $S'$ formed by the first $_{p+1}|\lambda|$ columns of $S$, the tableau  $S''$  formed by juxtaposing the first $\lambda(h_p)-k'+1$ columns of length $p$ of $S$ and $T$, and the tableau $S'''$ formed by the remaining
columns of $S$ (i.e., the tableau whose columns are those to the right of the $b$-th column of $S$ where $b$ is given by \eqref{e:b}). Then,
$$\bs\omega^{S'}=\prod_{i=p+1}^{i_0} Y_{i,r_{i}\lambda(h_i)}, \qquad
\bs\omega^{S''}= Y_{p,r_{p}+2(k'-1),\lambda(h_p)-k'+1}\ Y_{n,r,k},$$
$$\bs\omega^{S'''}=\left(\prod_{i=1}^{p-1} Y_{i,r_{i},\lambda(h_i)}\right) Y_{p,r_{p},k'-1},\quad\text{and}\quad
\bs\omega^{S'}\bs\omega^{S''}\bs\omega^{S'''} = \bs\lambda.$$ In
particular, $V_q(\bs\lambda)$ is the simple quotient of the
submodule of $V_q(\bs\omega^{S'})\otimes V_q(\bs\omega^{S''})\otimes
V_q(\bs\omega^{S'''})$ generated by the top weight space and part
(c) follows if we show that
\begin{equation}\label{e:mu not in}
\bs\lambda' \notin \wtl\left(V_q(\bs\omega^{S'})\otimes
V_q(\bs\omega^{S''})\otimes V_q(\bs\omega^{S''})\right).
\end{equation}
It is clear from the construction of $S'$ and $S'''$ that
$V_q(\bs\omega^{S'})$ and $V_q(\bs\omega^{S'''})$ are increasing
minimal affinizations. On the other hand, \eqref{eq p,k'} implies
that $V_q(\bs\omega^{S''})$ is a decreasing minimal affinization. In
any case, the $\ell$-weights of all three factors are represented by
the corresponding set of semi-standard tableaux.

By \eqref{eq T D in A ii},
\begin{equation}\label{eq mu ii}
\bs\lambda'=\bs\lambda \left(\prod_{l=1}^{k'}A^{-1}_{n,p,r+2(k-l)}\right) = \bs\lambda \prod_{\xi\in\Xi} A_\xi^{-1}
\end{equation}
for some $\Xi\subseteq I\times\mathbb Z$. Equation \eqref{e:mu not in}
follows if we show that there is no partition $\Xi=\Xi'\cup
\Xi''\cup \Xi'''$ such that
\begin{equation*}
\bs\omega^{S'}\prod_{\xi\in\Xi'}
A_\xi^{-1}\in\wtl(V_q(\bs\omega^{S'})), \quad
\bs\omega^{S''}\prod_{\xi\in\Xi''}
A_\xi^{-1}\in\wtl(V_q(\bs\omega^{S''})), \quad
\bs\omega^{S'''}\prod_{\xi\in\Xi'''}
A_\xi^{-1}\in\wtl(V_q(\bs\omega^{S'''})).
\end{equation*}
By contradiction, assume such a partition exists. As before, each
element $\xi\in\Xi$, say $\xi=(i,l)$, corresponds to a modification
of the form
\begin{equation*}
\ffbox{i}_{\:l-i} \longrightarrow \ffbox{i+1}_{\:l-i}
\end{equation*}
in some tableau belonging to $\stab(S')\cup\stab(S'')\cup\stab(S''')$. Inspecting \eqref{eq mu ii}, one
checks that
\begin{equation}\label{e:minmax}
\max\{l-i: i\in I, (i,l)\in\Xi\} = s \qquad \text{and} \qquad \min\{i: (i,l)\in\Xi \text{ for some } l\in\mathbb Z\} = p,
\end{equation}
where $s$ is given by \eqref{e:suppk'}.  This means that the boxes
of $S'$, $S''$, and $S'''$ with support larger than $s$ are not
modified. Together with \eqref{e:suppofb}, it follows that $S'$ and
the columns of $S''$ coming from $S$ are left unmodified.  The
condition on diagonals for semi-standard tableaux imply that the
(possibly) modified element of $\stab(S'')$ has the first $p$ boxes
of every column coinciding with that of $S''$.

On the other hand, another inspection of \eqref{eq mu ii}, recalling
that $k'\le k$, shows that
$$\xi_0:=(p,s-2(k'-1)+p)\in\Xi.$$
This corresponds to a modification of the form
\begin{equation}\label{eq change}
\ffbox{p}_{\:s-2(k'-1)} \longrightarrow \ffbox{p+1}_{\:s-2(k'-1)}.
\end{equation}
Since all the tableaux are column increasing, if the box on which
this modification is being performed is the $j$-th box on its
column, we must have $j\le p$. Hence, it follows from the previous
paragraph that $\xi_0\in\Xi'''$. We will show that this is a
contradiction.

Indeed, by construction, the last box of the last column of length
$p$ of $S'''$ is supported at $s-2(k'-1)+2$. If we had $\min
\supp(\lambda)=p$, it would follow that $S'''$ had no box supported
at $s-2(k'-1)$, yielding a contradiction. Thus, we can assume that
$\min \supp(\lambda)<p$ and, hence, the first column of $S'''$ which
has a box supported at $s-2(k'-1)$ has length $i<p$. This implies
that, before the modification \eqref{eq change} can be performed, we
need a modification of the form
\begin{equation*}
\ffbox{i}_{\:s-2(k'-1)} \longrightarrow \ffbox{i+1}_{\:s-2(k'-1)}
\end{equation*}
which implies that $(i,s-2(k'-1)+i)\in\Xi$, contradicting the second
statement in \eqref{e:minmax}.

\vspace{10pt}\noindent {\bf Part (d):} Fix $k'\le m \le \min\{k,|\lambda|_{p-1}+k'\}$, let $f= c+m-k'$, and define
\begin{equation*}
\bs\nu=\bs\omega^{S_{c,f,p}}\bs\omega^{T_{m,p}}, \qquad \bs\nu'=\bs\omega^{S_{c,f,p}}\bs\omega^{T_{m,p+1}}.
\end{equation*}
The summary of the proof is given by the picture 
$$\bs{\lambda} \xrightarrow{\{p+1,\cdots,n\}}  \bs\mu'_m \xrightarrow{\{1,\cdots,p-1\}}  \bs\nu' \xrightarrow{\{p\}}  \bs\nu$$ 

An analogous argument to that used to prove \eqref{e:claima} shows that
\begin{equation}\label{eq nu'in}
\bs\nu'\in\wtl(V_q(\bs\lambda)).
\end{equation}
Indeed, by \eqref{e:claima}, $\bs\mu'_m\in \wtl(V_q(\bs\lambda))$. Proceeding as in the paragraph after the proof of \eqref{e:claima}, this time with $J=\{1,\ldots,p-1\}$, we see that
\begin{equation}\label{e:mu'_mJin}
\wtl(\chi_{J}(V_q(\bs\mu'_m)))\subseteq \wtl(V_q(\bs\lambda)).
\end{equation}
Consider the tableaux $\tilde S$ obtained from $S$ by restriction to $J$, i.e., by removing the columns of length larger or equal to $p$ and regarding it as a tableau for the diagram sub-algebra associated to $J$. Consider also $\tilde c=c- _p|\lambda|, \tilde f=f- _p|\lambda|$, and observe that
\begin{equation*}
(\bs\mu'_m)_J=\bs\omega_J = \bs\omega^{\tilde S} \qquad\text{and}\qquad \bs\omega_J(\bs\omega^{-1}\bs\omega^{S_{c,f,p}})_J = \bs\omega^{\tilde S_{\tilde c,\tilde f,p}}\in\qch(V_q(\bs\omega_J)).
\end{equation*}
Since $$\bs\nu'=\bs\mu'_m\cdot \iota_J((\bs\omega^{-1}\bs\omega^{S_{c,f,p}})_J)=\bs\mu'_m\cdot
\iota_J((\bs\mu'_m)_J^{-1}\bs\omega^{\tilde S_{\tilde c,\tilde f,p}})\in\wtl(\chi_{J}(V_q(\bs\mu'_m))),$$
\eqref{eq nu'in} follows from \eqref{e:mu'_mJin}.

Next, we use Proposition \ref{prop appears} with $\bs\lambda$ in
place of $\bs\omega$ and $J=\{p\}$. Note that \eqref{eq nu'in} implies that
$\bs\nu'$ satisfies condition (i) of that Proposition. Equations
\eqref{e:omegaTtp}, with $p+1$ in place of $p$, and
\eqref{e:omegaScfp} imply that condition (ii) is also satisfied. As
for condition (iii), note that the condition $m\ge k'$ implies $f\ge
c$ and, since $c=\ _p|\lambda|+1$, it follows that
$$l_j< p \quad\text{for all}\quad j\ge c.$$ In particular, \eqref{eq T D in A ii} and
\eqref{e:ScfpfromSbyAs} imply that
$$\bs\nu'= \bs\lambda\bs\eta^{-1} \quad\text{with}\quad \bs\eta\in\mathcal Q_{J'}^+,\ J'=I\setminus\{p\}.$$
On the other hand, for any $\bs\mu\in\mathcal P_{\{p\}}^+$, the elements of
$\wt_\ell(\chi_{\{p\}}(\bs\mu))$ are of the form $\bs\mu \prod_{l\in \mathbb Z}A^{-m_l}_{p,l}$, proving that condition (iii) must also be
satisfied. It then follows from Proposition \ref{prop appears} that
$$\wt_\ell(\chi_{\{p\}}(\bs\nu'))\subseteq \wt_\ell (V_q(\bs\lambda))$$
and part (d) follows if we show that
\begin{equation}\label{e: nu factors A}
\bs\nu\in\wt_\ell(\chi_{\{p\}}(\bs\nu')).
\end{equation}
Observe that \eqref{eq T D in A ii} implies that
\begin{equation}\label{eq nu'}
\bs\nu=\bs\nu' \prod_{l=1}^{m}A_{p,r+2(k-l)+n-p+1}^{-1}.
\end{equation}
Moreover, \eqref{e:omegaTtp} and \eqref{e:omegaScfp} imply that
\begin{eqnarray*}
\pi_p(\bs\nu')&=&Y_{p,r+2(k-m)+n-p,m}\; Y_{p,r_p,\lambda(h_p)}\left(\prod_{l=0}^{m-k'}Y_{p,s+p-2(k'-1+l)-1}\right)\\
&\stackrel{\eqref{e:s}}{=}& Y_{p,r+2(k-m)+n-p,m} \;
Y_{p,r_p,\lambda(h_p)}\prod_{l=0}^{m-k'}Y_{p,r_p-2(l+1)}
\\
&=&  Y_{p,r+2(k-m)+n-p,m} \;
Y_{p,r_p,\lambda(h_p)} \; Y_{p,r_p+2(k'-1-t),m-k'+1}\\
&=&  Y_{p,r+2(k-m)+n-p,m}\; Y_{p,r_p+2(k'-1-m),\lambda(h_p)+m-k'+1}.
\end{eqnarray*}
The hypothesis  $p>\min\supp(\lambda)$ was used in the first equality (for $p=\min\supp(\lambda)$ the term in parenthesis would not exist).
Using that $r+2(k-m)+n-p=r_p+2(k'-1-m)$ (by \eqref{eq p,k'}), one
easily checks that the the last line above is the $q$-factorization of
$\pi_p(\bs\nu')$ and we apply \eqref{teo sl2 irred} to get
\begin{equation*}
V_q(\pi_p(\bs\nu'))\cong V_q(Y_{p,r+2(k-m)+n-p,m})\otimes
V_q(Y_{p,r_p+2(k'-1-m),\lambda(h_p)+m-k'+1}).
\end{equation*}
Applying \eqref{e:l-weights KRnroot} as before to the first factor of the
above tensor product and comparing with \eqref{eq nu'} one easily
deduces \eqref{e: nu factors A}.

\vspace{10pt}\noindent {\bf Part (e):} Fix $k'\le m \le \min\{k,|\lambda|_{p-1}+k'\}$ and set $f=c+m-k'-1$ and
$$\bs\nu=\bs\omega^{S_{c,f,p}}\bs\omega^{T_{m,p}}.$$
We need to show that
\begin{equation}\label{e:parte}
\bs\nu\in\wtl(V_q(\bs\lambda')).
\end{equation}
If $m=k'$, then $\bs\nu=\bs\lambda'$ and this is obvious. For $m>k'$, consider
\begin{equation*}
\bs\mu=\bs\omega^{S_{c,f,p-1}}\bs\omega^{T_{m,p}}, \quad\bs\nu'=\bs\omega\bs\omega^{T_{m,p}},\quad \text{and}\quad
\bs\nu''=\bs\omega\bs\omega^{T'},
\end{equation*}
where $T'$ is the tableau whose first $k'$ columns coincide with those of $T_{k',p}$ and the remaining columns are
equal to those of $T_{m,p+1}$. 
The summary of the proof is given by the picture 
$$\bs{\lambda}' \xrightarrow{\{p+1,\cdots,n\}}  \bs{\nu}'' \xrightarrow{\{p\}}  \bs\nu' \xrightarrow{\{1,\cdots,p-2\}}  \bs\mu \xrightarrow{\{p-1\}}  \bs\nu$$ 

One easily checks that $T'\in \stab(T)$ and uses \eqref{eq T D in A ii} to get
\begin{equation}\label{eq nu'' mu}
\bs\nu''=\bs\lambda'\left(\prod_{l=k'+1}^{m}A_{n,p+1,r+2(k-l)}^{-1}\right) \quad\text{and}\quad
\bs\nu'=\bs\nu''\left(\prod_{l=k'+1}^{m}A_{p,r+2(k-l)+n-p+1}^{-1}\right).
\end{equation}
An analogous argument to that used to prove \eqref{e:claima} (with $(\bs\omega^{T_{k',p}})_J=Y_{n,r,k-k'}$ in
place of $\bs\varpi_J$), shows that
\begin{equation}\label{eq nu''in}
\bs\nu''\in\wtl(\chi_{J}(V_q(\bs\lambda')))\subseteq \wtl(V_q(\bs\lambda')) \quad\text{with}\quad J=\{p+1,\ldots,n\}.
\end{equation}

Next, we use Proposition \ref{prop appears} with $\bs\lambda'$ in place of $\bs\omega$ and $J=\{p\}$. Note that
\eqref{eq nu''in} implies that $\bs\nu''$ satisfies condition (i) of
that Proposition. Equation \eqref{e:omegaTtp} implies that the term of the form $Y^{-1}_{p,t}$ coming from $T_{k',p}$ in $\bs{\nu}''$ is the same one that $T_{k',p}$ contributes to $\bs\lambda'$. Since $\bs \lambda'$ is dominant, it means that  $Y_{p,t}$ appears in $\bs{\omega}$. Therefore, $Y^{-1}_{p,t}$ does not appear in $\bs{\nu}''$ proving that condition (ii) is also satisfied.
Condition (iii) follows since the elements of $\wt_\ell(\chi_{p}(\bs\lambda'))$ are of the form
$\bs\lambda' \prod_{l\in \mathbb Z}A^{-m_l}_{p,l}$ which is incompatible with the first part of \eqref{eq nu'' mu}.
Thus, it follows from Proposition \ref{prop appears} that
$$\wt_\ell(\chi_{p}(\bs\nu''))\subseteq \wt_\ell (V_q(\bs\lambda')).$$

Observe that
\begin{eqnarray*}
\pi_{p}(\bs\nu'')&\stackrel{\eqref{eq nu'' mu}}{=}&\pi_{p}(\bs\lambda')\;\pi_{p}\left(\prod_{l=k'+1}^{m}A_{p+1,r+2(k-l)+n-p+1}^{-1}\right)\\
&=&Y_{p,r_{p}+2k',\lambda(h_p)-k'}\;\left(\prod_{l=k'+1}^{m}Y_{p,r+2(k-l)+n-p}\right)\\
&=&Y_{p,r_{p}+2k',\lambda(h_p)-k'}\;Y_{p,r+2(k-m)+n-p,m-k'}\\
&\stackrel{\eqref{eq p,k'}}{=}&
Y_{p,r_{p}+2k',\lambda(h_p)-k'}\;Y_{p,r_{p}+2(k'-m-1),m-k'}.
\end{eqnarray*}
One easily checks that the above is the $q$-factorization of
$\pi_p(\bs\nu'')$ and, hence,
\begin{eqnarray*}
V_q(\pi_p(\bs\nu''))&\cong&
V_q(Y_{p,r_{p}+2k',\lambda(h_p)-k'})\otimes
V_q(Y_{p,r_{p}+2(k'-t-1),t-k'})\\
&=&V_q(Y_{p,r_{p}+2k',\lambda(h_p)-k'})\otimes
V_q(Y_{p,r+2(k-t)+n-p,t-k'}).
\end{eqnarray*}
Applying \eqref{e:l-weights KRnroot} to the first factor of
the above tensor product as usual and comparing with the second part of \eqref{eq nu'' mu}, we get
\begin{equation}\label{eq nu'in e}
\bs\nu'\in \wt_\ell(\chi_{p}(\bs\nu''))\subseteq \wt_\ell (V_q(\bs\lambda')).
\end{equation}

We now apply a similar argument to prove that
\begin{equation}\label{eq nu'inlambda'p-1}
\wtl(\chi_{J}(V_q(\bs\nu')))\subseteq \wtl(V_q(\bs\lambda')) \quad\text{with}\quad J=\{1,\ldots,p-1\}.
\end{equation}
Note that \eqref{eq nu'in e} implies that $\bs\nu'$ satisfies
condition (i) of Proposition \ref{prop appears} with $\bs\lambda'$
in place of $\bs\omega$ and $J$ as in \eqref{eq nu'inlambda'p-1}.
Condition (ii) is clear from \eqref{e:omegaTtp}. Since the elements
of $\wt_\ell(\chi_{J}(\bs\lambda'))$ are of the form
$\lambda'\bs\eta^{-1}$ with $\bs\eta\in\mathcal Q_J$ which is
incompatible with \eqref{eq nu'' mu}, condition (iii) is also
satisfied and \eqref{eq nu'inlambda'p-1} follows from Proposition
\ref{prop appears}.

Since, by \eqref{e:omegaTtp},
$$\pi_{\{1,\ldots,p-2\}}(\bs\nu')=\prod_{i=1}^{p-2}Y_{i,r_i,\lambda(h_i)}=\pi_{\{1,\ldots,p-2\}}(\bs\omega),$$
it follows that $\bs\nu'_{\{1,\ldots,p-2\}}=\bs\omega^{\tilde S}$,
where $\tilde S$ is obtained from $S$ by restriction to $\tilde J=\{1,\ldots,p-2\}$.
Consider also $\tilde c=c- _{p-1}|\lambda|, \tilde f=f- _{p-1}|\lambda|$, and observe that
    \begin{equation*}
    (\bs\nu')_{\tilde J}=\bs\omega_{\tilde J} = \bs\omega^{\tilde S}
    \qquad\text{and}\qquad \bs\omega_{\tilde
        J}(\bs\omega^{-1}\bs\omega^{S_{c,f,p}})_{\tilde J} =
    \bs\omega^{\tilde S_{\tilde c,\tilde
            f,p}}\in\qch(V_q(\bs\omega_{\tilde J})).
    \end{equation*}
Since
$$\bs\mu=\bs\nu'\cdot \iota_{\tilde J}((\bs\omega^{-1}\bs\omega^{S_{c,f,p}})_{\tilde J})=
\bs\nu'\cdot\iota_{\tilde J}((\bs\nu')_{\tilde J}^{-1}\bs\omega^{\tilde S_{\tilde c,\tilde f,p}})\in\wtl(\chi_{\tilde J}(V_q(\bs\nu'))),$$
we get
    \begin{equation}\label{eq nu'''in e}
    \bs\mu\in \wt_\ell(\chi_{\tilde J}(\bs\nu'))\subseteq\wt_\ell(\chi_{J}(\bs\nu'))\subseteq \wt_\ell(V_q(\bs\lambda'))
    \end{equation}
with $J$ as in \eqref{eq nu'inlambda'p-1}.

After \eqref{eq nu'''in e},  \eqref{e:parte} follows if we show that
\begin{equation}\label{e: nu''' factors A e}
\wt_\ell(\chi_{p-1}(\bs\mu))\subseteq \wt_\ell
(V_q(\bs\lambda'))\qquad \text{and} \qquad
\bs\nu\in\wt_\ell(\chi_{p-1}(\bs\mu)).
\end{equation}
To prove the first claim in \eqref{e: nu''' factors A e}, we use Proposition \ref{prop appears} once more, this time with
$\bs\lambda'$ in place of $\bs\omega$ and $J=\{p-1\}$. By \eqref{eq nu'''in e}, $\bs\mu$ satisfies condition (i) of that Proposition while condition (ii) is easily checked using \eqref{e:omegaTtp}  and \eqref{e:omegaScfp}. Note that the first conclusion in \eqref{eq nu'''in e} (the ``$\in$''), together with \eqref{eq nu'' mu}, implies that
\begin{equation*}
\bs{\mu} = \bs{\lambda}'\bs{\eta}^{-1} \quad\text{with}\quad \bs{\eta}\in \mathcal Q^+_{I\setminus \{p-1\}}.
\end{equation*}
Therefore, \eqref{e:lrootsind} implies that condition (iii) is satisfied. The first claim in \eqref{e: nu''' factors A e} now follows from  Proposition \ref{prop appears}.

To prove the second claim in  \eqref{e: nu''' factors A e}, recall the definitions of $\bs{\mu}$ and $\bs{\nu}$ in terms of tableaux and observe that \eqref{e:ScfpfromSbyAs} together with \eqref{e:rifromri0} implies
\begin{equation}\label{eq nu''' nu A e}
\bs\nu=\bs\mu\left(\prod_{l=0}^{m-k'-1}A^{-1}_{p-1,r_p+2(k'-m+l)}\right).
\end{equation}
As before, we now compute the $q$-factorization of $\pi_{p-1}(\bs\mu)$ and apply the usual argument with tensor products to complete the proof. We split the analysis in two cases according to whether $l_f = p-1$ or $l_f<p-1$. In the former case we have
\begin{eqnarray*}
    \pi_{p-1}(\bs\mu)&\stackrel{\eqref{e:omegaTtp}}{=}&Y_{p-1,r_{p-1},\lambda(h_{p-1})} \left(\prod_{l=0}^{m-1}Y_{p-1,r+2(k-l)+n-p-1}\right)\\
    &\stackrel{\eqref{eq p,k'}}{=}&Y_{p-1,r_{p-1},\lambda(h_{p-1})}\left(\prod_{l=0}^{m-1}Y_{p-1,r_p+2(k'-l-1)-1}\right)\\
    &=&Y_{p-1,r_{p-1},\lambda(h_{p-1})}\left(\prod_{l=0}^{m-1}Y_{p-1,r_p+2(k'-m+l)-1}\right)\\
    &=&Y_{p-1,r_{p-1},\lambda(h_{p-1})}\left(\prod_{l=m-k'}^{m-1}Y_{p-1,r_p+2(k'-m+l)-1}\right)\left(\prod_{l=0}^{m-k'-1}Y_{p-1,r_p+2(k'-m+l)-1}\right)\\
    &\stackrel{\eqref{e:rifromri0}}{=}&Y_{p-1,r_{p-1},\lambda(h_{p-1})+k'}\left(\prod_{l=0}^{m-k'-1}Y_{p-1,r_p+2(k'-m+l)-1}\right)\\
    &=&Y_{p-1,r_{p-1},\lambda(h_{p-1})+k'} Y_{p-1,r_p+2(k'-m)-1,m-k'}.
\end{eqnarray*}
One easily checks that the above is the $q$-factorization of  $\pi_{p-1}(\bs\nu')$. Thus, \eqref{teo sl2 irred} implies that
\begin{eqnarray*}
    V_q(\pi_{p-1}(\bs\mu))&\cong&
    V_q(Y_{p-1,r_{p-1},\lambda(h_{p-1})+k'})\otimes
    V_q(Y_{p-1,r_p+2(k'-m)-1,m-k'}).
\end{eqnarray*}
Applying \eqref{e:l-weights KRnroot}  to the second factor of this
tensor product and comparing with \eqref{eq nu''' nu A e} one easily
deduces the second claim of \eqref{e: nu''' factors A e} in this
case. Finally, if $l_f<p-1$, \eqref{e:omegaTtp}  and \eqref{e:omegaScfp} give
\begin{eqnarray*}
    \pi_{p-1}(\bs\mu) &=& Y_{p-1,r_{p-1},\lambda(h_{p-1})} \left(\prod_{i=l_f+1}^{p-2} \prod_{l=1}^{\lambda(h_i)} Y_{p-1,r_i+2(\lambda(h_i)-l)+p-i-1}\right)\\
    &&\times\left(\prod_{l=1}^{d_f}Y_{p-1,r_{l_f}+2(\lambda(h_{l_f})-l)+p-l_f-1}\right)\left(\prod_{l=0}^{m-1}Y_{p-1,r+2(k-l)+n-p-1}\right)\\
    &\stackrel{\eqref{e:rifromri0}}{=}& \left(\prod_{l=0}^{m-k'-1}Y_{p-1,r_p+2(k'-m+l)-1}\right) \left(\prod_{l=0}^{m-1}Y_{p-1,r+2(k-l)+n-p-1}\right)\\
    &\stackrel{\eqref{eq p,k'}}{=}&  \left(\prod_{l=0}^{m-k'-1}Y_{p-1,r_p+2(k'-m+l)-1}\right) \left(\prod_{l=0}^{m-1}Y_{p-1,r_p+2(k'-l-1)-1}\right)\\
    &=&  \left(\prod_{l=0}^{m-k'-1}Y_{p-1,r_p+2(k'-m+l)-1}\right) \left(\prod_{l=0}^{m-1}Y_{p-1,r_p+2(k'-m+l)-1}\right)\\
    &=& Y_{p-1,r_p+2(k'-m)-1,m-k'} Y_{p-1,r_p+2(k'-m)-1,m}.
\end{eqnarray*}
One easily checks that the above is the $q$-factorization of $\pi_{p-1}(\bs\mu)$ and, hence,
\begin{eqnarray*}
V_q(\pi_{p-1}(\bs\mu))&\cong&
V_q(Y_{p-1,r_{p-1},\lambda(h_{p-1})+k'+1})\otimes
V_q(Y_{p-1,r_p+2(k-m)-1,t-k'})
\end{eqnarray*}
by \eqref{teo sl2 irred}. We are once again done by applying \eqref{e:l-weights KRnroot}  to the second factor of this tensor product and comparing with \eqref{eq nu''' nu A e}. The proof of the second claim of \eqref{e: nu''' factors A e} is complete, as well as the proof of the lemma.
\end{proof}

\section{The Corollaries}\label{s:othermain}

In this section, we prove Corollary \ref{c:main} and state and prove the other versions of Theorem \ref{t:main} concerning tensor products of general minimal affinizations with KR modules supported either at last or the first node of the Dynkin diagram. The proofs will rely most strongly on duality type arguments which we review in the first subsection.

\subsection{Duality and Cartan Involution}\label{ss:dual}

Given a finite-dimensional $U_q(\lie g)$-module $V$, let $V^*$
denote the dual module defined using the antipode as usual, and
similarly for $U_q(\tlie g)$. Given $\lambda\in P^+$ and
$\bs\omega\in\cal P^+$, we have:
\begin{gather}\notag
\quad V_q(\lambda)^* \cong V_q(\lambda^*) \qquad\text{and}\qquad\quad V_q(\bs\omega)^*\cong V_q(\bs\omega^*)\\ \label{e:drhw} \text{where}\\ \notag
\qquad\qquad\lambda^*=-w_0\lambda \quad\qquad\text{and}\qquad \bs\omega^*_i(u) = \bs\omega_{w_0\cdot i}(q^{-h^\vee}u).
\end{gather}
Here,  $h^\vee=n+1$ is the (dual) Coxeter number of $\lie g$, $w_0$ is the longest element of $\mathcal W$ and $w_0\cdot i =j$
iff $w_0\omega_i=-\omega_j$. To simplify notation, we set
\begin{equation}\label{e:bar}
\bar i= w_0\cdot i = h^\vee-i.
\end{equation}
For a proof of the second isomorphism in \eqref{e:drhw} see \cite{fremuk:qchar}.

We will also need the automorphisms given by the following proposition \cite[Propositions 1.5 and 1.6]{cha:minr2}.

\begin{prop}\label{p:Cartaninv}\hfill\\\vspace{-10pt}
    \begin{enumerate}[(a)]
        \item Given $a\in \mathbb C^\times$, there exists a unique Hopf algebra automorphism $\tau_a$ of $U_q(\tlie g)$ such that
            $$\tau_a(x_{i,r}^{\pm})=a^rx_{i,r}^{\pm},\qquad \tau_a(h_{i,r})=a^rh_{i,r},\qquad \tau_a(k_{i}^{\pm})=k_{i}^{\pm}.$$

        \item There exists a unique algebra automorphism $\kappa$ of $U_q(\tlie g)$ such that
        $$\kappa(x_{i,r}^{\pm})=-x_{i,-r}^{\mp},\qquad \kappa(h_{i,r})=-h_{i,-r},\qquad \kappa(k_{i}^{\pm1})=k_{i}^{\mp1}.$$
        Moreover $(\kappa\otimes \kappa)\circ \Delta=\Delta^{\text{op}}\circ \kappa$,
        where $\Delta^{\text{op}}$ is the opposite comultiplication of $U_q(\tlie g)$.\footnote{The automorphism $\kappa$ is most often denoted by $\hat\omega$ in the literature and its restriction to $U_q(\lie g)$, typically denoted by $\omega$, is referred to as the Cartan automorphism of $U_q(\lie g)$. We chose to modify the notation to avoid visual confusion with our most often used symbol for a Drinfeld polynomial: $\bs\omega$.}\endd
    \end{enumerate}
\end{prop}

Given $\bs\omega\in\cal P^+$, define $\bs\omega^{\tau_a}\in \cal P^+$ by $$\bs\omega^{\tau_a}_i(u)=\bs\omega_i(au).$$ One easily checks that the pullback
$V_q(\bs\omega)^{\tau_a}$ of $V_q(\bs\omega)$ by $\tau_a$ satisfies
\begin{equation}\label{e:tau_a}
V_q(\bs\omega)^{\tau_a}\cong V_q(\bs\omega^{\tau_a}).
\end{equation}
In particular, we get
\begin{equation}
(V_q(\bs{\omega})^*)^*\cong V_q(\bs{\omega})^{\tau_{q^{-2h^\vee}}}.
\end{equation}

Quite clearly, $\kappa$ is an involution,
\begin{equation}\label{e:kappaLambda}
\kappa(\Psi_i^\pm(u)) = \Psi_i^\mp(u), \qquad\text{and}\qquad \kappa(\Lambda_i^\pm(u)) = (\Lambda_i^\mp(u))^{-1}.
\end{equation}
The pullback $V^\kappa$ of an irreducible $U_q(\tlie g)$-module $V$ by $\kappa$ is, evidently, an irreducible module as well and we can use \eqref{e:kappaLambda} to compute its highest $\ell$-weight. Indeed, if $V\cong V_q(\bs{\omega})$ with $\bs{\omega}\in\mathcal P^+$, then the lowest-weight vector of $V$ is the highest-weight vector of $V^\kappa$. Therefore, if $\bs{\varpi}\in\wtl(V)$ is the lowest $\ell$-weight, we must have, by \eqref{e:kappaLambda} and \eqref{e:PsiLambda}, that
\begin{equation}\label{e:kappahislow}
V^\kappa\cong V_q\left((\bs{\varpi^-})^{-1}\right).
\end{equation}
Since
\begin{equation*}
\bs{\varpi}_i(u) = (\bs{\omega}_{\bar i}(q^{h^\vee}u))^{-1}
\end{equation*}
(see for instance
\cite[Proposition 3.6]{cm:qblock}), it follows that
\begin{equation*}
\bs{\varpi}^-_i(u) = (\bs{\omega}_{\bar i}^-(q^{-h^\vee}u))^{-1}.
\end{equation*}
In other words,
\begin{equation*}
(\bs{\varpi}^-)^{-1} = (\bs{\omega}^-)^*
\end{equation*}
and, hence,
\begin{equation}\label{e:kappa}
V_q(\bs{\omega})^\kappa \cong V_q(\bs\omega^\kappa)
\qquad\text{where}\qquad \bs{\omega}^\kappa = (\bs{\omega}^-)^*.
\end{equation}
One easily checks that
\begin{equation*}
(\bs{\omega}^*)^\kappa = \bs{\omega}^- \qquad\text{and}\qquad (\bs{\omega}^\kappa)^* = ((\bs{\omega}^-)^*)^* 
\end{equation*}
which proves that
\begin{equation}\label{e:prop invert}
(V_q(\bs\omega)^*)^\kappa\cong V_q(\bs{\omega}^-) \qquad\text{and}\qquad (V_q(\bs{\omega})^\kappa)^* \cong (V_q(\bs{\omega}^-)^*)^*\cong V_q(\bs{\omega}^-)^{\tau_{q^{-2h^\vee}}} .
\end{equation}

Note that the maps $\bs{\omega}\mapsto \bs{\omega}^-$ and $\bs{\omega}\mapsto \bs{\omega}^*$ are monoid endomorphisms of $\mathcal P^+$. Hence, they induce group endomomorphisms of $\mathcal P$ which we also denote by $-$ and $*$. Inspired by \eqref{e:kappa}, we denote by $\kappa$ the composition $*\circ -$. The following lemma is obvious.

\begin{lem}\label{l:dualiso}
    The endomomorphisms $-,*$, and $\kappa$ of $\mathcal P$ are bijective.\endd
\end{lem}

The proof of the next proposition is straightforward from \eqref{e:drhw} and \eqref{e:kappa} and the definition of $\bs{\omega}^-$.

\begin{prop}\label{p:incdec}
    Let $\bs{\omega}\in\mathcal P^+, \lambda = \wt(\bs{\omega}), V= V_q(\bs{\omega})$, and $V^-=V_q(\bs{\omega}^-)$. The following are equivalent.
    \begin{enumerate}[(i)]
        \item $V$ is an increasing minimal affinization of $V_q(\lambda)$.
        \item $V^-$ is a decreasing minimal affinization of $V_q(\lambda)$.
        \item $V^*$ is a decreasing minimal affinization of $V_q(-w_0\lambda)$.
        \item $V^\kappa$ is an increasing minimal affinization of $V_q(-w_0\lambda)$. \endd
    \end{enumerate}
\end{prop}

The following proposition is easily established by standard arguments.

\begin{prop}\label{p:length}
    Let $V\in \wcal C_q$ and $a\in\mathbb C^\times$. Then, the lengths of $V^{\tau_a}, V^*$, and of $V^\kappa$ are all equal to that of $V$.\endd
\end{prop}

Finally, recall that, given a Hopf algebra $H$, the antipode is an anti-automorphism of the comultiplication. Using this and the last statement of Proposition \ref{p:Cartaninv}, it follows that, for any two finite-dimensional $U_q(\tlie g)$-modules $V$ and $W$, we have
\begin{equation}\label{e:tpdual}
(V\otimes W)^*\cong W^*\otimes V^* \qquad\text{and}\qquad (V\otimes W)^\kappa\cong W^\kappa\otimes V^\kappa.
\end{equation}

\subsection{The General Version of the Main Theorem}\label{ss:Gen} Fix $\lambda\in P^+$, $\bs\omega$ as in \eqref{eq cond min A}, and set
$$i_0=\max(\supp(\lambda)), \qquad a=a_{i_0}.$$
as before. Thus, there exist $r_i\in\mathbb Z, i\in I$, such that
\begin{equation}
\bs\omega_i = \bs{\omega}_{i,a_i,\lambda(h_i)} = Y_{i,r_i,\lambda(h_i)}.
\end{equation}
By definition of $Y_{i,r,m}$, we have
\begin{equation*}
r_{i_0} = 1-\lambda(h_{i_0}).
\end{equation*}
If $V_q(\bs{\omega})$ is an increasing minimal affinization, then the $r_j$ are related as in \eqref{e:rifromri0}. If $V_q(\bs{\omega})$ is a decreasing minimal affinization, then
\begin{equation}\label{e:rifromri0dec}
r_i =r_{i_0}+2\ _{i+1}|\lambda|_{i_0} +i_0-i \quad \text{for all}
\quad i\in\supp(\lambda).
\end{equation}
Fix also
\begin{equation*}
\bs\varpi = Y_{e,r,k} \qquad\text{for some}\qquad r,k\in\mathbb Z,\ k>0,\ e\in\{1,n\}\subseteq I,
\end{equation*}
and set
\begin{equation*}
\bs\lambda = \bs\omega\bs\varpi, \qquad V=V_q(\bs{\omega}), \quad\text{and}\quad W=V_q(\bs{\varpi}).
\end{equation*}
Let also $V^-= V_q(\bs{\omega}^-)$ and  $W^-=V_q(\bs\varpi^-)$.

\begin{cor}
    The length of $V\otimes W$ is at most $2$ and the following are equivalent
    \begin{enumerate}[(i)]
        \item  $V\otimes W$ is irreducible.
        \item  $V^*\otimes W^*$ is irreducible.
        \item  $V^\kappa\otimes W^\kappa$ is irreducible.
        \item  $V^-\otimes W^-$ is irreducible.
    \end{enumerate}
\end{cor}

\begin{proof}
    The equivalence of the four statements is an immediate consequence of Proposition \ref{p:length} together with \eqref{e:tpdual} and \eqref{e:prop invert} and the fact that the Grothendieck ring of $\wcal C_q$ is commutative. If $V$ is an increasing minimal affinization and $e=n$, the first statement is part of Theorem \ref{t:main}. If $V$ is increasing and $e=1$, then $V^\kappa$ is increasing, $\bar e = n$ and, hence, $V^\kappa\otimes W^\kappa$ has length at most $2$ by Theorem \ref{t:main}. The case that $V$ is decreasing and $e=1$ follows by a similar argument using dualization instead of $\kappa$. Finally, if $V$ is decreasing and $e=n$, we apply dualization and $\kappa$ to obtain a tensor product as that of Theorem \ref{t:main}.
\end{proof}

By Lemma \ref{l:dualiso}, there exists a unique pair $(\tilde{\bs\omega},\tilde{\bs{\varpi}})\in\mathcal P^+\times \mathcal P^+$ such that
\begin{equation*}
(\bs\omega,\bs{\varpi}) =
\begin{cases}
(\tilde{\bs\omega}^*,\tilde{\bs{\varpi}}^*),& \text{if $V$ is decreasing and } e=1;\\
(\tilde{\bs\omega}^\kappa,\tilde{\bs{\varpi}}^\kappa),& \text{if $V$ is increasing and } e=1;\\
(\tilde{\bs\omega}^-,\tilde{\bs{\varpi}}^-),& \text{if $V$ is decreasing and } e=n.
\end{cases}
\end{equation*}
Moreover, by Proposition \ref{p:incdec}, $V_q(\tilde{\bs\omega})$ is an increasing minimal affinization and $V_q(\tilde{\bs{\varpi}})$ is a
Kirillov-Reshetikhin module with $\wt(\tilde{\bs{\varpi}}) = k\omega_n$. In other words, Theorem \ref{t:main} applies to $V_q(\tilde{\bs{\varpi}})\otimes V_q(\tilde{\bs{\varpi}})$. In the case that $V\otimes W$ is reducible and, hence, so is $V_q(\tilde{\bs\omega})\otimes V_q(\tilde{\bs\varpi})$ by the previous corollary, we denote by $\tilde{\bs{\lambda}}$ the Drinfeld polynomial of the extra irreducible factor of $V_q(\tilde{\bs\omega})\otimes V_q(\tilde{\bs\varpi})$ and by $\bs{\lambda}'$ that of the extra irreducible factor of $V\otimes W$. Set
$$i_1=\min(\supp(\lambda))$$
and note that
\begin{equation*}
\bar{i_1} = \max(\supp(\lambda^*)).
\end{equation*}
We are read to prove the completion of Theorem \ref{t:main}

\begin{cor}\label{c:mainc}\hfill\\ \vspace{-12pt}
    \begin{enumerate}[(a)]
        \item If $V$ is decreasing and $e=1$, then $V\otimes W$ is reducible if and only if there exist $p\in\supp(\lambda)$ and $k'>0$ such that  either one of the following options hold:\vspace{5pt}
        \begin{enumerate}[(i)]
            \item $k'\le\min\{\lambda(h_p),k\}$ and $r+2k+p+1=r_p+2k'$;
            \item $k'\le\min\{|\lambda|,k\}$ and  $r_{i_1}+2\lambda(h_{i_1})+i_1+1=r+2k'$ and $p=\min\{i:|\lambda|_i\ge k'\}$.
        \end{enumerate}
        In both cases, $\bs{\lambda}' = \tilde{\bs{\lambda}}^*$.

        \item If $V$ is increasing and $e=1$, then $V\otimes W$ is reducible if and only if there exist $p\in\supp(\lambda)$ and $k'>0$ such that  either one of the following options hold:\vspace{5pt}
        \begin{enumerate}[(i)]
            \item $k'\le\min\{\lambda(h_p),k\}$ and $r_{p}+2\lambda(h_{p})+p+1=r+2k'$;
            \item $k'\le\min\{|\lambda|,k\}$ and  $r+2k+i_1+1=r_{i_1}+2k'$ and  $p=\min\{i:|\lambda|_i\ge k'\}$.
        \end{enumerate}
        In both cases, $\bs{\lambda}' = \tilde{\bs{\lambda}}^\kappa$.

        \item If $V$ is decreasing and $e=n$, then $V\otimes W$ is reducible if and only if there exist $p\in\supp(\lambda)$ and $k'>0$ such that  either one of the following options hold:\vspace{5pt}
        \begin{enumerate}[(i)]
            \item $k'\le\min\{\lambda(h_p),k\}$ and $r_{p}+2\lambda(h_{p})+n-p+2=r+2k'$;
            \item $k'\le\min\{|\lambda|,k\}$ and  $r+2k+n-i_0+2=r_{i_0}+2k'$ and $p=\max\{i\in I:\ _i|\lambda|\ge k'\}$.\vspace{5pt}
        \end{enumerate}
        In both cases, $\bs{\lambda}' = \tilde{\bs{\lambda}}^-$.
    \end{enumerate}
\end{cor}

\begin{proof}
    We know that $V\otimes W$ is reducible if and only if  $V_q(\tilde{\bs\omega})\otimes V_q(\tilde{\bs\varpi})$ is. Therefore, the conditions listed above are obtained from those of Theorem \ref{t:mainr} by applying $\kappa$ and/or by considering dual modules. Thus, assume $V_q(\tilde{\bs\omega})\otimes V_q(\tilde{\bs\varpi})$ is reducible and consider the exact sequence
    \begin{equation*}
    0\to M\to V_q(\tilde{\bs\omega})\otimes V_q(\tilde{\bs\varpi})\to N\to 0.
    \end{equation*}
    We have either $M\cong V_q(\tilde{\bs{\lambda}})$ or $N\cong V_q(\tilde{\bs{\lambda}})$.

    Assume first that we are in case (a). Thus, we also have  the exact sequence
    \begin{equation*}
    0\to N^*\to W\otimes V\to M^*\to 0,
    \end{equation*}
    which immediately proves the statement about $\bs{\lambda}'$. To obtain conditions (i) and (ii) above from those of Theorem \ref{t:mainr}, observe that, since
    $\tilde{\bs\omega}^* =\bs{\omega}$, \eqref{e:drhw} implies that
    \begin{equation*}
    \tilde{\bs{\omega}}_i = Y_{\bar i,r^*_i,\lambda(h_i)} \qquad\text{where}\qquad  r^*_i = r_i+h^\vee.
    \end{equation*}
    Notice also that $\tilde{\bs{\varpi}} = Y_{n,r^*,k}$ with $r^*=r+h^\vee$ and that $|\lambda|=|\lambda^*|$.
    By Theorem \ref{t:mainr}, there exist a pair $(p,k')$ such that $\bar p\in\supp(\lambda^*)$, $k'>0$, and either one of the following options hold:
    \begin{enumerate}[(i)]
        \item $k'\le\min\{\lambda^*(h_{\bar p}),k\}$ and $r^*+2k+n-\bar p+2 = r^*_{p}+2k'$;
        \item $k'\le\min\{|\lambda|,k\}$ and  $r^*_{\bar{i_1}}+2\lambda^*(h_{\bar{i_1}})+n-\bar {i_1}+2 = r^*+2k'$, and $\bar p=\max\{i\in I:\ _i|\lambda^*|\ge k'\}$.
    \end{enumerate}
    Recalling that $\bar i = n+1-i$, $\lambda^*(h_{\bar i})=\lambda(h_i)$,  and using the expressions for $r^*_i$ and $r^*$ above, we get the conditions in the statement of part (a).

    If we are in case (b), we have  the exact sequence
    \begin{equation*}
    0\to M^\kappa\to W\otimes V\to N^\kappa\to 0,
    \end{equation*}
    which proves the statement about $\bs{\lambda}'$ as before. This time, it follows from \eqref{e:drhw} and \eqref{e:kappa} that
    \begin{equation*}
    \tilde{\bs{\omega}}_i = Y_{\bar i,r'_i,\lambda(h_i)} \qquad\text{where}\qquad  r'_i = -r_i-2(\lambda(h_i)-1)+h^\vee
    \end{equation*}
    and
    \begin{equation*}
    \tilde{\bs{\varpi}} = Y_{n,r',k} \qquad\text{where}\qquad  r' = -r-2(k-1)+h^\vee.
    \end{equation*}
    By Theorem \ref{t:mainr}, there exist a pair $(p,k')$ such that $\bar p\in\supp(\lambda^*)$, $k'>0$, and either one of the following options hold:
    \begin{enumerate}[(i)]
        \item $k'\le\min\{\lambda^*(h_{\bar p}),k\}$ and $r'+2k+n-\bar p+2 = r'_{p}+2k'$;
        \item $k'\le\min\{|\lambda|,k\}$ and  $r'_{\bar{i_1}}+2\lambda^*(h_{\bar{i_1}})+n-\bar {i_1}+2 = r'+2k'$, and $\bar p=\max\{i\in I:\ _i|\lambda^*|\ge k'\}$.
    \end{enumerate}
    Using the expressions for $r'_i$ and $r'$ above, we get the conditions in the statement of part (b).

    Finally, in case (c), we have  the exact sequence
    \begin{equation*}
    0\to (N^*)^\kappa\to V\otimes W\to (M^*)^\kappa\to 0,
    \end{equation*}
    which, together with \eqref{e:prop invert}, proves the statement about $\bs{\lambda}'$. This time we have
    \begin{equation*}
    \tilde{\bs{\omega}}_i = Y_{i,r^-_i,\lambda(h_i)} \qquad\text{where}\qquad  r^-_i = -r_i-2(\lambda(h_i)-1)
    \end{equation*}
    and
    \begin{equation*}
    \tilde{\bs{\varpi}} = Y_{n,r^-,k} \qquad\text{where}\qquad  r^- = -r-2(k-1).
    \end{equation*}
    By Theorem \ref{t:mainr}, there exist a pair $p\in\supp(\lambda)$ and $k'>0$ such that and either one of the following options hold:
    \begin{enumerate}[(i)]
        \item $k'\le\min\{\lambda(h_{p}),k\}$ and $r^-+2k+n-p+2 = r^-_{p}+2k'$;
        \item $k'\le\min\{|\lambda|,k\}$ and  $r^-_{i_0}+2\lambda(h_{i_0})+n-i_0+2 = r^-+2k'$, and $p=\max\{i\in I:\ _i|\lambda|\ge k'\}$.
    \end{enumerate}
    Using the expressions for $r^-_i$ and $r^-$ above, we get the conditions in the statement of part (c).
\end{proof}

\subsection{Socle and Head}\label{ss:sh}

It remains to prove Corollary \ref{c:main} and its analogues for the other cases given in Corollary \ref{c:mainc}. We begin by observing that, if $V,W\in\wcal C_q$ are such that $V\otimes W$ is irreducible, then
\begin{equation}\label{e:stporder}
V\otimes W\cong W\otimes V.
\end{equation}
As we have seen in the previous subsections, \eqref{e:stporder} can be false if  $V\otimes W$ is reducible.

It will be convenient to introduce the following notation. Given  $V\in\wcal C_q$ we will say that $V$ has thin top if there exists $\lambda\in P^+$ satisfying:
\begin{enumerate}[(1)]
	\item $V_\lambda\ne 0$ and $\dim(V_\lambda)=1$;
	\item $V_\mu\ne 0$ only if $\mu\le \lambda$.
\end{enumerate}
Evidently, such $\lambda$ is unique if it exists. Observe that every tensor product of modules with thin top has thin top as well. Given a module with thin top, let ${\rm Top}(V)$ be the submodule generated by its top weight space.
Our main extra tool for proving Corollary \ref{c:main} is the following consequence of the main result of \cite{cha:braid} (see also \cite[Corollary 4.4]{mou}).

\begin{prop}\label{p:cyclic}
	Let $l\in\mathbb Z_{\ge 1}, i_j\in I, m_j\in\mathbb Z_{\ge 1}, a_j\in\mathbb C^\times$ for $j=1,\dots, l$.
	\begin{enumerate}[(a)]
		\item If $\frac{a_{j}}{a_{k}}\notin q^{\mathbb Z_{> 0}}$ for $j>k$, then $V_q(\gb\omega_{i_1,a_1,m_1})\otimes \cdots\otimes V_q(\gb\omega_{i_l,a_l,m_l})$ is a highest-$\ell$-weight module.
		\item If $\frac{a_{j}}{a_{k}}\notin q^{\mathbb Z_{< 0}}$ for $j>k$, then  ${\rm Top}\left(V_q(\gb\omega_{i_1,a_1,m_1})\otimes \cdots\otimes V_q(\gb\omega_{i_l,a_l,m_l})\right)$ is irreducible. \endd
	\end{enumerate}
\end{prop}

Henceforth, fix the notation of Theorem \ref{t:main}. In particular, 
\begin{equation}
s_{i_0}> \cdots > s_2> s_1,
\end{equation}
where $s_i$ defined in \eqref{eq r_l r_i_0} and we have:

\begin{cor}\label{e:deccilic}
	$V_q(\bs\omega_{i_0,a_{i_0},\lambda(h_{i_0})})\otimes\cdots\otimes V_q(\bs\omega_{2,a_2,\lambda(h_2)})\otimes V_q(\bs\omega_{1,a_1,\lambda(h_1)})$ is highest-$\ell$-weight and
	\begin{equation*}
	V_q(\bs{\omega}) \cong {\rm Top}\left(V_q(\bs\omega_{1,a_1,\lambda(h_1)})\otimes V_q(\bs\omega_{2,a_2,\lambda(h_2)})\otimes\cdots\otimes V_q(\bs\omega_{i_0,a_{i_0},\lambda(h_{i_0})})\right).
	\end{equation*}
	\endd
\end{cor}

Suppose $V=V_q(\bs\omega)\otimes V_q(\bs\varpi)$ is reducible, and let  $(p,k')$ be the pair satisfying either condition (i) or (ii) of the theorem.

\begin{cor}\label{c:irnotp}
	$V_q(\bs\omega_{i,a_i,\lambda(h_i)})\otimes V_q(\bs\varpi)$ is irreducible for every $i\in \supp(\lambda)\setminus\{p\}$. In particular, 	$V_q(\bs\omega_{i,a_i,\lambda(h_i)})\otimes V_q(\bs\varpi)\cong V_q(\bs\varpi) \otimes V_q(\bs\omega_{i,a_i,\lambda(h_i)})$.
\end{cor}

\begin{proof}
	If it were reducible, using Theorem \ref{t:main} with $\bs\omega_{i,a_i,\lambda(h_i)}$ in place of $\bs\omega$, it would follow that
	there would exist $0<k''\le\min\{\lambda(h_i),k\}$ such that the pair $(i,k'')$ satisfies either condition (i) or (ii) of Theorem \ref{t:main}. But then $(i,k'')$ satisfies the same condition for $\bs{\omega}$ as well contradicting the uniqueness of the pair $(p,k')$. The second statement is immediate from the first and Proposition \ref{p:qchtp}.
\end{proof}

We are ready to prove Corollary \ref{c:main}. To shorten notation, write
\begin{equation*}
V_i = V_q(\bs\omega_{i,a_i,\lambda(h_i)}), \quad i\le i_0.
\end{equation*}
If neither conditions (i) nor (ii) are satisfied, then $V$ and $V'$ are irreducible and, hence, highest-$\ell$-weight.
Otherwise, suppose first that the pair $(p,k')$ satisfies condition (i) and let us show that $V$ is highest-$\ell$-weight.
Corollary \ref{e:deccilic} implies that we have surjective map
\begin{equation*}
V_{i_0}\otimes\cdots \otimes V_2\otimes V_1\otimes V_q(\bs\varpi) \twoheadrightarrow V.
\end{equation*}
Condition (i) implies that $s_p>s$ and, hence, there exists $i_1\le p, i_1\in\supp(\lambda)$, such that $s_{i_1}> s$ and $s\ge s_i$ for $i\le i_1$. 
Corollary \ref{c:irnotp} then implies that we have a surjective map
\begin{equation}\label{e:hwproof}
V_{i_0}\otimes\cdots\otimes V_{i_1} \otimes V_q(\bs\varpi)\otimes V_{i_1-1}\otimes \cdots \otimes V_2\otimes V_1 \twoheadrightarrow V.
\end{equation}
Since, by choice of $i_1$, we have 
\begin{equation}
s_{i_0}> \cdots > s_{i_1}> s\ge s_{i_1-1}>\cdots  s_2> s_1,
\end{equation}
Proposition \ref{p:cyclic} then implies that the left-hand side of \eqref{e:hwproof} is highest-$\ell$-weight. Hence, so is $V$. This proves the indecomposability of $V$ as well as the first exact sequence of Corollary \ref{c:main}.

By inverting the order of all tensor products in the above argument and  working  with injective maps instead of surjective ones, one can similarly show that
\begin{equation*}
V_q(\bs\lambda)\cong {\rm Top}(V') \subsetneqq V'
\end{equation*}
which proves the second exact sequence. 

In the case that condition (ii) holds, then $s>s_p$ and the argument is similar. We omit the details. In particular, it follows that $V'$ is indecomposable. The following is now immediate from the proof of Corollary \ref{c:mainc}.

\begin{cor}\label{c:hwgen} Let $V$ and $W$ be as in Corollary \ref{c:mainc}.
	\begin{enumerate}[(a)]
		\item If $V$ is decreasing and $e=1$, then $V\otimes W$ is not highest-$\ell$-weigh if and only if (ii) holds while $W\otimes V$ is not highest-$\ell$-weigh if and only if (i) holds.
		\item If $V$ is increasing and $e=1$, then $V\otimes W$ is not highest-$\ell$-weigh if and only if (i) holds while $W\otimes V$ is not highest-$\ell$-weigh if and only if (ii) holds.
		\item If $V$ is decreasing and $e=n$, then $V\otimes W$ is not highest-$\ell$-weigh if and only if (ii) holds while $W\otimes V$ is not highest-$\ell$-weigh if and only if (i) holds. \endd
	\end{enumerate}
\end{cor}

To complete the proof of Corollary \ref{c:main}, it remains to show that $V'$ is indecomposable if (i) holds and, similarly, that $V$ is indecomposable if (ii) holds. We shall write down the proof of the latter. Let $\hat{\bs\omega}$ and $\hat{\bs\varpi}$ be such that
\begin{equation*}
\hat{\bs\omega}^* = \bs\omega \qquad\text{and}\qquad \hat{\bs\varpi}^* = \bs\varpi.
\end{equation*}
By part (a) of Corollary \ref{c:hwgen}, $\hat V = V_q(\hat{\bs\omega})\otimes V_q(\hat{\bs\omega})$ is highest-$\ell$-weight and, hence, indecomposable. But then, $V\cong \hat V^*$ must also be indecomposable.

\begin{rem}
	 Since $W$ is a real simple object of $\widetilde{\mathcal C}_q$ in the sense of \cite{hele:cluster}, it follows from \cite[Theorem 3.12]{kkko} that $V$ and $V'$ have simple socle and head from where their indecomposability is easily deduced. The proof of \cite[Theorem 3.12]{kkko} is based on an analysis of the action of the $R$-matrix, which is not require in our approach.\endd
\end{rem}

\bibliographystyle{amsplain}

\end{document}